\documentclass[11pt,reqno]{amsart}
 
  \usepackage[utf8]{inputenc}

 \usepackage[margin=1in]{geometry}
 \usepackage{amsmath,amssymb}
 \usepackage{algpseudocode}
 \usepackage{algorithm}
 \usepackage{algorithmicx}
 \usepackage{caption}
 \usepackage{subcaption}
 \usepackage{amsthm}
 \usepackage{amsfonts}
 \usepackage{graphicx}
 \usepackage[colorlinks=true, pdfstartview=FitV, linkcolor=blue,citecolor=blue, urlcolor=blue]{hyperref}
 \usepackage{makecell}
 \usepackage{longtable}
 \usepackage{times}
 \usepackage{xfrac}
 \usepackage{mathtools}

\renewcommand{\AA}{ \mathbf{A} }

\usepackage{stmaryrd}
\makeatletter
\newcommand{\fixed@sra}{$\vrule height 2\fontdimen22\textfont2 width 0pt\shortrightarrow$}
\newcommand{\shortarrow}[1]{%
  \mathrel{\text{\rotatebox[origin=c]{\numexpr#1*45}{\fixed@sra}}}
}
\makeatother

\newcommand{\shortnearrow}{\scalebox{0.6}{$\nearrow$}}
\newcommand{\converges}{\shortnearrow}

\title{Vorticity blowup in 2D compressible Euler equations}

\author[J.~Chen]{Jiajie Chen}
\address{Courant Institute of Mathematical Sciences, New York University, New York, NY 10012.}
\email{\href{jiajie.chen@cims.nyu.edu}{jiajie.chen@cims.nyu.edu}}

\author[G.~Cialdea]{Giorgio Cialdea}
\address{Courant Institute of Mathematical Sciences, New York University, New York, NY 10012.}
\email{\href{giorgio.cialdea@cims.nyu.edu}{giorgio.cialdea@cims.nyu.edu}}

\author[S.~Shkoller]{Steve Shkoller}
\address{Department of Mathematics, University of California Davis, Davis, CA 95616.}
\email{\href{shkoller@math.ucdavis.edu}{shkoller@math.ucdavis.edu}}

\author[V.~Vicol]{Vlad Vicol}
\address{Courant Institute of Mathematical Sciences, New York University, New York, NY 10012.}
\email{\href{vicol@cims.nyu.edu}{vicol@cims.nyu.edu}}

\newtheorem{theorem}{Theorem}[section]
\newtheorem{lemma}[theorem]{Lemma}
\newtheorem{proposition}[theorem]{Proposition}

\newtheorem{remark}[theorem]{Remark}

\numberwithin{equation}{section}

\newcommand{\s}{\sigma}

\let\pa=\partial
\let\al=\alpha
\let\f=\frac
\def\na{\nabla}

\DeclarePairedDelimiterX{\norm}[1]{\lVert}{\rVert}{#1}

\newcommand{\ts}{\tilde{\Sigma}}

\renewcommand{\S}{\Sigma}

\renewcommand{\div}{\operatorname{div}}

\newcommand{\sinefunction}{\text{sin }}
\renewcommand{\sin}{s_{\mathsf{in}}}

\usepackage{amsmath}
\usepackage{amssymb}
\usepackage{amsthm}
\usepackage{amsmath}
\usepackage{setspace}
\usepackage{xcolor}
\usepackage{fancyhdr}
\usepackage{amssymb}
\usepackage{amsthm}
\usepackage{listings}
    \usepackage{cite}
    \usepackage{tabularx}
    \usepackage{graphicx}
    \usepackage{epstopdf}    
    \usepackage{epsfig}
    \usepackage{float}
    \usepackage{ listings}
    \usepackage{appendix}
 \theoremstyle{plain}

 \let\pa=\partial
 \let\al=\alpha
 \let\b=\beta
 \let\d=\delta
 \let\g=\gamma
 \let\e=\varepsilon

 \let \kp = \kappa
 \let\lam=\lambda
 
 \let\s=\sigma
 \let\f=\frac
 \let \les = \lesssim
  \let \gtr = \gtrsim
 \let\om=\omega
 
 \let \th = \theta
  
 \let \pr = \prime
 \let \vp = \varphi
 \let\G= \Gamma
\let\B = \Big
 \let\D=\Delta
 
 \let\S=\Sigma
 
 \let\Om=\Omega
 \let\td = \tilde
 
 \let\wh=\widehat

 \let\teq \triangleq
 
 \let\pa=\partial

 \def\cA{{\mathcal A}}

 \def\cD{{\mathcal D}}
 \def\cE{{\mathcal E}}

 \def\cK{{\mathcal K}}
 \def\cL{{\mathcal L}}
 
 \def\cN{{\mathcal N}}

 \def\cR{{\mathcal R}}
 \def\cS{{\mathcal S}}
 \def\cT{{\mathcal T}}

 \def\cW{{\mathcal W}}
 \def\cX{{\mathcal X}}
 
 \def\cZ{{\mathcal Z}}

 \def\cN{{\mathcal N}}

 \def\na{\nabla}
 \def\la{\langle}
 \def\ra{\rangle}

\def\one{\mathbf{1}}

  \def\Re{\mathrm{Re}}
  \def\Rg{\mathrm{Rg}}

 \newcommand{\beq}{\begin{equation}}
 \newcommand{\eeq}{\end{equation}}
  \newcommand{\bal}{\begin{aligned} }
  \newcommand{\eal}{\end{aligned}}
 \newcommand{\ben}{\begin{eqnarray}}
 \newcommand{\een}{\end{eqnarray}}
 \newcommand{\beno}{\begin{eqnarray*}}
 \newcommand{\eeno}{\end{eqnarray*}}

 \newcommand{\ee}{\mathbf{e}}

 \newcommand{\uu}{\mathbf{u}}

 \newcommand{\vA}{\mathbf{A}}

 \newcommand{\FF}{\mathbf{F}}
 \newcommand{\GG}{\mathbf{G}}

 \newcommand{\R}{\mathbb{R}}
 
 \newcommand{\UU}{\mathbf{U}} 
\newcommand{\VV}{\mathbf{V}}

\newcommand{\sgn}{\mathrm{sgn}}
 \newcommand{\supp}{\mathrm{supp}}

\begin{document}

\begin{abstract}
We prove finite-time vorticity blowup for smooth solutions of the 2D compressible Euler equations with smooth,  localized, and non-vacuous initial
data.  The vorticity blowup occurs at the time of the first singularity, and is accompanied by an axisymmetric implosion in which the swirl velocity enjoys full stability, as opposed to finite co-dimension stability.
\end{abstract}

\maketitle

\section{Introduction}
\label{sec:intro} 

We consider the isentropic compressible Euler equations in multiple space dimensions
\begin{subequations}
\label{eq:classical:euler}
\begin{align}
 \pa_t (\rho \uu) + \operatorname{div}(\rho \uu \otimes \uu) + \nabla p & =0  \,, \\ 
\pa_t \rho + \operatorname{div} (\rho \uu) & = 0 \,, 
\end{align}
where
$\uu \colon \R^d \times [0,\infty) \to \R^d$ is the velocity vector field and  $\rho \colon \R^d \times [0,\infty) \to \R_+$  is the strictly positive density function.  The pressure  $p\colon \R^d \times [0,\infty) \to \R_+$ is determined from the density via the ideal gas law
\begin{equation}
p = \tfrac{1}{\gamma} \rho^\gamma, 
\end{equation}
where  $\gamma>1$ is the adiabatic exponent. 
We supplement the Euler equations with initial data 
\begin{equation}
(\uu,\rho) |_{t=0} = (\uu_0,\rho_0),
\end{equation}
\end{subequations}
which is {\em as nice as it gets}. By this,  we mean that $(\uu_0,\rho_0) \in C^\infty(\R^d)$,  $\rho_0 \geq {\rm constant} >0 $ on $\R^d$, and  $(|\uu_0(x)|,\rho_0(x))\to (0,{\rm constant})$ as $|x|\to \infty$, at a sufficiently fast rate; here, $d\geq 2$ is the space dimension.

With these assumptions on $(\uu_0,\rho_0)$,  the classical well-posedness theory~\cite{Fr1954,La1955,Ka1975} guarantees the existence of a maximal time $T>0$, and of a unique solution $(\uu,\rho) \in C^\infty(\R^d \times (0,T))$ of the Cauchy problem for the Euler equations~\eqref{eq:classical:euler}. Generically, we expect that $T<\infty$, a singularity forms in finite time~\cite{La1964,Jo1974,Si1985}.

While it is well-known that at the time $T$ of the first singularity the solution must develop infinite gradients, i.e.~$\limsup_{t\converges T}  \|(\nabla \uu,\nabla \rho)(\cdot,t)\|_{L^\infty(\R^d)}  = \infty$,  a complete classification of all possible singularity mechanisms attainable by smooth  Euler solutions $(\uu,\rho)(\cdot,t)$ as $t \converges T < \infty$ remains, to date, elusive. 

For instance, while it is known that a pre-shock singularity~\cite{Ch2007,ChMi2014,BuShVi2022,BuDrShVi2022,BuShVi2023a,BuShVi2023b,LuSp2018,LuSp2024} or an implosion singularity~\cite{MeRaRoSz2022a,MeRaRoSz2022b,CLGSShSt2023,BuCLGS2022} may form as $t\converges T$ from smooth initial conditions, it is an {\em outstanding open problem to determine whether} $\limsup_{t\converges T} \| \omega(\cdot,t)\|_{L^\infty(\R^d)} = \infty$, where $\omega = \nabla \times \uu$ denotes the fluid vorticity. 

In the case that the Euler solution evolves to a {\em stable} gradient blowup at a time $T<\infty$, from {\em generic}, compressive, and smooth initial data, the vorticity 
$\omega(\cdot,t)$ remains uniformly bounded in $C^{1/3}(\R^d)$ for $t\in [0,T]$, as was shown in~\cite{BuShVi2022,BuDrShVi2022,BuShVi2023a,BuShVi2023b,ShVi2023}. The fact that the vorticity does not blow up at the time of the first 
gradient singularity\footnote{The spacetime point at which the first gradient singularity forms is the point on the
spacetime co-dimension-$2$ manifold of pre-shocks, with a minimal time coordinate (see~\cite{ShVi2023}).} should not come as a surprise for $d=2$, since the specific vorticity $\omega/\rho$ is transported by the fluid velocity:
\begin{equation}
\label{eq:vorticity:transport}
\partial_t \Bigl( \frac{\omega}{\rho} \Bigr)
+ (\uu \cdot \nabla) \Bigl( \frac{\omega}{\rho} \Bigr) = 0 \,.
\end{equation}
Since the density in~\cite{BuShVi2022,BuDrShVi2022,BuShVi2023a,BuShVi2023b,ShVi2023}  remains bounded at the time of the first gradient blowup, 
and the initial datum is non-vacuous, it follows from the maximum principle for~\eqref{eq:vorticity:transport} that 
$\|\omega(\cdot,T)\|_{L^\infty(\R^2)} < \infty$. For $d=3$, the proof of this upper bound for the maximum vorticity is more subtle (see~\cite{BuShVi2023b}).

The specific vorticity transport equation~\eqref{eq:vorticity:transport}  shows that for $d=2$  the vorticity can blow up at time $T$ 
{\em only if} an implosion singularity simultaneously develops at time $T$,~i.e.~$\lim_{t\converges T} \|\rho(\cdot,t)\|_{L^\infty(\R^2)}$.  This, however, need not be an {\em if and only if} statement. Indeed, the smooth implosions established in~\cite{MeRaRoSz2022a,MeRaRoSz2022b,BuCLGS2022} are confined to radially symmetric solutions, and hence $\omega(\cdot,t) \equiv 0$ for all $t\in [0,T]$. The only rigorous proof of a non-radially symmetric  implosion from smooth initial data was recently established in~\cite{CLGSShSt2023} for $d=3$.  Such non-radial imploding solutions possess non-trivial vorticity, but establishing vorticity blowup is incongruous with the finite co-dimension stability framework employed in~\cite{CLGSShSt2023}.  In particular, since vorticity waves propagate along fluid-velocity characteristics\footnote{The specific vorticity $\omega/\rho$ is transported when $d=2$ (see~\eqref{eq:vorticity:transport}) and is Lie-advected when $d=3$ (see~\eqref{spec-3d}).}, there exists a distinguished initial fluid particle $x_*$ such that $\omega_0(x_*)$ is ``carried to'' the  implosion along the fluid-velocity characteristics, thereby determining the vorticity $\omega(\cdot,T)$ at the point of the implosion. As the stability analysis in~\cite{CLGSShSt2023}  requires $\omega_0$ to be a finite co-dimension perturbation of the zero function, and does not specify the precise subspace associated to this finite co-dimension, it is not possible to guarantee that $\omega_0(x_*)\neq 0$ (which must be ensured in order to establish vorticity blowup). We shall explain this issue in greater detail  in  Section~\ref{sec:intro_comp} below.
 
The goal of this paper is to give the first proof of finite-time  blowup for the vorticity in the compressible Euler equations, with smooth initial data --- see Theorem~\ref{thm:blowup} below.

\subsection{The two-dimensional axisymmetric ansatz}
\label{sec:axis-symmetry}
In order to state our main result, it is convenient to first  introduce a symmetric form of the isentropic Euler system~\eqref{eq:classical:euler}, in which the fundamental variables are the fluid velocity $\uu$ and the rescaled sound speed $\sigma$, which is defined by 
\begin{equation*}
\s = \tfrac{1}{\al} \sqrt{\tfrac{dp}{d\rho}} =  \tfrac{1}{\al} \rho^{\al},\quad \al = \tfrac{\g-1}{2} .
\end{equation*}
In terms of the unknowns $\uu$ and $\sigma$ the Euler equations~\eqref{eq:classical:euler} are equivalent to the following system\footnote{In the full Euler system the specific entropy is transported along the flow of the vector field $(\partial_t + \uu \cdot \nabla)$. As such, for isentropic initial data (i.e., with constant specific entropy at time $t=0$) the solution will remain isentropic up to the time $T$ of the very first singularity, which justifies the usage of the~\eqref{eq:classical:euler} and hence of~\eqref{eq:euler} on $[0,T]$.}:
\begin{subequations}
\label{eq:euler} 
\begin{align}
	\pa_t \uu + \uu \cdot \na \uu + \al \,  \s \,   \na \s & = 0 \,,
	\label{eq:euler:a} \\
	\pa_t \s +  \uu \cdot \na \s + \al \, \s  \,  \mathrm{ div} \uu  &= 0 \,.
	\label{eq:euler:b} 
\end{align}
\end{subequations}

For the remainder of the paper, {\em we restrict our analysis to the two-dimensional case}:  $d=2$. As such, it is convenient to rewrite \eqref{eq:euler} in
polar coordinates $(R,\theta)$,  defined as
\begin{equation*}
R = |x|, \quad \theta = \arctan \bigl( \tfrac{x_2}{x_1} \bigr), 
\quad \ee_R = ( \cos \th, \sinefunction \th),
\quad \ee_{\th} = (-\sinefunction \th , \cos \th) .
\end{equation*}
Throughout the paper,  we shall use  boldface font $\uu, \UU, 
\ldots $ to denote vectors,
 and standard unbolded font $u, U, \ldots$ to denote scalars.
 
We analyze vorticity blowup within the class of {\em two-dimensional axisymmetric flows}, a class of flows with built-in rotation. That is, we consider a velocity field $\uu$ with the following axisymmetry\footnote{If $u^{\th} \equiv 0$ in~\eqref{eq:ansatz}, we recover radially symmetric functions $(\uu,\sigma)$.}:
\begin{equation}
\label{eq:ansatz}
 \uu = \uu(R, \theta,t) =  \uu^R +  \uu^{\th}, \quad  \uu^R = u^R(R,t) \ee_R ,
 \quad  \uu^{\th} = u^\theta(R,t) \ee_\theta ,
 \quad \s = \s(R,t). 
\end{equation}
The vorticity and velocity-divergence associated to the axisymmetric ansatz~\eqref{eq:ansatz} are given by
\begin{align*}
\omega &= \omega(R,t) = \bigl(\tfrac{1}{R} + \partial_R \bigr) u^{\theta}(R,t),\\
\div \uu &= \div \uu^R (R,t) = \bigl(\tfrac{1}{R} + \partial_R \bigr) u^{R}(R,t).
\end{align*}
In general, for a smooth  function $f=f(|x|)\colon\mathbb{R}^2 \to \mathbb{R}$ we have $\div( f  \ee_R) =  R^{-1} \pa_R( R  f(R) )$; as such, we shall often abuse notation  and simply write $\div(f(R)) =  R^{-1} \pa_R( R f(R) )$
in place of $\div( f  \ee_R) $.

In order to show that the two-dimensional axisymmetric ansatz~\eqref{eq:ansatz} is preserved by the Euler evolution, we note that the system~\eqref{eq:euler} is equivalent (for smooth solutions) to the system~\eqref{eq:euler_r0} below, which provides evolution equations for $(u^R,u^\theta,\s) = (u^R,u^\theta,\s) (R,t)$. For this purpose, it is convenient to recall the following identities, which hold for smooth functions $f = f(R)$:
\begin{equation}
\label{eq:vec_iden}
 \uu^R \cdot \na f  = u^R \pa_R f , \quad u^{\th} \ee_{\th} \cdot \na (f \ee_R)
 = \tfrac{1}{R}  u^{\th} f  \ee_{\th}, \quad  
 u^{\th} \ee_{\th}  \cdot \na (f \ee_{\th}) = - \tfrac{1}{R}  u^{\th} f  \ee_R.
\end{equation}
Using the ansatz~\eqref{eq:ansatz} and the identities in~\eqref{eq:vec_iden},  we may decompose the system~\eqref{eq:euler} into radial and angular parts, leading to the equivalent formulation
\begin{subequations}
\label{euler-r}
\begin{align}
 \pa_t \uu^R + \uu^R \cdot \na  \uu^R + \uu^{\th} \cdot \na \uu^{\th} + \al \, \s \, \na \s & = 0 \,, \\
 \pa_t \uu^{\th} + \uu^R \cdot \na  \uu^{\th} + \uu^{\th} \cdot \na \uu^R & = 0  \,, \\
 \pa_t \s + \uu^R \cdot \na \s + \al \, \s  \div \uu^R &= 0 \,.
\end{align}
\end{subequations}
In terms of the scalar quantities $(u^R,u^\theta,\sigma)=(\uu^R\cdot\ee_R,\uu^\theta\cdot\ee_\theta,\sigma)$, the system~\eqref{euler-r} may be written as 
\begin{subequations}
\label{eq:euler_r0}
\begin{align}
\pa_t u^R + u^R \pa_R u^R + \alpha \sigma \pa_R \sigma  & = \tfrac{1}{R} (u_\theta)^2   \,, \\
\pa_t u^ \theta + u^R \pa_R u^\theta + \tfrac{1}{R}  u^\theta u^R &= 0 \,, \\
\pa_t \sigma + u^R \pa_R \sigma + \alpha \, \sigma   \div( u^R) & = 0\,.
\end{align}
\end{subequations}
Both reformulations of the isentropic Euler system for two-dimensional axisymmetric flows, i.e.~systems~\eqref{euler-r} and~\eqref{eq:euler_r0}, when transformed to self-similar variables and coordinates (see~\eqref{ss-var} below), play important roles in the subsequent
analysis.  The self-similar version of \eqref{eq:euler_r0} is used to study blowup profiles, while the self-similar version of  \eqref{euler-r} is used for
energy estimates and the majority of bounds on the solution.

\subsection{The main result}
As mentioned in the discussion following~\eqref{eq:vorticity:transport}, in two space dimensions  the vorticity can blow up at time $T$ {\em only if} an implosion singularity simultaneously develops at time $T$. Since we seek smooth Euler solutions on $[0,T)$, we aim to 
construct a stable perturbation (for the angular velocity) of 
a {\em smoothly imploding self-similar solution}, as was recently constructed in~\cite{MeRaRoSz2022a}, and further refined in~\cite{BuCLGS2022}. 

The self-similarity employed in~\cite{MeRaRoSz2022a}, \cite{BuCLGS2022}, and herein, is of the second kind. 
A similarity exponent $r>1$ is determined as the solution of a nonlinear eigenvalue problem, and is then used to define the
self-similar space $\xi = \frac{R}{(T-t)^{1/r}}$ and time  $s  = - \frac 1r \log(T-t)$ coordinates, together with self-similar state variables   $(U,A,\Sigma)(\xi,s)= r (T-t)^{(r-1)/r} (u^R,u^\theta,\sigma)(R,t)$; see~\eqref{ss-var} below. The self-similar variables $(U,A,\Sigma)$ solve the self-similar version of the Euler system; see~\eqref{eq:euler_ss} below.   
Of particular interest for our analysis are special $C^\infty$-smooth {\em stationary} solutions $(\bar U,0,\bar \Sigma)$ to this self-similar system; we refer to the triplet
 $(\bar U,0,\bar \Sigma)$ as the self-similar  blowup  {\em profile}.

We choose to work directly with the smooth imploding self-similar profiles $(\bar U,0,\bar \Sigma)$ from~\cite[Theorem 1.2, Corollary 1.3]{MeRaRoSz2022a}, as these profiles, together with their repulsive properties (see~\eqref{eq:rep1} below), are readily available for $d=2$.\footnote{While it is expected that the arguments in~\cite{BuCLGS2022} also apply to the case that $d=2$ (used in this paper),~\cite{BuCLGS2022} only considers $d=3$.} The slight drawback of using the results from~\cite{MeRaRoSz2022a} is that we need to exclude  an exceptional countable sequence of adiabatic exponents $\Gamma=\{ \g_n\}_{n \geq 1}$, which is possibly empty, and whose accumulation points can only be $\{1, 2, \infty \} $. For each $\gamma \not \in \Gamma$, it is shown in~\cite{MeRaRoSz2022a} (and recalled in Theorem~\ref{thm:merle:implosion} below) that there exists a discrete sequence of admissible blowup speeds $\{ r_k \}_{k\geq 1} \subset (1,r_{\mathsf{eye}}(\alpha))$, accumulating at the value $r_{\mathsf{eye}}(\alpha)$ defined in~\eqref{r-range}, such that the Euler equations~\eqref{eq:euler_r0} exhibit smooth, radially symmetric (meaning
that $u^\theta,A \equiv 0$), globally self-similar, imploding solutions with similarity exponent $r_k$ and profile $(\bar U,0,\bar \Sigma)$ (see~Section~\ref{sec:self-similar:variables} below). With this notation fixed, we state our main result.  

\begin{theorem}[\bf Main result]\label{thm:blowup}
Fix a non-exceptional adiabatic exponent $\gamma \not \in \Gamma$, let $\alpha = (\gamma-1)/2$, and let $r>1$ be an admissible blowup speed satisfying~\eqref{r-range}, as in~\cite{MeRaRoSz2022a}. Then, there exists initial data $(u_0^R, u_0^{\th}) \in C^\infty_c(\mathbb{R}^2)$ and $\sigma_0 \in C^\infty(\mathbb{R}^2)$ with $\sigma_0 \geq {\rm constant} >0$,\footnote{The initial sound speed $\sigma_0$ may be chosen to equal a constant outside a compact set in $\mathbb{R}^2$; see Section~\ref{sec:non}.} such that the solution $(u^R,u^\theta,\sigma)$ of~\eqref{eq:euler_r0} exhibits a finite time singularity at a time $T < \infty$. The radial velocity $u^R$, the sound speed $\sigma$, and the vorticity $\omega$ all blow up at time $T$. The blowup is asymptotically self-similar in the sense that 
\begin{subequations}
\label{eq:blow:asym}
\begin{align}
\lim_{t \converges T} r \big(T-t\big)^{\!\!{\frac{r-1}{r}} } u^R( (T-t)^{\frac{1}{r}}  \xi, t) & = \bar U(\xi),
\label{eq:blow:asym:a} \\
\lim_{t \converges T} r \big(T-t\big)^{\!\!{\frac{r-1}{r}} } u^{\th}( (T-t)^{\frac{1}{r}} \xi, t) &= 0 ,  
\label{eq:blow:asym:b}\\
\lim_{t \converges T} r \big(T-t\big)^{\!\!{\frac{r-1}{r}} } \s( (T-t)^{\frac{1}{r}} \xi, t) &= \bar \S(\xi) ,
\label{eq:blow:asym:c}
\end{align}
for $\xi \in (0,\infty)$, where $\bar U$ and $\bar \Sigma$ are the similarity profiles from~\cite{MeRaRoSz2022a}. The vorticity blows up at $R=0$ as
\begin{equation}
\lim_{t \converges T}  \big(T-t\big)^{\!\!{\frac{r-1}{\alpha r}} } \omega(0, t) = \omega_0(0) \left( \tfrac{\bar{\Sigma}(0)}{r \sigma_0(0)}\right)^{\!\!\frac{1}{\alpha} } \neq 0,
\label{eq:blow:asym:d}
\end{equation}
\end{subequations}
and there are no singularities away from the point of implosion: $(R,t) = (0,T)$.
\end{theorem}

An outline of the proof of Theorem~\ref{thm:blowup} is given in Section~\ref{sec:main:ideas}. Before turning to the proof, we make a few comments concerning the statement of Theorem~\ref{thm:blowup}.

\begin{remark}[\bf The set of initial data for vorticity blowup]
\label{rem:initial:data:set}
As explained in Remark~\ref{rem:data} below, there exists an open set $X_2$ (a ball in a weighted  Sobolev space) and a finite co-dimension set $X_1$
such that the initial data in Theorem~\ref{thm:blowup} may be taken such that $(\uu^R_0, \s_0)  \in X_1$ and $\uu_0^{\th} \in X_2$. It is the fact that $\uu_0^{\th}$ belongs to an open set which allows us to ensure that the initial vorticity $\omega_0$ does not vanish at $R=0$.
\end{remark}

\begin{remark}[\bf The curl blows up slower than the divergence]
\label{rem:slow:curl}
Since $\omega_0(0)\neq 0$, identity \eqref{eq:blow:asym:d} shows that the vorticity blows up at $R=0$ at the rate $(T-t)^{-\frac{r-1}{\alpha r}}$. 
Moreover, since the admissible blowup speed $r$ satisfies the constraint~\eqref{r-range}, it follows that 
$0 < \frac{r-1}{\alpha \, r} < \frac{r_{\mathsf{eye}}(\alpha)-1}{\alpha \, r_{\mathsf{eye}}(\alpha)} < 1$ for all $\alpha>0$, and therefore
$\int_0^T \|\omega(\cdot,t)\|_{L^\infty} dt \approx \int_0^T |\omega(0,t)| dt < + \infty$.  As the magnitude of vorticity is integrable in time, there does not exist  a Beale-Kato-Majda type blowup criterium for the compressible Euler equations.  
As to the effects of compression, from  \eqref{eq:blow:asym:a} and \eqref{eq:rep3}, we have that $|\!\div \uu(0,t)| \approx (T-t)^{-1}(2/r)  |\bar U^\prime(0)|=(T-t)^{-1}  (r-1)/(\alpha r) $, which implies 
that $\int_0^T \|\div \uu(\cdot,t)\|_{L^\infty} dt = + \infty$.  The fact that $\operatorname{div} u$ blows up faster than $\omega$ indicates that sound waves steepen in compression much faster than particles can ``swirl''.
\end{remark}
   
\begin{remark}[\bf There are no singularities away from $R=0$ and $t=T$]
\label{rem:blowup:at:0}
While Theorem~\ref{thm:blowup} shows that the rescaled sound speed $\sigma(0,t)$ and the magnitude of the gradients  $|\nabla \uu|(0,t)$ and  $|\nabla \sigma|(0,t)$ blow up as $t \converges T$, we emphasize that, by construction, the Euler solution in Theorem~\ref{thm:blowup} does not form  singularities for times $t<T$, and moreover, as $t\converges T$ no singularities develop at points $x  \in {\mathbb R}^2$ with $x\neq 0$. Indeed, for $x\neq 0$ and $t\in [0,T)$ arbitrary, upper bounds for $|\uu(x,t)|$ are obtained in~\eqref{eq:bounds:for:original:variables:a}, upper and lower bounds bounds for $\sigma(x,t)$ are given by~\eqref{eq:bounds:for:original:variables:c}, while upper bounds for $|\nabla \uu(x,t)|$ and $|\nabla \sigma(x,t)|$ are established in~\eqref{eq:bounds:for:original:variables:b}. We refer to Section~\ref{sec:additional:estimates} for these details.
\end{remark}

\subsection{Recent results concerning singularity formation for Euler}
\label{sec:literature}
The literature concerning the formation and propagation of singularities for the compressible Euler equations (and related hyperbolic PDEs) is too vast 
to review here.\footnote{In contrast, relatively few results have been established for the finite time singularity formation of the {\em incompressible} Euler equations; for instance, finite-time vorticity blowup remains open for the incompressible Euler equations posed on ${\mathbb R}^d$ or ${\mathbb T}^d$, for $C_c^\infty$, respectively $C^\infty$, initial conditions. 
We note, however, that very recently, the finite time blowup for the vorticity was established  in the presence of smooth solid boundaries~\cite{ChHo2023,ChHo2023b}, for the equations posed in conical domains~\cite{ElJe2019}, or in the absence of boundaries for $C^{1,\alpha}$ data with $\alpha$ very small~\cite{El2021,ChHo2021,ElGhMa2021,ElPa2023,CoMZZh2023}.} 
 In this section, we shall focus only on results for the {\em Euler equations},\footnote{There is a large literature concerning singularity formation for quasilinear wave equations; see the recent review articles~\cite{HoKlSpWo2016,Sp2016}.} and review only those results for the  {\em multi-dimensional problem}.\footnote{For a detailed exposition of the 1D theory we refer the reader to~\cite{ChWa2002,Da2005,Li2021} and references therein.}

As noted above, it is well-known that the Euler equations develop finite-time gradient singularities from initial data which is as {\it nice as it gets}. The first proof of this fact was given in~\cite{Si1985}, using a {\em  proof by contradiction}, inspired by the classical results~\cite{La1964,Jo1974}. In recent decades, authors have given {\em constructive proofs} of singularity formation, which yield a very precise description of the solution at the time of the first gradient blowup. We categorize these constructive proofs into {\em shock formation} and {\em smooth implosions}.

\subsubsection{Shock formation}
Shock formation from smooth and localized initial conditions is, by now,  well understood, {\em up to the time of the first gradient singularity}. In the context of  irrotational solutions for the multi-dimensional Euler equations,  blowup results  based on ideas developed for second-order
quasilinear wave equations and general relativity were established in~\cite{Al1999a,Al1999b,Ch2007,ChMi2014}; the latter two results were generalized to allow for vorticity and entropy in the Euler solution~\cite{LuSp2018,LuSp2024}; results for 2D Euler solutions, with azimuthal symmetry, using modulated self-similar analysis were proven in~\cite{BuShVi2022,BuIy2022}; results for 3D Euler solutions in the absence of symmetry assumptions, based on modulated self-similar stability analysis were established in~\cite{BuShVi2023a,BuShVi2023b}.
In the context of shock development,  results for Euler solutions with  spherical symmetry  were obtained in~\cite{Yi2004,ChLi2016}, while results for 
Euler flows with azimuthal symmetry were proven in~\cite{BuDrShVi2022}.
  
Very recently, shock formation {\em past} the time of the first singularity was analyzed in~\cite{ShVi2023} and~\cite{AbSp2022}. In~\cite{ShVi2023} the last two authors of this paper  analyze the evolution of fast-acoustic characteristic surfaces and use this smooth geometric information to describe the Euler solution up to a portion of the boundary of the maximal globally hyperbolic development  (MGHD) of smooth Cauchy data, which includes both the
singular set as well as the Cauchy horizon; the results of~\cite{AbSp2022} evolve the Euler solution up to a portion of the boundary of the MGHD, containing only the
singular set.

\subsubsection{Smooth Implosions}
The classical Guderley problem~\cite{Gu1942} deals with the implosion of a converging shock-wave, and the continuation as an expanding shock. In spite of the vast literature on this subject, the existence of a finite-time  {\em smooth implosion} (without a shock already present in the solution at the initial time), was established only recently in~\cite{MeRaRoSz2022a,MeRaRoSz2022b}, and further refined in~\cite{CLGSShSt2023,BuCLGS2022}. While these results are inspired by the ansatz and general strategy in~\cite{Gu1942}, the existence of $C^\infty$ self-similar implosion profiles requires a great deal of mathematical sophistication. Moreover, the stability of these solutions can only be established for finite-co-dimension perturbations which are not precisely quantified (see the numerical work~\cite{Bi2021} which attempts to quantify some of these instabilities). As explained in Section~\ref{sec:main:ideas} below, this finite co-dimension stability of the smooth implosion profiles is one of the fundamental obstacles in proving vorticity blowup.

\subsubsection{Solutions evolving from a singular state}
While throughout this paper we are only concerned with smooth initial data, and we are only interested in the behavior of the resulting Euler solution up to the time of the first gradient singularity, there are many physically important problems in which the Euler evolution is initiated with data that already contains a singular state. A non-exhaustive list of examples includes the following:
the short time evolution of discontinuous shock fronts~\cite{Ma1983a,Ma1983b};
the shock reflection problem~\cite{ChFe2010,ChFe2018};
the shock implosion~\cite{JaJiSc2023} and shock explosion~\cite{JaJiSc2024} problems; 
non-isentropic implosions with strictly positive pressure fields~\cite{JeTs2020,JeTs2023};
the stability of rarefaction waves evolving from multi-dimensional Riemann data~\cite{LuYu2023a,LuYu2023b};
the long-time stability of an irrotational expansion fan bounded by two shock surfaces~\cite{GiRo2024}; 
and the ill-posedness of the Euler equations in multiple space dimensions for admissible weak solutions~\cite{ChDLKr2015,KlKrMaMa2020,DeSkWi2023}. We also mention that there is a vast literature on transonic shocks for Euler both in the context of airfoils~\cite{Mo1985,ChSlWa2008} and in the setting of nozzles~\cite{XiYi2005,XiYaYi2009,WaXi2019}; we refer the reader to the recent excellent review articles~\cite{Mo2004,ChFe2022,Ch2024} for an extensive literature review.  

\subsubsection{Other singularities}
A complete classification of all possible singularities that {\em multi-dimensional and smooth} Euler dynamics can reach in finite time is, to date, elusive. As discussed above, it is known that for regular initial data, at the time $T$ of the first singularity  the Euler solution is capable of developing a shock, an implosion, and the goal of this paper is to show that it can also develop a vorticity blowup. It remains unsettled if these are the  {\em only} possible {\em first} singularities attainable by smooth Euler evolution  (with the classical ideal gas law equation of state).   It is,  for example, an open problem to determine whether vacuum\footnote{By this we mean that at the time $T$ of the first singularity, the density $\rho(\cdot,T)$ is equal to $0$, at at least one spatial point.} can form in finite time from smooth, localized, and non-vacuous initial data, in multiple space dimensions.\footnote{In one space dimension, vacuum cannot dynamically arise form from nice initial data without a shock forming first~\cite{Ch2017,ChPaZh2019}.}
 
\section{Main ideas of the proof and comparison with previous results}
\label{sec:main:ideas}

Recall that in the two-dimensional setting of this paper, the specific vorticity $\om / \rho$ is transported, see~\eqref{eq:vorticity:transport}. In order to prove Theorem \ref{thm:blowup}, we construct initial data with non-trivial vorticity near $R = 0$, and for this data we construct an imploding solution by perturbing the radial implosion from~\cite{MeRaRoSz2022a}.  
We emphasize that the  finitely many unstable modes of the self-similar blowup profile from~\cite{MeRaRoSz2022a} can potentially lead  to either {\em trivial vorticity}, or to the {\em destruction of the smooth implosion mechanism}. 
To overcome this difficulty, we will crucially use the two-dimensional axisymmetric ansatz~\eqref{eq:ansatz} and the structure of the equations~\eqref{euler-r}, in order to establish {\em full-stability of the swirl velocity}, and hence of the vorticity.

\subsection{Self-similar coordinates}
\label{sec:self-similar:variables}
We introduce the self-similar coordinates, variables, and  equations used in our analysis. For $r > 1$ which satisfies~\eqref{r-range} below, define the self-similar time  and space coordinates via
\begin{subequations}
\label{ss-var}
\begin{equation}
 s  =   \log\frac{1}{(T-t)^{1/r}} , \quad 
 y = \frac{x}{(T-t)^{1/r}}, \quad 
 \xi = |y| = \frac{R}{(T-t)^{1/r}} .
 \label{ss-var:a}
\end{equation}
Throughout the paper we will use $y \in \R^2$ to denote the self-similar vectorial space coordinate, and $\xi = |y|$ to denote the self-similar radial variable. 
The self-similar velocity components and rescaled sound speed are then given by 
\begin{equation}
\bigl(u^R,u^\theta,\sigma\bigr)(R,t) 
= \tfrac{1}{r}  (T-t)^{\frac{1}{r}-1} \bigl(U,A,\Sigma\bigr)(\xi,s).
 \label{ss-var:b}
\end{equation}
It is convenient to also introduce the vectorial form of the self-similar velocity, that is
\begin{equation}
\UU(y,s) = U(\xi,s) \ee_R, \quad \vA(y,s) = A(\xi,s) \ee_{\th},
\label{ss-var:c}
\end{equation}
\end{subequations}
where as before, $\xi = |y|$.
Throughout the paper we will assume that the parameter $r$ appearing in \eqref{ss-var:a} and \eqref{ss-var:b} satisfies the following inequalities\footnote{The definition of $r_{\mathsf{eye}}$ is directly adapted from~\cite{MeRaRoSz2022a}, upon letting $d=2$ and $\gamma = 1 + 2\alpha $; \eqref{r-range} is the same as \cite[equation (1.15)]{MeRaRoSz2022a}.}
\begin{subequations} 
\label{r-range}
\begin{align}
1<r<r_{\mathsf{eye}}(\alpha),
\qquad
&r_{\mathsf{eye}}(\alpha) = \tfrac{1 + 2\alpha}{1+\alpha \sqrt{2}}, 
& \alpha > \tfrac{1}{2}, \label{r-rangea} \\
\tfrac{1 + 2\alpha}{1+\alpha \sqrt{2}}< r<r_{\mathsf{eye}}(\alpha), 
\qquad
&
r_{\mathsf{eye}}(\alpha) =1+ \tfrac{\alpha}{(\sqrt{\alpha}+1)^2},
&
0 < \alpha < \tfrac{1}{2} \label{r-rangeb}.
\end{align}
\end{subequations}

Using the vectorial self-similar variables introduced in \eqref{ss-var:c}, and with the self-similar rescaled sound speed from \eqref{ss-var:b}, the two-dimensional axisymmetric Euler system~\eqref{euler-r} may be rewritten as
\begin{subequations} 
\label{eq:euler_ss}
\begin{align}
& \pa_s \UU + (r-1) \UU + (y + \UU) \cdot \na \UU + \alpha \S  \na \S  = - \vA \cdot \na \vA , \\
& \pa_s \vA  + (r-1)\vA+ (y+ \UU) \cdot \na \vA  = - \vA \cdot \na \UU   ,  \\
& \pa_s \S + (r-1)\S + (y + \UU) \cdot \na \S + \alpha \S \, \div(\UU)   = 0 .
\end{align}
\end{subequations}
The initial data is prescribed at the self-similar time 
\begin{equation*}
\sin \teq - \tfrac{1}{r} \log T,
\end{equation*} 
which corresponds to $t=0$ in \eqref{ss-var:a}.

\subsection{Globally self-similar blowup profiles}
We summarize the results of \cite[Theorem 1.2 and Corollary 1.3]{MeRaRoSz2022a}, which give the existence of smooth stationary self-similar profiles in $\R^2$, i.e.~of smooth solutions $(\bar \UU, 0 , \bar \S)$ to~\eqref{eq:euler_ss} which are $s$-independent. 

\begin{theorem}[\bf Existence of smooth similarity profiles~\cite{MeRaRoSz2022a}]
\label{thm:merle:implosion}
There exists a (possibility empty) exceptional countable sequence $\Gamma = \{ \g_n\}_{n \geq 1}$ whose accumulation points can only be $\{1, 2, \infty \} $ such that the following holds. For each $\g \notin \Gamma$, 
there exists a discrete sequence of blow up speeds $\{ r_k \}_{k\geq 1}$  in the range \eqref{r-range}, accumulating at $r_{\mathsf{eye}}(\frac{\gamma-1}{2})$, such that the system \eqref{eq:euler_ss} admits a $C^\infty$-smooth radially symmetric stationary solution $\bar \UU(y) = \bar U(\xi) \ee_R, \bar \AA = 0, \bar \S(\xi)$.
\end{theorem}
The requirements on $\gamma$ and $r_k$ stated above are the same as in~\cite[Theorem 1.2]{MeRaRoSz2022a}, with $d=2$. Throughout the paper we will fix a non-exceptional adiabatic exponent $\gamma \not \in \Gamma$, a similarity exponent $r \in \{r_k\}_{k\geq 1}$, and a similarity profile $(\bar U,0,\bar{\S})$ in our analysis. We treat any constants related to the profile $(\bar U,0,\bar{\S})$, to $\alpha = \frac{\gamma-1}{2}$, and to $r \in(1,r_{\mathsf{eye}}(\alpha))$, as absolute constants. 

In the next lemma we collect some useful properties of the profiles $(\bar{U}, \bar{\S})$ constructed in~\cite{MeRaRoSz2022a}.

\begin{lemma}[\bf Properties of the similarity profiles]
\label{lem:profile}
The radial vector field $\bar U(\xi) \ee_R$  and the radially symmetric function $\bar{\Sigma}(\xi) > 0$ are smooth. For every integer $k \geq 0$ we have the following decay\footnote{Here and throughout the paper we denote by $\langle \cdot \rangle$ the quantity $\sqrt{1+|\cdot|^2}$.} as $\xi \to \infty$:
\begin{subequations}
\label{eq:profile:properties}
\begin{equation}
|\partial^k_\xi  \bar{U}(\xi)| \les_k  \la \xi \ra^{1-r-k} , \quad
|\partial^k_\xi  \bar{\S}(\xi)| \les_k  \la \xi \ra^{1-r-k} .
\label{eq:dec_U}   
\end{equation}
There exists $\kp>0$ and $\xi_1 > \xi_s$\footnote{The value $\xi_s$ corresponds to the point $P_2$ in the phase portrait of~\cite{MeRaRoSz2022a} or to the point $P_s$ in the 3D phase portrait of~\cite{BuCLGS2022}.} such that 
\begin{align}
1 + \partial_\xi \bar{U}(\xi) - \alpha |\pa_{\xi} \bar{\S}(\xi)| &> \kp,  \quad \xi \in [0, \xi_1] \label{eq:rep1}, \\
\xi + \bar U(\xi) - \al \bar{\S}(\xi)&> 0, \quad \xi > \xi_s, \label{eq:rep2} \\
1 + \xi^{-1} \bar U(\xi)  &> \kp, \quad \xi \geq 0 .  \label{eq:rep22}
\end{align}
Lastly, since $r$ lies in the range \eqref{r-range}, we have
\begin{equation}
{\lim}_{\xi \to 0^+} \xi^{-1} \bar{U}(\xi)  = \partial_\xi \bar{U}(0) =- \tfrac{r-1}{2 \alpha} > -\tfrac{r}{2}. 
\label{eq:rep3}
\end{equation}
\end{subequations}
\end{lemma}
Properties~\eqref{eq:profile:properties} of the profiles $(\bar U, \bar \Sigma)$ were either stated as is in~\cite{MeRaRoSz2022a}, or they follow from the proof given there. For instance, the outgoing property of the flow~\eqref{eq:rep2} follows from the phase portrait proved in~\cite{MeRaRoSz2022a} and continuity. The limit~\eqref{eq:rep3} has been proved in~\cite{MeRaRoSz2022a}. The interior repulsive property~\eqref{eq:rep1} is proved in~\cite[Lemma 1.6]{MeRaRoSz2022a} for $\xi \in [0, \xi_s]$; since the solution is smooth, there exists $\xi_1>\xi_s$ such the condition~\eqref{eq:rep1} holds for a domain $[0, \xi_1]$ strictly larger than $[0, \xi_s]$. The decay property~\eqref{eq:dec_U} was stated in~\cite[Theorem 2.3]{MeRaRoSz2022b}. 
For the sake of completeness, we give a proof of Lemma~\ref{lem:profile} in Appendix~\ref{app:ODE}. 

\begin{remark}[\bf No repulsive property in the far field]
\label{rem:no:repulsive}
The repulsive property~\eqref{eq:rep1} may not hold true in the entire exterior domain $\xi > \xi_s$ for the 2D compressible Euler equations, which is related 
to \cite[Lemma 1.7]{MeRaRoSz2022a}. We also quote 
\begin{quote}
\textit{
The Euler case in $d = 2$ and $d = 3$ for $l \leq \sqrt{3}$ requires special consideration.
In those cases, property (P) of (2.24), which ensures coercivity of the corresponding quadratic form in (6.14), does not hold for $Z > Z_2$.
}
\end{quote}
from Section 8 on Page 846 in~\cite{MeRaRoSz2022b}, which suggests that the exterior repulsive property does not hold. The property (P) is related to the exterior repulsive property,  and $Z_2$ corresponds to $\xi_s$ in \eqref{eq:rep1}. 
\end{remark}

\subsection{Evolution of the perturbation}
\label{sec:perturbation}

We consider a solution to the two-dimensional axis-symmetric Euler system in self-similar coordinates \eqref{eq:euler_ss} in perturbative form:
\begin{equation}
(\UU,\vA,\Sigma) 
=  ( \bar \UU, 0, \bar{\S} ) + (\td \UU, \vA, \ts)
=  ( \bar U \, \ee_{R}, 0, \bar{\S} ) + (\td U \, \ee_{R},  A \, \ee_{\th}, \ts).
\label{eq:linear+nonlinear}
\end{equation} 
Our goal is to analyze the evolution of the perturbation
\begin{subequations}
\label{eq:lin}
\begin{align}
\pa_s \td \UU  &= \cL_U( \td U, A, \td \S) + {\cN}_{U} 
\label{eq:lin:a} \\
\pa_s \AA &= \cL_A( \td U, A, \td \S) + \cN_A  \label{eq:lin:b} \\ 
\pa_s \td \S &= \cL_{\S}( \td U, A, \td \S) + \cN_{\S} 
\label{eq:lin:c}
\end{align}
where the linearized operators are defined as
\begin{align} 
& \cL_U( \td U, A, \td \S) = \cL_U( \td \UU, \td \S) =  -  (r-1) \td \UU -  (y +\bar{ \UU}) \cdot \na \td \UU  -  \td \UU \cdot \na \bar{\UU} 
 - \alpha \ts \na \bar{\S}  - \alpha \bar{\S} \na \ts  
\label{eq:lin:LU}\\
 & \cL_A( \td U, A, \td \S)  = \cL_A(\vA)  =  - (r-1) \vA -   (y +\bar{ \UU}) \cdot \na \vA  - \vA \cdot \na \bar \UU  
 \label{eq:lin:LA}\\
 & \cL_{\S}( \td U, A, \td \S) = \cL_{\S}( \td \UU, \td \S) =  -  (r-1) \ts - (y +\bar{ \UU}) \cdot \na \td \S
 -  \td \UU  \cdot \na \bar{\S}     - \alpha \ts  \div(\bar \UU)
 - \alpha \bar{\S} \div(\td \UU)  
 \label{eq:lin:LS}
\end{align}
\end{subequations}
and the nonlinear terms are defined as
\begin{subequations}
\label{eq:non}
\begin{align}
& {\cN}_{  U} =  -  \td \UU \cdot \na \td \UU - \alpha \ts \na \ts   - \vA \cdot \na \vA  \\
& {\cN}_{ A} =   - \td \UU \cdot  \na \vA  - \vA \cdot \na \td \UU  \\
& {\cN}_{ \S} =  - \td \UU \cdot \na \td \S -  \alpha \ts   \div(\td \UU) .
\end{align}
\end{subequations}
The initial data for the system~\eqref{eq:lin} is specified at self-similar time $\sin = \log T^{-1/r}$.

\subsection{Two key observations}
\label{sec:two:observations}
First, in the linearized equations around the profile $(\bar U, 0,\bar \Sigma)$, the evolution of 
the perturbation $ \AA$ and those of the perturbation $(\td \UU, \td \S)$ are decoupled. 
That is, in~\eqref{eq:lin} the linearized operators $\cL_U$ and $\cL_{\S}$ do not depend on $A$.
Therefore, we can perform linear stability analysis of $(\td \UU, \td \S)$ and $\AA$ separately. We will show that the linearized equations for $(\td \UU, \td \S)$ are stable upon modulating potentially finitely many unstable directions, developing a simpler stability proof.  See Section \ref{sec:new_pf}.

Second, from \eqref{eq:lin:b} we deduce that  the {\em linearized equation} satisfied by $A(\xi,s) / \xi$ is given by\footnote{Recall from~\eqref{ss-var:c} that $\AA(y,s) = A(\xi,s) \ee_\theta$, and thus the self-similar vorticity at $\xi=0$ is related to $\lim_{\xi\to0^+} A(\xi,s)/(2\xi)$. We note that for a smooth perturbation, we have $A(0,s) = 0$ and hence $A / \xi$ is well-defined.}
\begin{equation*}
 \pa_t \Bigl( \frac{A}{\xi} \Bigr)
 + (  \xi + \bar U) \pa_{\xi} \Bigl( \frac{A}{\xi} \Bigr)
 = - \Bigl( r + 2 \frac{\bar U}{\xi} \Bigr) \Bigl( \frac{A}{\xi} \Bigr).
\end{equation*}
Our crucial observation is that the blowup profile satisfies $r + 2 (\bar U(\xi) / \xi) < 0$ near $\xi = 0$, see~\eqref{eq:rep3}, which suggests the linear stability of $A(\xi) / \xi$ near $\xi =0$. Using this property, and the outgoing property of the self-similar velocity, namely that 
$ \xi + \bar U(\xi) \geq \kp \xi$ for for some $\kp >0$ (see~\eqref{eq:rep22}), we obtain that $A$ is stable in a suitably weighted space. These stability properties can also be captured by performing weighted $L^2$ estimates with weights that are singular near $\xi= 0$. See Theorem \ref{thm:coer_est}.

In order to prove Theorem~\ref{thm:blowup}, we note that since the radial velocity  $\bar \UU + \td \UU$ has $0$ vorticity, and since $\AA$ is stable, we can choose initial data for $\AA$ in an \textit{open} set (in the topology of a weighted Sobolev space) and thus construct non-trivial initial vorticity. Moreover, since $(\UU + \td \UU+ \AA)(0,s) = 0$, the origin is a stationary point of the flow; thus, once we prove global nonlinear stability for~\eqref{eq:lin}--\eqref{eq:non}, the finite-time blowup of vorticity claimed in~\eqref{eq:blow:asym:d} now follows.

\subsection{A new stability proof of implosion}\label{sec:new_pf}

The goal of this paper is to also develop a new stability proof of the implosion for the compressible Euler equations, which is simpler than the one given  in~\cite{MeRaRoSz2022b,MeRaRoSz2022c} (based on estimates of nonlinear wave equation), and  in~\cite{BuCLGS2022} (based on the transport structure of Riemann-type variables).

While in this paper we focus on the two-dimensional setting, the stability proof for the Euler equations in $\R^d$ with radially symmetric data (around a profile satisfying the properties in Lemma~\ref{lem:profile}) would be the same,  by setting $\vA = 0$ in Sections~\ref{sec:lin}, and~\ref{sec:non}. Our stability  proof consists of the following steps. 

\subsubsection*{Step 1: Weighted energy estimates}

In the linear stability analysis of Section~\ref{sec:coer_Hm} we perform weighted $H^{2m}$ estimates, with sufficiently large $m$ and a carefully designed weight $\vp_{2m}$, to obtain coercive estimates for the ``top order'' terms. The interior repulsive property \eqref{eq:rep1} allows us to obtain a damping term proportional to $m$ in the region $|y| = \xi  \leq \xi_1 $,  slightly beyond the sonic point $\xi_s < \xi_1$.  
For $|y| = \xi > \xi_s$, we use the outgoing property related to the transport term \eqref{eq:rep2}. We design a radially symmetric weight $\vp_{2m}(y)$ decreasing in $|y| = \xi$ for $\xi \in [\xi_1, R_2]$, with $R_2$ large enough.  
The monotonicity of the weight and the outgoing property provides a strong damping term in the energy estimates, as seen in Lemma~\ref{lem:wg}. 
For $|y| = \xi > R_2$, due to the natural decay of the profile~\eqref{eq:dec_U}, the linear dynamics is purely governed by the scaling terms, and is thus stable. The upshot of the linear stability analysis is that we obtain weighted $H^{2m}$ estimates of the kind
\begin{equation}
\label{eq:intro_coer}
 \la \cL (\UU, \S), (\UU, \S) \ra_{\cX^m} 
 \leq - \lam \| (\UU, \S)\|_{\cX^m}^2 + \int_{|y| \leq R_4}  \cR_m d y ,
\end{equation}
for some inner product $\cX^m, m \geq m_0$ with $m_0$ sufficiently large, a suitably chosen $R_4>0$ and a coercive parameter $\lam > 0$ (uniform in $m \geq m_0$). 
Here, $\cR_m$ denotes some ``lower order'' remainder terms (in terms of  regularity, when compared to  $\| (\UU, \S)\|_{\cX^m}^2$), which is made explicit in~\eqref{eq:coer_est}. 
For the stability estimates of the  $\AA$ variable, we use the ideas discussed earlier in Section~\ref{sec:two:observations}.

\subsubsection*{Step 2: Compact perturbation}

To handle the lower order remainder term $\cR_m$ in \eqref{eq:intro_coer}, for each $m\geq m_0$, in Section~\ref{sec:Riesz} we use 
Riesz representation  to construct a linear operator $\cK_m$ from a bilinear form similar to
\begin{equation*}
 \la \cK_m f , g \ra_{\cX^m} =  \int \chi  f \cdot  g d y,
\end{equation*}
where $\chi \geq 0$ is a smooth cutoff function with $\chi = 1$ for $|y| \leq R_4$. The second term on the right side of \eqref{eq:intro_coer} then arises as $ \la \cK_m f , f \ra_{\cX^m}$ for $f= (\UU,\Sigma)$. Since the right hand side in the above display does not involve derivatives on $f$ and $g$, and since $\chi$ has a compact support,  we can show that $\cK_m : \cX^m \to \cX^m$ is {\em compact} and that it has a smoothing effect $\cK_m: \cX^0 \to \cX^{m+3}$. Afterwards, we obtain that $\cL - \bar C  \cK_m$ is {\em dissipative}, for $\bar C>0$ chosen large enough.

\subsubsection*{Step 3: Construction of semigroup}

To handle the compact perturbation, in Sections~\ref{sec:semi} and~\ref{sec:semi_grow} we use semigroup theory to estimate the growth bound for the semigroups $e^{ (\cL - \cK_m) s}$ and $e^{ \cL  s} $.  We prove a local well-posedness (LWP) result of the linearized equations in the space $\cX^m$ in order to construct the semigroup. Note that the linearized equations are a symmetric hyperbolic system, whose LWP theory is classical. 

\subsubsection*{Step 4: Smoothness of unstable direction}

To construct a smooth initial data, in Section~\ref{sec:smooth_unstab} we show that the potentially unstable directions of the linearized operators are spanned by smooth functions. An  unstable direction $f$ satisfies a variant of the equation $(\cL - \lam ) f = 0$, for $f \in \cX^{m_0}$. Rewriting this as $(\cL -  \cK_m -\lam ) f =  -\cK_m f$, the smoothing effect of $\cK_m \colon \cX^{0} \to \cX^{m+3}$, the invertibility $(\cL -  \cK_m -\lam )^{-1} \colon \cX^m \to \cX^m$, 
and induction on $m \geq m_0$, allows us to prove  prove $f \in C^{\infty}$. See Lemma~\ref{lem:smooth}.

We note that {\em Steps 2 \& 4} discussed above resemble the functional approach to solve the elliptic equations, see e.g.~\cite{Ev2022}. The coercive estimates \eqref{eq:intro_coer} allow us to construct the compact perturbation, which has a smoothing effect. These steps do not depend on specific structure of the equations.

\subsubsection*{Step 5: Nonlinear estimates}

To close the nonlinear stability estimates, in Section~\ref{sec:non} we decompose the perturbation $\td \VV = (\td \UU, \AA, \td \S)$ into a stable part $\td \VV_1$ and an unstable part $\td \VV_2$, following the splitting method of~\cite{ChHo2023}. For the unstable part, we represent $\td \VV_2$ using Duhamel's formula and estimate it using the semigroup $e^{\cL s}$. 
The stable part $\td \VV_1$ is estimated purely using  energy estimates. Such a decomposition allows us to construct the global in self-similar time solution using a Banach fixed point argument and simpler bootstrap assumptions; moreover, this decomposition implicitly  determines the stable initial data in the correct finite co-dimension space. See~\cite{JiSv2015,ElPa2023} for related constructions of a global-in-time self-similar solution based on a Banach fixed point argument. 

\subsection{Comparison with the stability estimates in~\cite{MeRaRoSz2022b},~\cite{MeRaRoSz2022c} and~\cite{BuCLGS2022}}

Part of our stability estimates are similar to those in \cite{MeRaRoSz2022b},~\cite{MeRaRoSz2022c}, and in~\cite{BuCLGS2022}: high order weighted Sobolev estimates to extract the dissipation, 
semigroup theory to estimate the compact perturbation, and finite co-dimension stability. Moreover, our $H^{2m}$ stability estimates based on the variable $(\UU, \S)$ are inspired by~\cite{BuCLGS2022} particularly in the region $|y| = \xi \leq \xi_s$. There are however several significant differences between our estimates and those in~\cite{MeRaRoSz2022b,MeRaRoSz2022c,BuCLGS2022}. 

First, we do not localize the linearized operator, nor do we use Riemann-type variables to perform transport estimates. Instead, we  
exploit the outgoing properties of the flow~\eqref{eq:rep2} by designing decreasing weights and incorporate such an effect in the weighted $H^{2m}$ estimates. This allows us to obtain simpler nonlinear estimates. We note that weighted energy estimates with singular weights and weights depending on the profiles have been used successfully in a series of works by the first author and collaborators~\cite{ChHoHu2021,ChHo2021,ChHoHu2022,ChHo2023,Ch2020a}.

Second, we construct the semigroup based on the local well-posedness (LWP) of the linearized equations, which is obtained by appealing to a classical result.
See~{\em Step 2} mentioned above, and Theorem~\ref{thm:ACP}. 
In contrast, in~\cite{BuCLGS2022,MeRaRoSz2022c}, the construction is based on the notion of a maximal dissipative operator. Moreover, using the functional approach outlined in {\em Step 4} above, we do not need to solve specific ODE or PDE problems as~\cite{BuCLGS2022,MeRaRoSz2022c}, and hence we obtain simpler linear stability estimates. 

Third, we use a Banach fixed point argument to obtain nonlinear finite co-dimension stability, different from the topological arguments employed in~\cite{BuCLGS2022,MeRaRoSz2022a}. 

Finally, we do not use the exterior repulsive property of the profile in the stability proof, which is required in \cite{BuCLGS2022}, but may not hold for the similarity profile of 2D Euler equations (see the discussion in Remark~\ref{rem:no:repulsive}). Thus, we cannot adopt the stability proof in \cite{BuCLGS2022} to establish the finite co-dimension stability of $(\td \UU, \td \S)$.

Our stability estimates can be extended to the compressible Navier-Stokes equations in the same settings as in~\cite{BuCLGS2022,MeRaRoSz2022b}, 
with a minor modification since the diffusion is treated perturbatively, and the main term of the diffusion in the weighted estimates still has a good sign after performing integration by parts. 

\subsection{Comparison with the non-radial implosion in~\cite{CLGSShSt2023}}
\label{sec:intro_comp}

There appear to be essential difficulties in using the non-radial imploding solutions built in~\cite{CLGSShSt2023} to exhibit vorticity blowup for the compressible Euler equations in three space dimensions.

We recall that for $d=3$, the specific vorticity of isentropic Euler satisfies the vector-transport equation 
\beq \label{spec-3d}
\partial_t \Bigl( \frac{\om}{\rho} \Bigr) + (\uu \cdot\nabla) \Bigl( \frac{\om}{\rho}\Bigr)  = (\uu  \cdot\nabla) \Bigl(\frac{\om}{\rho}\Bigr).
\eeq
If we denote by $X(a,t)$ the Lagrangian flow\footnote{The Lagrangian flow $X$ solves the ODE $\partial_t X(a,t) = u(X(t,a),t)$, with initial condition $X(a, 0)=a$.
\label{footnote:flow:map}} generated by $u$, then \eqref{spec-3d} can be rewritten as
\begin{equation} 
\Bigl( \frac{\om}{\rho}\Bigr)(X(a,t), t) 
=  \nabla_a X(a,t)  \Bigl(\frac{\om}{\rho}\Bigr)(a, 0).
\label{eq:vector:transport}
\end{equation}
If we imagine $\uu$ and $\sigma$ (and hence $\rho = (\alpha \sigma)^{1/\alpha}$) to be {\em exactly the self-similar solutions} to \eqref{eq:euler} built in \cite{MeRaRoSz2022a} and~\cite{BuCLGS2022}, then the vorticity $\om$ would just be identically $0$ by radial symmetry. 

Nevertheless, we can replace the $\omega(X(a,t),t)$ appearing on the right side of~\eqref{eq:vector:transport} with a generic quantity \begin{equation*} 
f = f(X(a,t),t)= \rho( X(a,t),t) \nabla_a X(a,t) \tfrac{f_0}{\rho_0}(a) ,
\end{equation*} 
and ask whether $\| f (\cdot,t)\|_{L^\infty} $ blows up as $t\converges T$, with $T$ being the time of the implosion. In the next lemma we observe that $f$ remains uniformly bounded in $a$ and $t$, as long as we choose the initial data $f_0$ with a high-enough vanishing order at $a=0$. 

\begin{lemma}[\bf Not all transported quantities blow up in a 3D implosion]
\label{lem:non_blowup}
Consider a self-similar profile $(\bar{U},\bar \Sigma)$ as constructed in~\cite{MeRaRoSz2022a} or~\cite{BuCLGS2022}, and let $(\uu,\sigma)$ be determined according to exact self-similarity from this profile via~\eqref{ss-var}.
Then, there exists $k_{\sf van}>0$ such that for every smooth function $f_0$ that satisfies $|f_0(a)| \les \min(|a|^{k_{\sf van}}, 1) $, we have
\begin{equation*}
\sup_{a \in \R^3, t \in [0,T)} | \rho( X(a,t), t) \nabla_a X(a,t)  \tfrac{f_0}{\rho_0} (a)  |< + \infty .
\end{equation*}
The required vanishing order $k_{\sf van}$ is some constant that depends only on the profile $(\bar U,\bar \Sigma)$, on $r$, and on $\alpha$.
\end{lemma}
We defer the proof to Appendix~\ref{sec:non_blowup}.
At this stage we note that the finite co-dimension of the initial data considered in~\cite{CLGSShSt2023} equals to the number of unstable directions of $\cL$, which is a (sufficiently large) finite rank perturbation $\cK$ to a maximally dissipative operator. Since the rank of $\cK$ is used as a large parameter, the number of unstable directions $k_{\sf uns}$ is not quantified in \cite{CLGSShSt2023}, and therefore it is hard to determine whether $k_{\sf uns} < k_{\sf van}$. Moreover, it is not shown in~\cite{CLGSShSt2023} that the finite co-dimension set of initial data contain any functions $f_0$ violating the vanishing order $f_0 =\mathcal{O}(|a|^{k_{\sf van}} )$ near $a = 0$.
Therefore, in light of Lemma~\ref{lem:non_blowup} it is not clear whether the setting of~\cite{CLGSShSt2023} is at all amenable to vorticity blowup.

\subsection{Notation}
Next, we discuss the notation used in the paper.

We use $A\les B$ to mean that there exists a constant $C = C(\gamma, r, \bar U,\bar \Sigma,T)\geq 1$ such that $A \leq C B$. We use $A \asymp B$ to mean that both $A \les B$ and $B \les A$.
We use $A\teq B$ to say that quantity $A$ is {\em defined} to be equal to quantity/expression $B$. In Section~\ref{sec:non} we shall use the notation $A= B + \mathcal{O}_{h}(B^\prime)$ if there exists $C_{h} = C (\gamma,r,\bar U,\bar \S, T,h) \geq 1$ such that $|A - B| \leq C_\ast B^\prime$. 

We will use $\eta, \eta_s, \lam, \lam_1$ to denote constants related to the decay rates; see~\eqref{eq:decay_stab}, \eqref{eq:decay_unstab}, and~\eqref{eq:decay_para}.
We use $\kp_1, \kp_2, \kp_3$ to denote exponents for the weights; see~\eqref{eq:wg}, \eqref{eq:vp:g:def}, and~\eqref{eq:kappa:3:def}. We use $\vp_1, \vp_m, \vp_b, \vp_f,  \vp_g, \vp_A$ to denote various weights; see~\eqref{eq:wg}, \eqref{eq:wg_bulk}, \eqref{eq:vp:g:def}, \eqref{eq:wg_A}, and use $E_{k}, E_{\infty}$ for energy estimates~\eqref{eq:non_E},~\eqref{norm:linf}. We use $\cX^m, \cX^m_A, \cW^m, \cW^m_A, Y_1, Y_2 $ to denote various functional spaces; see~\eqref{norm:Xk}, \eqref{norm:Wk}, and~\eqref{norm:fix}. 

We will use $c, C$ to denote absolute constants that can vary from line to line, and use $\mu, \nu, \e_m, \bar C, \mu_m, \d, \d_Y$ to denote some absolute constants that do not change from line to line. 

Throughout Section~\ref{sec:lin}, we shall use $(\UU, \S)$ to denote the perturbation $(\td \UU, \td \S)$ mentioned earlier in~\eqref{eq:linear+nonlinear}, since for {\em linear stability analysis} there is no ambiguity. In contrast, in the nonlinear stability of Section~\ref{sec:non}, $(\UU, \S)$ is reserved for the full radial velocity and rescaled sound speed, consistently with~\eqref{eq:linear+nonlinear} and~\eqref{eq:euler_ss}.

\section{Linear stability}\label{sec:lin}

In this section, we perform weighted $H^{2m}$ estimates and study the 
asymptotic behavior of the linear operators $\cL_U$, $\cL_A$, and $\cL_\Sigma$ defined in~\eqref{eq:lin}. The  design of the weights is done in Section~\ref{sec:weights}, while the Sobolev energy estimates are performed in Section~\ref{sec:coer_Hm}.

\subsection{Choice of weights} 
\label{sec:weights}

In this section, we design the weight $\vp_{2m}$ needed for weighted $H^{2m}$ estimates. 
\begin{lemma}\label{lem:wg}
There exists a radially symmetric weight $\vp_1(y)$, and there exists a constant $\mu_1 > 0$, such that 
\begin{subequations}
\begin{align}
& \vp_1 \asymp \la \xi \ra, \quad |\na \vp_1| \les 1 , \quad \xi = |y|, \label{eq:wg_asym}  \\
& \f{  (\xi + \bar U)\pa_{\xi} \vp_1 }{\vp_1} 
+ \ell \al \bar \S  \B| \f{ \pa_{\xi} \vp_1 }{\vp_1} \B|  - 
(1 + \pa_{\xi} \bar U - \ell  \al |\pa_{\xi} \bar \S|  )  \leq - \frac{\mu_1}{\la  \xi \ra}  , \label{eq:repul_wg}
\end{align}
\end{subequations}
for all $\xi \in (0,\infty)$ and all $\ell \in \{0, 1\}$.
For $m\geq 1$ and $\vp_1$ satisfying~\eqref{lem:wg}, define 
\begin{equation}
\varphi_m(y) = \varphi_1(y)^m.
\label{eq:phi:m:def}
\end{equation}
\end{lemma}

The quantity on the left hand side of \eqref{eq:repul_wg} appears in the damping term in the weighted $H^{2k}$ estimates. The first term comes from the effect of the transport term. The last two terms come from the derivatives. See \eqref{eq:lin_est1}, \eqref{eq:lin_est2}, \eqref{eq:lin_est3}.

\begin{proof}

Since $\al \bar \S \geq 0$, we only need to prove the estimate for $\ell = 1$.

To design the weights,  we will introduce several parameters and determine them in the following order
\beq\label{eq:wg_para}
R_1 \rightsquigarrow  
R_2 \rightsquigarrow 
c_1 \rightsquigarrow 
R_3 \rightsquigarrow
\kp_2  \rightsquigarrow 
\nu \rightsquigarrow 
\mu_1,
\eeq
with parameters appearing later being allowed to depend on the previous ones. Since we fix the profiles $(\bar U,\bar \Sigma)$ and related parameters (e.g.,~$\kp, \xi_1, \xi_s$ in Lemma \ref{lem:profile}), we drop the dependence of $R_1$ and $R_2$ from \eqref{eq:wg_para} on these parameters.

Recall the parameters $\xi_s, \xi_1$ from Lemma \ref{lem:profile}. We fix $R_1$ with $1 <\xi_s <  R_1 < \xi_1$. From Lemma \ref{lem:profile}, since $\pa_{\xi} \bar U, \pa_{\xi} \bar \S$ decay, we can choose $R_2$ such that for $\xi > R_2$, we have 
\beq\label{eq:repul2}
 \pa_{\xi}(\xi + \bar U ) -   \al |\pa_{\xi}\bar \S |> \tfrac{1}{2}, \quad \xi > R_2.
\eeq
For $\kp_2,  \nu  > 0$ to be determined (see~\eqref{eq:repul3} below), we construct the weights as  
\begin{equation}
\label{eq:wg}
\vp_{m}(y) \teq \vp_1(y)^m, \qquad 
\vp_1 \teq \vp_b^{\kp_2} \vp_f,  \qquad 
\vp_f(y) \teq 1 + \nu  \la y \ra, 
\end{equation}
where $\vp_b$ is chosen in \eqref{eq:wg_bulk} below. Here the sub-indices \textit{b, f} are short for {\em bulk}, and {\em far}, respectively. 
We use the bulk part   to capture the outgoing effect \eqref{eq:rep2} for $ |y|> \xi_s$, and the far part to capture the decay of the perturbation. With  $ \xi_s < R_1 <\xi_1$  and $R_2$ that have already been chosen, there exists $c_1 = c_1(R_1, R_2)>0$ which allows us to define $\vp_b$ as a radially symmetric function with
\begin{subequations}
\label{eq:wg_bulk}
\begin{align}
\vp_b(y) & = 1, \qquad |y| \leq \xi_s, \\  
\vp_b(y) &=  \tfrac 12 , \qquad |y| \geq R_2 + 1,\\
\pa_{\xi} \vp_b &\leq 0, \qquad \forall y \in {\mathbb R}^2,   \\
\pa_{\xi} \vp_b \leq - c_1 & < 0, \qquad |y| \in [R_1, R_2]. 
\end{align}
\end{subequations}

Denote the self-similar slow acoustic wave speed by 
\begin{equation*}
\bar V = \xi + \bar U - \al \bar \S.
\end{equation*}
From \eqref{eq:dec_U}, we can choose $R_3$ large enough such that
\beq\label{eq:wg_R3}
 -  (1 + \pa_{\xi} \bar U - \al |\pa_{\xi} \bar \S | ) \leq -1 + \tfrac{1}{8 } \la \xi \ra^{-1}, 
 \quad \bar V + 2 \al \bar \S \leq \xi + \tfrac{1}{8 }, \quad \forall \xi \geq R_3.
\eeq

Next, we let 
\begin{equation*}
D_{wg} \teq
 \f{  (\xi + \bar U)\pa_{\xi} \vp_1 }{\vp_1} 
+ \al \bar \S  \B| \f{ \pa_{\xi} \vp_1 }{\vp_1} \B|  
=  \f{  \bar V \pa_{\xi} \vp_1 }{\vp_1} 
+ \al \bar \S  \left( \B| \f{ \pa_{\xi} \vp_1 }{\vp_1} \B| +\f{ \pa_{\xi} \vp_1 }{\vp_1} \right) .
\end{equation*}
Using the definition of $\vp_1$ in~\eqref{eq:wg}, the monotonicity properties $\pa_{\xi}\vp_b \leq 0$ and $\pa_{\xi} \vp_f \geq 0$, and the identity  $\f{\pa_{\xi} \vp_1}{\vp_1} = \kp_2 \f{\pa_{\xi} \vp_b}{\vp_b} + \f{\pa_{\xi} \vp_f}{\vp_f}$, we deduce 
\begin{subequations}
\label{eq:D_wg}
\begin{align}
& \B| \f{ \pa_{\xi} \vp_1 }{\vp_1} \B| +\f{ \pa_{\xi} \vp_1 }{\vp_1} 
\leq 2 \f{ \pa_{\xi} \vp_f}{\vp_f} 
= \f{2 \nu \xi}{ \la \xi \ra ( 1  + \nu \la\xi\ra )},  \\
&D_{wg}   \leq 
   \kp_2   \f{ \bar V \pa_{\xi} \vp_b}{\vp_b}  + \f{ \bar V \nu \xi}{ \la \xi \ra ( 1  + \nu \la\xi\ra )} 
+ 2 \al \bar \S  \f{ \nu \xi}{ \la \xi \ra ( 1  + \nu \la\xi\ra )}.
\end{align}
\end{subequations}

Next, our goal is to find positive constants $\kp_2$ (see~\eqref{eq:kappa:2:choice}) and $\nu$ (see~\eqref{eq:nu:choice}), in order to ensure that 
\begin{subequations}
\label{eq:repul3}
\begin{align}
 D_0 &\teq \kp_2   \f{ \bar V \pa_{\xi} \vp_b}{\vp_b} - 
(1 + \pa_{\xi} \bar U - \al |\pa_{\xi} \bar \S| )
 \leq - c < 0, 
 \label{eq:repul3:a} \\
D_1  & \teq D_{wg} -(1 + \pa_{\xi} \bar U - \al |\pa_{\xi} \bar \S| ) \leq D_0 + 
 \f{(\bar V + 2 \al \bar \S) \nu \xi}{ \la \xi \ra ( 1  + \nu \la\xi\ra )}
  \leq - \frac{\mu_1}{ \la \xi \ra},
  \label{eq:repul3:b}
\end{align}
\end{subequations}
hold for some constants $\mu_1= \mu_1(R_1, R_2,R_3, \kp_2, \nu) > 0 $ and $c= c(R_1, R_2)> 0$.
Here we recall that $\bar V =   \xi + \bar U  -  \al \bar \S >0 $ due to~\eqref{eq:rep2}.
Note that the estimate in~\eqref{eq:repul3:b} is exactly the same as~\eqref{eq:repul_wg}.

In order to achieve~\eqref{eq:repul3}, we first note that $\bar V \pa_{\xi} \vp_b \leq 0$ by the definition of $\vp_b$ in \eqref{eq:wg_bulk}.
For $\xi \in [0, \xi_1]$ and $\xi \geq R_2$,  using the repulsive property \eqref{eq:rep1} and inequality \eqref{eq:repul2}, we get 
\[
D_0  \leq - (1 + \pa_{\xi} \bar U - \al |\pa_{\xi} \bar \S| ) \leq \max(- \kp, - \tfrac{1}{2} ) < 0 .
\]
For $\xi \in [\xi_1, R_2]$, from \eqref{eq:rep2} we  get $\bar V > c_2  > 0$ for some $c_2 = c_2(R_2) > 0$, and from \eqref{eq:wg_bulk} we have $\pa_{\xi} \vp_b / \vp_b \leq - c_1(R_1, R_2) < 0$; together, these bounds imply
 \[
 \f{ \bar V \pa_{\xi} \vp_b}{\vp_b} < - c_1 c_2 < 0. 
 \]
Since $| 1 + \pa_{\xi} \bar U - \al |\pa_{\xi} \bar \S| | \leq C$  holds by \eqref{eq:dec_U}, choosing $\kp_2 = \kp_2(R_1, R_2)$ large enough, 
\begin{equation}
 \quad D_0 < - \kp_2 c_1 c_2  + C   = - \tfrac 12 \kp_2 c_1 c_2 \teq - c
 \label{eq:kappa:2:choice}
\end{equation}
for some $c = (R_1, R_2)>0$, as claimed in~\eqref{eq:repul3:a}.

In order to prove~\eqref{eq:repul3:b}, we note that for $\nu \leq 1$, using~\eqref{eq:wg_R3} we get that for all $\xi \geq R_3$ we have
\[
\bal
D_1  & \leq
 -  (1 + \pa_{\xi} \bar U - \al |\pa_{\xi} \bar \S | )
+ \f{(\bar V + 2 \al \bar \S) \nu \xi}{ \la \xi \ra ( 1  + \nu \la\xi\ra )}
\leq -1 + \f{1}{8 \la \xi \ra } + \f{\nu \xi^2}{ \la \xi \ra + \nu  (1 + \xi^2) } + \f{1}{ 
8 \la \xi \ra } \\
& \leq - \f{ \la \xi \ra + \nu }{ \la \xi \ra + \nu (1 + \xi^2) } + \f{1}{4 \la \xi \ra } \leq 
- \f{ \la \xi \ra}{ 2 \la \xi \ra^2} + \f{1}{4 \la \xi \ra } 
= -  \f{1}{4 \la \xi \ra }.
\eal
\]
For $\xi \in [0, R_3]$, using  \eqref{eq:repul3:a} and choosing $\nu = \nu(R_1, R_2, R_3)>0$ small enough, we get 
\begin{equation}	
D_1 \leq D_0 + \nu \|\bar V + 2 \al \bar \S\|_{L^{\infty}[0, R_3]}  \leq -c + 
\nu \|\bar V + 2 \al \bar \S\|_{L^{\infty}[0, R_3]}  < -\tfrac{1}{2} c.
 \label{eq:nu:choice}
\end{equation}
Combining the above two estimates, we establish \eqref{eq:repul3:b} for some $\mu_1(R_1, R_2,R_3, \kp_2,\nu)>0$. This concludes the proof of~\eqref{eq:repul_wg}. 

It remains to establish~\eqref{eq:wg_asym}.
Since $\vp_b$ is constant for $|y| \leq \xi_s$ and $|y| \geq  R_2 + 1$, 
from the definition \eqref{eq:wg}, we get $\vp_{1} \asymp  \vp_f$, $|\na \vp_1| \asymp |\na \vp_f| \les 1 $ in this region; the implicit constants depend on $R_1,R_2,R_3, \nu, \kappa_2$ which have been already fixed. On the compact set $\xi \in [\xi_s,R_2+1]$, the claim \eqref{eq:wg_asym} follows since $\vp_f$ and $\vp_b$ (hence $\vp_1$)  are $C^1$ smooth.
\end{proof}

\subsection{Weighted $H^{2m}$ Coercive estimates}\label{sec:coer_Hm}
Throughout this section, since we are only concerned with {\em linear} stability analysis, we will use $(\UU, \S)$ to denote the perturbation $(\td \UU, \td \S)$; that is, we drop the $\td \cdot $ in \eqref{eq:lin}.

Recall the weights $\vp_m$ defined in \eqref{eq:wg} and recall the definitions \eqref{eq:lin:LU}--\eqref{eq:lin:LS} of the linearized operators $\cL_U,\cL_A$, and $\cL_\Sigma$. Note that the operators
\begin{equation}
\cL \teq (\cL_U, \cL_{\S})
\label{eq:cL:def}
\end{equation} 
and $\cL_A$ are fully decoupled, in the sense that $A$ does not enter in the definition of $\cL$, and $(U,\Sigma)$ do not enter in the definition of $\cL_A$; see~\eqref{eq:lin}. With this notation, we have the following coercive estimates.

\begin{theorem}\label{thm:coer_est}
Denote $\cL = (\cL_U, \cL_{\S})$, with $(\cL_U,\cL_{\S})$ as defined in \eqref{eq:lin}. There exists $m_0 \geq 6$, $R_4, \bar C>0$ large enough and $\lam > 0$ small enough,\footnote{The parameters $m_0,R_4, \bar C$ and $\lambda$  depend only on the weight  $\vp_1$ from Lemma~\ref{lem:wg}, on $\gamma>1$, $r>0$, and on the profiles $(\bar \UU,\bar \Sigma)$.} such that the following statements hold true. For any $m \geq m_0$, there exists $\e_m = \e_m(m_0, R_4, \bar C, \lam)>0$ such that
\begin{subequations}
\begin{align}
 \la   \cL(\UU,  \S), (\UU, \S) \ra_{\cX^m} 
& \leq - \lam \| (\UU, \S)\|_{\cX^m}^2 
+ \bar C \int_{ |y |\leq R_4} (|\UU|^2 + |\S|^2 )  \vp_g d y , 
\label{eq:coer_est}
 \\ 
 \la  \cL_A (\vA), \vA  \ra_{\cX_A^m} 
& \leq - \lam \| \vA \|_{\cX_A^m}^2,   \label{eq:coer_estA}
\end{align}
\end{subequations}
for any $(\UU,\Sigma) \in D(\cL) = \{(\UU,\Sigma) \in \cX^m \colon \cL(\UU,\Sigma) \in \cX^m\} $ and $\vA \in D(\cL_A) = \{\vA \in \cX^m_A \colon \cL_A(\vA)\in \cX^m_A\}$. Here, the Hilbert spaces $\cX^m$ and $\cX_A^m$ are defined as the completion of the space of $C^\infty_c(\R^2)$ {\em radially-symmetric}\footnote{By radially symmetric functions we mean $f(y) = f (|y|)$ and by radially symmetric vectors we mean $\boldsymbol{f}(y) = f(|y|) \ee_{R}$ or $\boldsymbol{f}(y) = f(|y|) \ee_{\theta}$; recall the definition~\eqref{eq:ansatz}.} scalar/vector functions, with respect to the norm induced by the inner products 
\begin{subequations}
\label{norm:Xk}
\begin{align}
& \la f , g \ra_{\cX^m} := \e_m \int \D^m f \cdot \D^m g  \; \vp_{2m}^ 2 \vp_g + f \cdot g \; \vp_g  d y, \ m \geq 1, \quad  \la f , g \ra_{\cX^0} := \int f \cdot g \; \vp_g , \\
& \la f , g \ra_{\cX_A^m} := \e_m \int   \D^m f \cdot \D^m g  \; \vp_{ 2 m}^ 2 \vp_g + f \cdot g \; \vp_A  d y, \ m \geq 1, \quad  \la f , g \ra_{\cX_A^0} := \int f \cdot g \; \vp_A ,
\end{align}
\end{subequations}
where $f, g: \R^2 \to \R^\ell, \ell \in \{1,2\}$.\footnote{It is also convenient to denote $\cX^{\infty} = \cap_{m \geq 0} \cX^{m}$.} The inner products $\la\cdot,\cdot\ra_{\cX^m}$ and $\la\cdot,\cdot\ra_{\cX_A^m}$, and the associated norms, are defined in terms of the constants $\e_m$ (defined in the last paragraph of Section~\ref{sec:est_L}), the weight $\vp_{2m} = \vp_1^{2m}$ defined in~\eqref{eq:phi:m:def}, in terms of the weight
\begin{equation}
\vp_g = \vp_g(y)\teq \la y \ra^{-\kp_1- 2}, \quad \kp_1 =  \tfrac{1}{4},
\label{eq:vp:g:def}
\end{equation}
and in terms of the weight $\vp_A$ constructed in \eqref{eq:wg_A} below. At this stage we note that $\vp_A$ satisfies $\vp_A(y) \asymp |y|^{-\b_1} \la y \ra^{\b_1 -\kp_1- 2 }$ for some $\b_1 \in (3, 4)$, and thus $\vp_g \asymp \vp_A$ for $|y| \geq 1$, and $\vp_g \les \vp_A$ for all $|y|>0$.
\end{theorem}
We emphasize that the constants $\lam, \bar C $ appearing in \eqref{eq:coer_est} are independent of $m$. 
Note also that the norms $\cX^m$ and $\cX_A^m$ only differ by the weighted $L^2$ parts, i.e.~$\|f\|_{\cX^m}^2 -\|f\|_{\cX_A^m}^2 = \|f\|_{\cX^0}^2 - \|f\|_{\cX^0_A}^2$, see~\eqref{norm:Xk}.

Before proving Theorem~\ref{thm:coer_est}, we note the following simple {\em nestedness property} of the spaces $\cX^m$, which follows from Lemmas \ref{lem:interp_wg}, \ref{lem:norm_equiv} with $\d_1 = 1, \d_2 = -2 - \kp_1$. 
\begin{lemma}\label{lem:Xm_chain}
For $n > m$, we have $\|f \|_{\cX^m} \les_n \|f \|_{\cX^n}$ and $\cX^n \subset \cX^m$.
\end{lemma}

Next, we turn to the proof of Theorem~\ref{thm:coer_est}. Since $\cL$ and $\cL_A$ are decoupled (cf.~\eqref{eq:lin} and~\eqref{eq:cL:def}),  we prove \eqref{eq:coer_est} 
and \eqref{eq:coer_estA} separately, in Sections~\ref{sec:est_L} and~\ref{sec:est_LA}, respectively.

\subsubsection{Estimates for $\cL$}\label{sec:est_L}

Applying the operator $\D^m$ to the linearized operators $\cL_U = \cL_U(\UU,\Sigma)$ and $\cL_{\S}= \cL_\S(\UU,\Sigma)$ defined in \eqref{eq:lin:LU}, \eqref{eq:lin:LS}, and using Lemma~\ref{lem:leib} to extract the leading order parts from the terms containing $\na \UU, \na \S$,  
we get 
\begin{subequations}
\label{eq:lin_Hk}
\begin{align}
\D^m \cL_U & = 
\underbrace{- (y + \bar \UU) \cdot \na \D^m \UU - \al \bar \S \na \D^m \S }_{\cT_{U}}
\underbrace{ - (r -1) \D^m \UU- 2m \pa_{\xi}( \xi + \bar U) \D^m \UU- 2m \al \na \bar \S \D^m \S  }_{ \cD_{U} } \notag \\
& \quad \underbrace{- \D^m  \UU \cdot \na \bar \UU 
- \al \D^m \S  \na \bar \S}_{ \cS_U}
+ \cR_{U, m}, \\
\D^m \cL_{\S} & = \underbrace{  - (y + \bar \UU) \cdot \na \D^m \S
- \al \div(\D^m \UU) \bar \S }_{\cT_{\S}}
\underbrace{- (r-1) \D^m \S 
- 2m \pa_{\xi}( \xi + \bar U) \D^m \S - 2 m \al \na \bar \S \cdot \D^m \UU   }_{ \cD_{\S}} \notag \\
& \quad \underbrace{ - \D^m \UU \cdot \na \bar \S 
- \al \div(\bar \UU) \D^m \S }_{\cS_{\S}} + \cR_{\S, m}.
\end{align}
\end{subequations}
In~\eqref{eq:lin_Hk} we have denoted by $\cR_{U, m}$ and $\cR_{\S, m}$ {\em remainder} terms which are of lower order (in terms of highest derivative count on an individual term); moreover, we have used the notation $\cT, \cD,\cS$ to single out {\em transport}, {\em dissipative}, and {\em stretching} terms.

Using Lemma \ref{lem:leib} and the decay estimates \eqref{eq:dec_U}, we obtain that the remainder terms are bounded as  
\begin{subequations}
\label{eq:lin_lower}
\begin{align}
|\cR_{U, m}| 
&\les_m \sum_{ 1 \leq i \leq 2 m-1}  | \na^{ 2 m + 1- i}  \bar \UU | \; | \na^i \UU |  
+ |\na^{2m + 1 - i} \bar \S| \; |\na^i \S| 
\notag\\
&\les_m \sum_{1 \leq i \leq 2m - 1} \la y \ra^{- r - 2m + i} 
( | \na^i \UU |   +  |\na^i \S|  ) , \\
|\cR_{\S, m}| 
&\les_m \sum_{ 1 \leq i \leq 2 m-1}  | \na^{ 2 m + 1- i}  \bar \UU | \; | \na^i \S |  
+ |\na^{2m + 1 - i} \bar \S| \; |\na^i \UU|
\notag\\
&\les_m \sum_{1 \leq i \leq 2m - 1} \la y \ra^{- r - 2m + i} 
( | \na^i \UU |   +  |\na^i \S|  ) .
\end{align}
\end{subequations}

Next, in order to bound the left side of~\eqref{eq:coer_est}, we perform weighted $H^{2m}$ estimates with weight given by $\vp_{2m}^2 \vp_g$, as dictated by the definitions in~\eqref{eq:coer_est}. To this end, we estimate the term
\begin{equation}
 \int ( \D^m \cL_U \cdot \D^m \UU 
 + \D^m \cL_{\S}  \, \D^m \S   ) 
 \vp_{2m}^2 \vp_g,
 \label{eq:main:term:0}
\end{equation}
by appealing to the decomposition in~\eqref{eq:lin_Hk}.

\paragraph{\bf{Estimate for $\cT_U, \cT_{\S}$}}
Using the identity
\beq\label{eq:lin_cross1}
 \na \D^m \S \cdot \D^m \UU + \div(\D^m \UU) \D^m \S = \na \cdot ( \D^m \UU \cdot \D^m \S)
\eeq
and integration by parts, we obtain that the contribution of the transport terms in \eqref{eq:lin_Hk} to the expression~\eqref{eq:main:term:0} is given by 
\[
\bal
I_{\cT} & = -\int \B( (y + \bar \UU) \cdot \na \D^m \UU \cdot \D^m \UU  
+ (y + \bar \UU) \cdot \na \D^m \S \cdot \D^m \S  \\
& \qquad \qquad \qquad \qquad + \al  \bar \S \cdot \na \D^m \S \cdot \D^m \UU
+ \al \bar \S \div( \D^m \UU) \D^m \S 
\B) \vp_{2m}^2 \vp_g \\ 
& = \int \f{1}{2}  \f{ \na \cdot ( (y + \bar \UU) \vp_{2m}^2 \vp_g) }{ \vp_{2m}^2 \vp_g }
( |\D^m \UU|^2 + |\D^m \S|^2  ) \vp_{2m}^2 \vp_g 
+ \f{ \na (\al \bar \S \vp_{2m}^2 \vp_g)}{ \vp_{2m}^2 \vp_g}\cdot \D^m \UU  \D^m \S \vp_{2m}^2 \vp_g  .
\eal
\]

Recall from~\eqref{eq:phi:m:def} and \eqref{eq:vp:g:def} that $\vp_{2m} = \vp_1^{2m}$, and $\vp_g = \la y \ra^{-2-\kp_1}$. Using the decay estimates in \eqref{eq:dec_U}, the outgoing property $\xi + \bar U>0$ \eqref{eq:rep22}, 
and denoting\footnote{The fact that $r<2$ follows since for $d=2$ (the case of this paper) the inequality $r<r_{\sf eye}(\alpha)$ implies $r<2$, see~\eqref{r-range}.}
\begin{equation}
\kp_3 = \min(2, r) = r > 1,
\label{eq:kappa:3:def}
\end{equation}
we obtain
\[
 \f{\pa_{\xi} \vp_g}{\vp_g} = (-2 - \kp_1) \f{\xi}{1 + \xi^2} , \quad 
 (\xi + \bar U)  \f{\pa_{\xi} \vp_g}{\vp_g}
 \leq - 2 - \kappa_1 + C \la \xi \ra^{-2} + C \la \xi \ra^{-r}
 \leq  - 2 + C \la \xi \ra^{-\kappa_3}.
\]
Using the above inequality and~\eqref{eq:dec_U} (with $k=1$),  we estimate
\[
\bal
\f{1}{2}  \f{ \na \cdot ( (y + \bar \UU) \vp_{2m}^2 \vp_g) }{ \vp_{2m}^2 \vp_g }
& = \f{1}{2} \B( 2 + \div(\bar \UU)
+ 4m (\xi + \bar U) \f{\pa_{\xi} \vp_1}{\vp_1}
 + (\xi + \bar U) \f{\pa_{\xi} \vp_g}{\vp_g} \B) \\
 & \leq 
 C \la \xi \ra^{-\kappa_3}+ 2m (\xi + \bar U) \f{ \pa_{\xi} \vp_1}{\vp_1} , \\
\B| \f{ \na( \al \bar \S \vp_{2m}^2 \vp_g)}{ \vp_{2m}^2 \vp_g} \B|
 &  \leq \al |\na \bar \S| + \al \bar \S \Bigl( 4m \f{ |\na \vp_1|}{\vp_1}
 + \f{ |\na \vp_g| }{ \vp_g }\Bigr) 
 \leq 4m \al \bar \S  \f{ |\pa_{\xi}\vp_1|}{\vp_1}
 + C \la \xi \ra^{-\kappa_3} ,
  \eal
\]
where we have used $|\na f| = |\pa_{\xi} f| $ for any radially symmetric function $f$. Combining the above estimates and using $|a b| \leq \f{1}{2}(a^2 + b^2) $
on $\D^m \UU \D^m \S$, we get
\begin{align}
I_{\cT} 
& \leq \int \f{1}{2} \B(  \f{ \na \cdot ( (y + \bar \UU) \vp_{2m}^2 \vp_g) }{ \vp_{2m}^2 \vp_g }
 +  \f{ | \na (\al \bar \S \vp_{2m}^2 \vp_g) | }{ \vp_{2m}^2 \vp_g}
\B) ( |\D^m \UU|^2 + |\D^m \S|^2  ) \vp_{2m}^2 \vp_g  \notag \\
& \leq \int \B(2m \B( (\xi + \bar U) \f{\pa_{\xi} \vp_1}{\vp_1} 
+ \al \bar \S \Bigl| \f{\pa_{\xi} \vp_1}{\vp_1} \Bigr| \B) + C \la \xi \ra^{-\kappa_3} \B) 
( |\D^m \UU|^2 + |\D^m \S|^2  ) \vp_{2m}^2 \vp_g,
\label{eq:lin_est1}
\end{align}
with $C>0$ independent of $m$.

\paragraph{\bf{Estimate of $\cD_U, \cD_{\S}, \cS_U, \cS_{\S}$}}
Recall  the definitions of the terms $\cD_U, \cD_{\S}, \cS_U, \cS_{\S}$ from \eqref{eq:lin_Hk}. Using the decay estimates \eqref{eq:dec_U}, we get 
\[
|\cS_U|, |\cS_{\S}| \les \la \xi \ra^{-r} (|\D^m \UU| + |\D^m \S|) .
\]
For $\cD_U, \cD_{\S}$, using Cauchy-Schwarz inequality for the cross term
\beq\label{eq:lin_cross2}
|\na \bar \S \D^m \S \cdot \D^m \UU | 
+ |\na \bar \S \cdot   \D^m \UU \cdot  \D^m \S | 
\leq |\na \bar \S| ( |\D^m \S|^2 + |\D^m \UU |^2  ),  \quad |\na \bar \S| = |\pa_{\xi} \bar \S|,
\eeq
we obtain 
\begin{align}
I_{\cD+\cS} &= \int \B(   ( \cD_U + \cS_U)    \D^m \UU +  
 ( \cD_{\S} + \cS_{\S})   \D^m \S \B) \vp_{2m}^2 \vp_g \notag \\
& \leq  \int \B( - (r-1)- 2m (1 + \pa_{\xi} \bar U)  
+ 2 m \al |\pa_{\xi} \bar \S| + C \la \xi \ra^{-r} \B) ( |\D^m \UU|^2 + |\D^m \S|^2) \vp_{2m}^2 \vp_g,
\label{eq:lin_est2}
\end{align}
with $C>0$ independent of $m$.

\paragraph{\bf{Estimates of $\cR_U, \cR_{\S}$}}
Recall that the remainder terms $ \cR_U, \cR_{\S}$ from \eqref{eq:lin_Hk} satisfy~\eqref{eq:lin_lower}. Moreover, using that $\vp_{2m} \asymp_m \la y \ra^{2m}$ from~\eqref{eq:wg_asym}, and recalling the definition of $\vp_g$ in~\eqref{eq:vp:g:def}, we obtain 
\[
\vp_{2m}^2 \vp_g \la y \ra^{-r-2m + i} \asymp_m \la \xi \ra^{2m + i - r - \kp_1-2} 
= \la \xi \ra^{2m + i + 2 \d_2},  \quad \d_2 =\tfrac{ -r - \kp_1 - 2}{2}.
\]
At this stage we apply  Lemma~\ref{lem:interp_wg} and Lemma~\ref{lem:norm_equiv}, with $\d_1 = 1$ and $\d_2$ as given in the line above, for $1 \leq i \leq 2m-1$, and an arbitrary $\e > 0$, to obtain 
\begin{align*}
 & \int \la y \ra^{-r - 2m + i} |\na^i F| |\D^m G | \vp_{2m}^2 \vp_g  \\
 &\qquad \leq   \e  \| \la y \ra^{2m + \d_2} \D^m G \|^2_{L^2} 
 + C_{m ,\e } \|\la y \ra^{i + \d_2} \na^i F  \|^2_{L^2} \\
 &\qquad\leq  
 \e \| \la y \ra^{2m + \d_2} \D^m G  \|_{L^2}^2
 + \e \| \la  y\ra^{2m + \d_2} \na^{2 m} F   \|_{L^2}^2
  + C_{m, \e} \| \la y \ra^{\d_2} F \|_{L^2}^2 \\
   &\qquad\leq   2 \e \| \la y \ra^{2m + \d_2} \D^m G  \|^2_{L^2} 
   + 2\e \| \la y \ra^{2m + \d_2} \D^m F    \|^2_{L^2} 
  + C_{m, \e} \| \la y \ra^{ \d_2}  F \|_{L^2}^2 .
\end{align*}
We may apply the above estimates to each term in \eqref{eq:lin_lower}. Using the bound $ \la y \ra^{2( 2m + \d_2)} \les_m \la y \ra^{-r} \vp_{2m}^2 \vp_g $, and using the fact that  $\e > 0$ is arbitrary, we get 
\begin{align}
I_{\cR} &= \int \bigl| \cR_{U, m}  \cdot \D^m \UU + \cR_{\S, m} \cdot \D^m \S\bigr| \vp_{2m}^2 \vp_g \notag \\
& \leq \int \la \xi \ra^{-r} \B(  ( |\D^m \UU|^2 + |\D^m \S|^2 ) \vp_{2m}^2 \vp_g
+ C_{m}  (| \UU |^2 + \S^2)  \vp_g \B).
\label{eq:lin_est3}
\end{align}

Combining the bounds \eqref{eq:lin_est1}, \eqref{eq:lin_est2}, \eqref{eq:lin_est3},  using the estimates in Lemma~\ref{lem:wg}, and recalling~\eqref{eq:kappa:3:def}, 
we arrive at
\begin{align*}
&  \int (\D^m \cL_U \cdot \D^m \UU + \D^m \cL_{\S} \, \D^m \S )   \vp_{2m}^2 \vp_g  
= I_{\cT} + I_{\cD+\cS}  + I_{\cR} \notag  \\
 & \leq \int \left(  - (r-1) 
 + 2m \B( (\xi + \bar U) \f{\pa_{\xi} \vp_1}{\vp_1} 
+ \al \bar \S | \f{\pa_{\xi} \vp_1}{\vp_1} | 
- (1 + \pa_{\xi} \bar U - \al | \pa_{\xi} \bar \S|)\B) + C \la \xi \ra^{-\kappa_3} \right) \notag \\
& \qquad \qquad \times  ( |\D^m \UU|^2 + |\D^m \S|^2) \vp_{2m}^2 \vp_g 
+ C_m \int (|\UU|^2 + |\S|^2 ) \vp_g \notag \\
&  \leq \int \left(- (r-1) - 2m \mu_1 \la \xi \ra^{-1} + a_1 \la \xi \ra^{- \kp_3} \right)  ( |\D^m \UU|^2 + |\D^m \S|^2) \vp_{2m}^2 \vp_g
+ C_m (|\UU|^2 + |\S|^2 ) \vp_g ,
\end{align*}
for some constant $a_1 > 0$. 
Since $\kp_3  = r > 1$ (see~\eqref{eq:kappa:3:def}), there exists $m_0$ sufficiently large, e.g.~$m_0 = \lceil \f{a_1}{2\mu_1} \rceil + 1$, such that for any $ m \geq m_0$, we get
\[
- (r-1) - 2m \mu_1 \la \xi \ra^{-1} + a_1 \la \xi \ra^{- \kp_3} 
\leq - (r-1) - 2m_0 \mu_1 \la \xi \ra^{-1} + a_1 \la \xi \ra^{- \kp_3} 
\leq - (r-1),
\]
resulting in 
\begin{align}
&\int (\D^m \cL_U \cdot \D^m \UU + \D^m \cL_{\S} \, \D^m \S )   \vp_{2m}^2 \vp_g  
\notag\\
&\quad 
\leq - (r-1) \int ( |\D^m \UU|^2 + |\D^m \S|^2) \vp_{2m}^2 \vp_g
+ C_m (|\UU|^2 + |\S|^2 ) \vp_g.
\label{eq:lin_coerHk}
\end{align}

\paragraph{\bf{Weighted $L^2$ estimates} }
For $m = 0$, we do not have the lower order terms $\cR_{U, 0}, \cR_{\S, 0}$ in \eqref{eq:lin_lower} and we do not need to estimate $I_{\cR}$ as in \eqref{eq:lin_est3}. Combining \eqref{eq:lin_est1} and \eqref{eq:lin_est2} (with $m=0$), we obtain 
\begin{equation*}
 \int ( \cL_U \cdot \UU + \cL_{\S} \cdot \S )    \vp_g 
\leq \int (- (r-1) + \bar C \la \xi \ra^{-\kp_3}) ( | \UU|^2 + |\S|^2 ) \vp_g,
\end{equation*}
for some constant $\bar C>0$, independent of $m$.
Since $r>1$ and $\kappa_3>0$, there exists a sufficiently large $R_4$ such that for all $\xi= |y| \geq R_4$ we have
\[
- (r-1) + \bar C \la \xi \ra^{-\kp_3} \leq - (r-1) + \bar C \la R_4 \ra^{-\kp_3} \leq - \tfrac 12 (r-1).
\]
Therefore, combining the two estimates above, we arrive at 
\begin{equation}
\label{eq:lin_coerL2}
 \int ( \cL_U \cdot \UU + \cL_{\S} \cdot \S )    \vp_g 
 \leq  \int \bigl(\bar C\one_{|y| \leq R_4}  - \tfrac 12 (r-1) \bigr)  ( | \UU|^2 + |\S|^2 ) \vp_g .
\end{equation}

\paragraph{\bf{Choosing the $\e_m$} } In order to conclude the proof of~\eqref{eq:coer_est},  we combine \eqref{eq:lin_coerHk} and \eqref{eq:lin_coerL2}. Choosing $\e_m$ sufficiently small, e.g.~$\e_m = \f{r-1}{4 C_m}$, where $C_m$ is as in \eqref{eq:lin_coerHk}, 
multiplying \eqref{eq:lin_coerHk} with $\e_m$ and then adding to~\eqref{eq:lin_coerL2}, we deduce \eqref{eq:coer_est} for $\lambda>0$ sufficiently small which is independent of $m$, e.g.~$\lambda = \frac{r-1}{4}$.

\subsubsection{Estimates for $\cL_{A}$} \label{sec:est_LA}
For $\cL_{A}$, the key step is to establish  the weighted $L^2$ estimate, i.e.~ the bound~\eqref{eq:coer_estA} for $m = 0$.
For a radially symmetric weight $\vp_A$ which is to be chosen, 
using that $\vA \cdot \na \bar \UU =  \vA \f{\bar U}{\xi}$ (see~\eqref{eq:vec_iden}) and integrating by parts, we obtain 
\begin{equation*}
  \int \cL_A(\vA) \cdot \vA \vp_A d y 
  =  - \int \B( \f{1}{2} (y + \bar \UU) \cdot \na |\vA|^2 + \Bigl( r-1 + \f{\bar U}{\xi} \Bigl) |\vA|^2 \B) 
 \vp_A  =  - \int D_A(\vp_A) |\vA|^2 \vp_A,
 \end{equation*}
where we have defined the damping term
\begin{equation*}
   D_A(\vp_A)  \teq
 -  \f{ \na \cdot ( ( y + \bar \UU ) \vp_A ) }{ 2 \vp_A }
 + \f{\bar U}{\xi} + (r- 1)   .
\end{equation*}
 
\paragraph{\bf{Choice of weights}}

Our goal is to construct $\vp_A$ such that the above defined damping term satisfies
\begin{equation}
D_A(\vp_A)(\xi) \geq \td \lam > 0, \qquad \forall \xi>0,
\label{eq:DA:lower:bound}
\end{equation}
for some $\td \lam >0$. We will achieve this by designing $\vp_A$ such that 
\begin{subequations}
\label{eq:wg_A}
\begin{align}
 & \vp_A = \xi^{ -\b_1 } g(\xi)^{\b_2} (1 + \b_3 \la \xi \ra)^{\b_1 - 2 - \kp_1}  , \  \b_1 \in (3, 4),  \b_2\gg 1, 0<\b_3\ll 1, \label{eq:wg_A:a}  \\
 & \vp_A \asymp \xi^{-\b_1} \la \xi \ra^{\b_1 -2-\kp_1} ,  \quad |\na \vp_A| \les \xi^{-1} \vp_A, \label{eq:wg_A:b}
\end{align}
\end{subequations}
for a radial weight $g \asymp 1$ (to be chosen), where $\kp_1$ is chosen in \eqref{eq:vp:g:def}. Since $\vA(0) = 0$, for any sufficiently smooth $\vA$, we have that $|\vA^2| \vp_A$ is locally integrable near $y=0 \in \R^2$ when $\b_1 
< 4$.  We use the singular weight $\xi^{-\b_1}$ to extract the damping effect of the transport operator $( y + \bar \UU) \cdot \na $; a direct calculation yields 
\begin{equation*}
D_A(\vp_A) 
=
- \f{1}{2} \f{ ( \xi + \bar U) \pa_{\xi} \vp_A}{ \vp_A }
- \f{1}{2} (2 + \div(\bar \UU)) +  \f{\bar U}{\xi}
+ (r-1) .
\end{equation*}
Recalling that $\na \cdot \bar \UU = \pa_{\xi} \bar U + \f{\bar U}{\xi}$, and using the formula for $\vp_A$ in \eqref{eq:wg_A}, we further obtain
\[
\bal
D_A(\vp_A)& = \f{\b_1}{2} \f{\xi + \bar U}{\xi}
- \f{\b_2}{2} \f{(\xi + \bar U) \pa_{\xi} g}{g} 
- \f{ \b_1 - 2 - \kp_1}{2} \f{(\xi + \bar U) \b_3 \xi}{ (1 + \b_3 \la \xi \ra) \la \xi \ra}
+ \Bigl( r - 2- \f{1}{2} \pa_{\xi} \bar U + \f{1}{2} \f{\bar U}{\xi} \Bigr)   \\
& \teq I_1(\xi) + I_2(\xi) + I_3(\xi) + I_4(\xi).
\eal
\]
We will choose $4 -\b_1>0,\b_3>0$ small, $\b_2 > 0$ large, and $\pa_{\xi} g \leq 0$ with $g\asymp 1$. Using \eqref{eq:dec_U} and \eqref{eq:rep3},  
\[
\bal
&\lim_{\xi \to 0^+} \bigl( I_1 + I_4\bigr)(\xi) 
= \tfrac{\b_1}{2} (1 + \pa_{\xi} \bar U(0)) + (r-2)
= (2 \pa_{\xi} \bar U(0) + r) - (2 - \tfrac{\b_1}{2}) (1 + \pa_{\xi} \bar U(0) ), \\
& \lim_{\xi \to \infty} \bigl( I_1 + I_4\bigr)(\xi)  = \tfrac{\b_1}{2} + (r- 2) . 
\eal
\]
Using that $r > 1$, $2 \pa_{\xi} \bar U(0) + r > 0$ (see~\eqref{eq:rep3}), and $1 + \pa_{\xi} \bar U(0)>0$ (see~\eqref{eq:rep1}),  we can choose $\b_1 \in (3, 4)$ close to $4$ such that $(2 \pa_{\xi} \bar U(0) + r) - \frac 12 (4 -  \b_1) (1 + \pa_{\xi} \bar U(0) ) \geq 2 c_1$ for some $0 < c_1 \leq \frac 14 $ (which only depends on $\bar U$ and $r$). Therefore, by also appealing to $\pa_{\xi} g \leq 0$, 
$\xi + \bar U >0$ (see~\eqref{eq:rep22}), and to the continuity of $I_i(\xi)$ in $\xi$, there exist $ 0 < p < q$ such that
\begin{equation*}
I_1(\xi) + I_4(\xi) + I_2(\xi) \geq I_1(\xi) + I_4(\xi) \geq c_1 > 0, \quad \xi \in [0, p] \cup 
[q, \infty).
\end{equation*}
Next, we choose a smooth and radially symmetric function $g$ such that 
\[
g(\xi) = 1, \xi \leq \tfrac{p}{2}, \quad  g(\xi)  = \tfrac{1}{2}, \xi \geq 2 q, \quad  \pa_{\xi} g \leq 0, \xi >0 , \quad  
\pa_{\xi} g \leq - c_2,  \quad  \xi \in [p, q].
\]
for some $c_2 > 0$ which only depends on $p$ and $q$. Since \eqref{eq:rep22} implies that $\xi + \bar U > c_3 \one_{[p, q]}(\xi)$  for   $c_3 = p \kappa>0$, choosing $\b_2$ large enough we may ensure that 
\[
I_1(\xi) + I_2(\xi) + I_4(\xi) \geq  \tfrac{\b_2}{2} c_2 c_3 - \max_{\xi \in[p, q]}( |I_1(\xi)| + |I_4(\xi)| )  > c_1 > 0,
\quad \xi \in [p, q].
\] 
Thus, we have shown that by choosing the function $g$ and the parameters $\beta_1$ and $\beta_2$, we may ensure that $I_1(\xi) + I_2(\xi) + I_4(\xi) > c_1>0$ for all $\xi>0$, for some $c_1>0$. 

Finally, we handle $I_3$. Note that the remaining free parameter in the definition \eqref{eq:wg_A} of $\vp_A$ is $\beta_3$, which only enters the definition of $I_3$. Since $I_2 \geq 0$ and $\b_1 - 2 - \kp_1 > 0$, using the asymptotics \eqref{eq:dec_U}, the bound $ \f{\b_3 \xi}{(1 + \b_3 \la \xi \ra) \la \xi \ra} \leq \frac{1}{\xi}$, for $\xi \gg 1$ we have 
\[
D_A \geq I_1 + I_3 + I_4 
\geq \Bigl( \f{\b_1}{2} - \f{\b_1 -\kp_1 - 2}{2} \Bigr) \f{\xi + \bar U}{\xi} 
 + (r- 2) - C \la \xi \ra^{-r}
 \geq  \f{\kp_1+2}{2}  + (r- 2) - C \la \xi \ra^{-r}.
\]
By also appealing to the previously established lower bound  for $I_1 + I_2 + I_4$, for $\xi  \les 1$ we have
\[
D_A \geq  I_1 + I_2 + I_4 - C \b_3 \xi \geq c_1 - C \b_3 \xi .
\]
Since $\kp_1> 0$ and $r-1 > 0$, the first of the two lower bounds for $D_A$ above 
imply that there exists $q_2>0$ sufficiently large such that $D_A \geq \frac{\kappa_1}{4} = \frac{1}{16}$ for all $\xi \geq q_2$. On the other hand, taking $\b_3>0$ small enough to ensure that $\frac{c_1}{2} \geq C \b_3 q_2$, from the second of the two lower bounds for $D_A$ above, we obtain $D_A \geq \frac{c_1}{2}$ for all $\xi \in (0,q_2)$. Together, these bounds establish \eqref{eq:DA:lower:bound}, for $\td \lam = \min\{\frac{1}{16},\frac{c_1}{2}\}>0$.

We also note that the properties of $\vp_A$ stated in \eqref{eq:wg_A:b} follow from the definitions of $g, \vp_A$.  

Having established~\eqref{eq:DA:lower:bound}, we obtain the desired weighted $L^2$ stability estimate
\begin{equation}
\label{eq:coer_L2}
\int \cL_A (\vA) \cdot \vA \vp_A \leq - \td \lam \int |\vA|^2 \vp_A.
\end{equation}

\paragraph{\bf{Weighted $H^{2m}$ estimates}}
The weighted $H^{2m}$ estimates with weights $\vp_{2m}^2 \vp_g$ are similar to those for $\cL$ in Section \ref{sec:est_L}. 
In analogy to the derivation in \eqref{eq:lin_Hk}, we may apply $\Delta^m$ to $\cL_A$ (as defined in~\eqref{eq:lin:LA}) and then decompose the resulting expression into a transport, damping, stretching, and remainder term; to avoid redundancy we do not spell out these details. Analogously to the bounds \eqref{eq:lin_est1}, 
\eqref{eq:lin_est2}, \eqref{eq:lin_est3}, and recalling~\eqref{eq:kappa:3:def}, we obtain 
\begin{equation*}
\int \D^m \cL_A (\vA) \cdot \D^m \vA \vp_{2m}^2 \vp_g d y
  \leq \int D_{A, 2m} | \D^m \vA|^2 \vp_{2m}^2 \vp_g
+ C_m \int |\vA|^2 \vp_g ,
\end{equation*}
where
\begin{equation*}
D_{A, 2m}  =   2m  (\xi + \bar U) \f{\pa_{\xi} \vp_1}{\vp_1} 
 + C \la \xi \ra^{-\kappa_3} 
 - (r-1) -  2m (1 + \pa_{\xi} \bar U)  
.
\end{equation*}

Compared to \eqref{eq:lin_est1}, \eqref{eq:lin_est2}, in the above estimates, we do not need to estimate cross terms \eqref{eq:lin_cross1}, \eqref{eq:lin_cross2}, which contribute to the bounds $\al \bar \S | \f{\pa_{\xi} \vp_1}{\vp_1} |$ in \eqref{eq:lin_est1} and $2 m \al |\pa_{\xi} \bar \S|$ in \eqref{eq:lin_est2}. 
Using Lemma~\ref{lem:wg} with $\ell = 0$, and choosing $\td m_0$ large enough, for $m \geq \td m_0$, we obtain 
\[ 
 D_{A, 2m} \leq  - (r-1) - 2m \mu_1  \la \xi \ra^{-1} +  C \la \xi \ra^{-\kappa_3}
 \leq - (r-1),
\]
uniformly in $\xi>0$. Since $\vp_g \les \vp_A$, combining the above estimates and \eqref{eq:coer_L2}, and then choosing $\e_{m}$ small enough, we obtain
\begin{align*}
 \int \Bigl(\e_m  \D^m \cL_A (\vA) \cdot \D^m \vA \vp_{2m}^2  \vp_g
+ \cL_A (\vA) \cdot \vA \vp_A\Bigr)
  d y \leq - \td \lam_1 \int \Bigl( \e_m |\D^m \vA |^2 \vp_{2m}^2 \vp_g
  + |\vA|^2 \vp_A \Bigr) dy 
\end{align*}
for some $\td \lam_1 > 0$ independent of $m$.
This proves \eqref{eq:coer_estA}. Clearly, by choosing $\lam$ smaller, $m_0$ larger, and $\e_m$ smaller, we can obtain the coercive estimates \eqref{eq:coer_est}, \eqref{eq:coer_estA} for the same set of parameters. This concludes the proof of Theorem~\ref{thm:coer_est}.

\subsection{Compact perturbation via Riesz representation}
\label{sec:Riesz}
In this section, using the estimates established earlier in Theorem~\ref{thm:coer_est}, we construct a compact operator $\cK_m$ such that $\cL - \cK_m$ is dissipative in $\cX^m$. For this purpose we fix $m_0\geq 6$. We also recall from Lemma~\ref{lem:Xm_chain} that the Hilbert spaces $\{\cX^i\}_{i \geq 0}$ are nested, with $\cX^{i+1} \subset \cX^i \subset \ldots \subset \cX^0$ for all $i\geq 0$; we will implicitly use this fact throughout this section.

\begin{proposition}\label{prop:compact}
For any $m \geq m_0 $, there exists a bounded linear operator $\cK_m \colon \cX^0 \to \cX^m$ with:
\begin{itemize}
\item[(a)] 
for any $f \in \cX^{0}$ we have
\[
{\rm supp}(\cK_m f) \subset B(0, 4 R_4),
\]
where $R_4$ is chosen in Theorem~\ref{thm:coer_est} (in particular, it is independent of $m$);
\item[(b)] the operator $\cK_m$ is compact from $ \cX^{m} \to \cX^m$;  
\item[(c)] the enhanced smoothing property  $\cK_m: \cX^{0} \to \cX^{m+ 3}$ holds;
\item[(d)] the operator $\cL - \cK_m$ is dissipative on $\cX^m$ and we have the estimate
\begin{equation}
 \la (\cL - \cK_m) f ,f \ra_{\cX^m} \leq - \lam \| f \|_{\cX^m}^2
 \label{eq:dissip}
\end{equation}
for all $f \in \{ (\UU,\Sigma) \in \cX^m \colon \cL(\UU,\Sigma) \in \cX^m\}$, $\cL = (\cL_U,\cL_S)$, and where $\lambda >0$ is the  parameter from~\eqref{eq:coer_est} (in particular, it is independent of $m$).
\end{itemize}
\end{proposition}

Item~(d) shows that $\cL$ is a compact perturbation of the dissipative operator $\cD_m \teq  \cL - \cK_m$ in $\cX^m$, for any $m \geq  m_0 $. 
By further regularizing $\cK_m$ one could define $\cK_m : \cX^0 \to \cX^{\infty}=\cap_{m \geq 0} \cX^{m}$, but we do not need this extra regularity and thus do not seek this refinement; we only need the $\cX^{m +3}$-regularity claimed in item~(c) in order to simplify the nonlinear estimates in Section~\ref{sec:fix_pt}.

\begin{proof}
Fix $R_4$ as in Theorem~\ref{thm:coer_est}. We let $0 \leq \chi_1, \chi_2 \leq 1$ be smooth cutoff functions  such that 
\begin{subequations}
\label{eq:cutoff}
\begin{align}
& \chi_1(y) = 1, \   |y| \leq  2  R_4 , \quad \chi_1(y) = 0,  \ |y| \geq 4 R_4 ,  \\
& \chi_2(y) = 1, \ |y| \leq 4 R_4,  \quad \chi_2(y) = 0, \ |y| > 8 R_4 .
\end{align}
\end{subequations}
Throughout this proof we also fix $m\geq m_0$, where $m_0\geq 6$ is given by Theorem~\ref{thm:coer_est}.

\paragraph{\bf{Riesz representation}}
We consider the bilinear form $B_m(\cdot,\cdot) \colon \cX^0 \times \cX^m \to {\mathbb R}$ defined by 
\[
B_m(f , g) \teq    {\sum}_{1\leq i\leq 3}
 \int \chi_1  f_i \,  g_i,  
\]
for $f = (\UU, \S) = (U_1,U_2,\S) \in \cX^0$ and $g = ( \UU^\prime, \S^\prime)=(U^\prime_1,U^\prime_2,\S^\prime) \in \cX^m$. 

We note that for $\vp_g$ as defined in \eqref{eq:vp:g:def}, we have that $\langle 4 R_4\rangle^{-\kappa_1-2} \leq \vp_g \leq 1$ on ${\rm supp}(\chi_1)$ and hence 
by also recalling the nested property $\cX^m \subset \cX^0$ (cf.~Lemma~\ref{lem:Xm_chain}) we have the trivial estimate
\begin{equation}
\label{eq:bilin}
|B_m(f, g) | 
\les 
\| f \|_{\cX^0} \| g \|_{\cX^0}
\les_m 
\|f \|_{\cX^{0}} \| g\|_{\cX^{m }} . 
\end{equation}
Here and throughout the proof the implicit constants in $\les$ are allowed to depend on $R_4$. It is clear from the above bound that $B_m(\cdot,\cdot)$ is well defined on the larger domain $\cX^0 \times \cX^0$, but we shall not use this fact.

For each fixed $f \in \cX^{0}$, since $B_m(f, \cdot)$ is a bounded linear functional on the Hilbert space $\cX^m$, using the Riesz representation theorem, there exists a unique $\cK_m f \in \cX^m$ such that 
\[
\la \cK_m f , g \ra_{\cX^m} = B_m(f, g) ,\quad \forall g \in \cX^m. 
\]
Using \eqref{eq:bilin} and duality, we immediately obtain
\beq\label{eq:bilin2}
\| \cK_m f \|_{ \cX^m} = \sup_{ \| g \|_{ \cX^m } \leq 1} \la \cK_m f , g \ra_{ \cX^m}
= \sup_{ \| g \|_{ \cX^m} \leq 1} B_m(f, g) \les_m \| f \|_{\cX^0} .
\eeq
Since $B_m(f, g)$ is linear in $f$, from~\eqref{eq:bilin2} we deduce that $\cK_m$ is a bounded linear operator from $\cX^{0}$ to $\cX^m$.

\paragraph{\bf{Proof of item (a)}}
For any $f \in \cX^0$, from the definitions of $\chi_1$ and $B_m$, we deduce that $ \la \cK_m f , g \ra_{ \cX^m } = 0$ for all $g\in \cX^m$ such that $\supp(g) \subseteq  \{ y : |y | \geq 4 R_4 \} $. By the definition of the inner-product on $\cX^m$ (it is equivalent to the $H^{2m}$ inner-product on compact subsets), it follows that $\supp(\cK_m f ) \subset B(0, 4 R_4 )$. Moreover, with $\chi_2$ as defined in~\eqref{eq:cutoff} we have $B_m(f,\cdot) = B_m(\chi_2 f,\cdot)$ and hence $\cK_m f  = \cK_m(\chi_2 f)$. With these two facts, it also follows that  $\chi_2 \cK_m f  = \cK_m f $. 

\paragraph{\bf{Proof of item (b)}}
Since $\vp_{2m} \asymp_{m,\Omega} 1$ and $\vp_g \asymp_\Omega 1$ on any compact set $\Omega\subset \mathbb{R}^2$ (in particular, for $\Omega = \overline{B(0, 4 R_4 )}$), it follows that for any function $h$ such that ${\rm supp}(h) \subset \Omega$, we have $\| h \|_{\cX^m} \asymp_{m,\Omega} \| h \|_{H^{2m}(\Omega)}$; this norm equivalence will be used several times throughout the remainder of the proof.

To establish the compactness of $\cK_m \colon \cX^m \to \cX^m$, we note that if $\{f_k\}_{k\geq 1}$ is a bounded sequence in $\cX^m$, then using  $\|\chi_2 f \|_{ \cX^m } \les \| f \|_{ \cX^m}$ for any $f \in \cX^m$, and appealing to the above discussed norm equivalence, we obtain that $\{\chi_2 f_k\}_{k\geq 1}$ is a bounded sequence of functions in $H^{2m}(B(0,8 R_4))$. Using the Rellich-Kondrachov compact embedding theorem, there exists a subsequence $\{\chi_2 f_{k_j}\}_{j\geq 1}$ which is convergent in $L^2(B(0,8 R_4))$, and a-posteriori in $\cX^0$. By the continuity of $\cK_m$ it follows that $\{\cK_m(\chi_2 f_{k_j})\}_{j\geq 1}$  is convergent in $\cX^m$, and since $\cK_m(\chi_2 f_{k_j}) = \cK_m(f_{k_j})$ we deduce the desired pre-compactness of the sequence $\{\cK_m(f_k)\}_{k\geq 1}$ in $\cX^m$. We obtain that $ \cK$ is a bounded compact operator from $\cX^m \to \cX^m$. 

\paragraph{\bf{Proof of item (c)}}
We fix $f \in \cX^0$. Using the definition of $\cX^m$ in~\eqref{norm:Xk}, the fact that  $\vp_{2m} \asymp_{m} 1$ and $\vp_g \asymp  1$ on $B(0,4R_4)$, the middle inequality in~\eqref{eq:bilin}, the inclusion in Lemma~\ref{lem:Xm_chain}, the bound~\eqref{eq:bilin2}, and the compact support property of $\cK_m$ established in item~(a), we deduce
\begin{align}
\Bigl|\int \D^m ( \cK_m f ) \cdot \D^m g  \Bigr|
&\les_m  \bigl|\la \cK_m f, g \ra_{\cX^m} \big| + \bigl|\la \cK_m f, g \ra_{\cX^0} \big|\notag\\
&\les_m \|f\|_{\cX^0} \|g\|_{\cX^0} + \| \cK_m f\|_{\cX^0} \| g \|_{\cX^0}
\notag\\
&\les_m \|f\|_{\cX^0} \|g\|_{\cX^0} + \| \cK_m f\|_{\cX^m} \| g \|_{\cX^0}
\les_m \| f \|_{\cX^0} \| g \|_{\cX^0}. 
\label{eq:K_gain}
\end{align}
for any $f\in \cX^0$ and $g\in \cX^m$.

We would like to set $g=\Delta^m(\cK_m f)$ in \eqref{eq:K_gain}, but this is not immediately possible since we only know $\cK_m f \in \cX^{m}$, which provides insufficient regularity. As such, we need to regularize $\cK_m f$. 
We choose a radially symmetric non-negative function $\phi \in C_c^{\infty}$ with $\int \phi = 1$ and define $\phi_{\e}(y)= \e^{-2} \phi(y / \e)$.
Let 
\[
F \teq \Delta^m(\cK_m f),
\quad G_{\e} \teq \phi_{\e} \ast( \phi_{\e} \ast F). 
\]
Using the equivalence of the norms $\|\cdot\|_{\cX^m}$ and $\|\cdot \|_{H^{2m}}$ for functions which have compact support in $B(0,4R_4)$, from the fact that $\cK_m f \in \cX^m$ and item~(a) we deduce that $F \in L^2$ has compact support, and $G_{\e} \in C^\infty_c \subset \cX^m$. Inserting $g\mapsto G_{\e} $ in \eqref{eq:K_gain}, upon integrating by parts we deduce that 
\begin{align*}
\|\na^m (\phi_{\e} \ast F)\|_{L^2}^2
&\leq_m \Bigl|\int F \cdot \D^m ( \phi_{\e} \ast( \phi_{\e} \ast F) ) \Bigr|  \notag\\
&\les_m \| f \|_{\cX^0} \| G_{\e} \|_{\cX^0}
\les_m \| f \|_{\cX^0} \|   F\|_{L^2}
\les_m \| f \|_{\cX^0} \| \cK_m f\|_{\cX^m}
\les_m \| f \|_{\cX^0}^2.
\end{align*}
Passing $\e \to 0$, we deduce that $\na^m F\in L^2$ and $\|\na^m F\|_{L^2} \les_m \| f\|_{\cX^0}$.  
Since $\cK_m f$ has compact support, we thus obtain 
\[
\|\cK_m f\|_{\cX^{ \lfloor 3m/2 \rfloor}} 
\les_m
\| \cK_m f\|_{\dot{H}^{3m}} 
\les_m
\int |\na^m  \D^m (\cK_m f)|   
=
\| \na^m F\|_{L^2}
\les_m  \| f \|_{\cX^0}
.
\]
Since $ m \geq 6$, using the nested property from Lemma~\ref{lem:Xm_chain} we deduce the boundedness $\cK_m : \cX^{0} \to \cX^{m + 3}$. 

\paragraph{\bf{Proof of item (d)}}
Using the definitions of $B_m $ and $\cK_m$, the estimate~\eqref{eq:coer_est}, recalling that $\vp_g \leq 1$, and letting $\bar C > 0$ be exactly as in~\eqref{eq:coer_est}, we deduce
\[
\la (\cL - \bar C \cK_m)f , f \ra_{ \cX^m} \leq -\lam \la f, f \ra_{\cX^m} +  \bar C  \int 
(\one_{|y| \leq R_4 } - \chi_1)   |  f  |^2  dx 
\leq -\lam \la f, f \ra_{ \cX^m},
\]
where $f = (\UU, \S), \cL = (\cL_U, \cL_S)$. Upon {\em re-defining} $\cK_m$ to equal $\bar C \cK_m$, which does not alter the proofs of items~(a)--(c), we deduce the dissippativity claimed in~\eqref{eq:dissip}.
\end{proof}

\subsection{Construction of the semigroup}\label{sec:semi}

For any $m \geq m_0$ we construct the strongly continuous semigroups 
\[
e^{\cL s} : \cX^m \to  \cX^m ,
\quad e^{\cD_m s} : \cX^m \to  \cX^m ,  \quad \cL = (\cL_U, \cL_{\S}) =  \cD_m + \cK_m.
\]
by directly solving the associated linear PDEs. Here $\cK_m$ is as constructed in Proposition~\ref{prop:compact}. 
In this section, without loss of generality we consider semigroups defined for $s\geq 0$, although in applications for Section~\ref{sec:non} the interval $s\geq \sin$ is the relevant one.

\subsubsection{From linear PDEs to semigroups}

The relation between linear evolutionary  PDEs and semigroups is classical; see e.g.~\cite{ElNa2000}. Following Definition 6.1, Chapter II of~\cite{ElNa2000}, we consider the abstract initial value problem for a linear operator $\cA : D(\cA) \subset X \to X$ with domain $D(\cA)$: 
\beq\label{eq:ACP}
 \tfrac{d}{dt} u(t) = \cA u(t) ,  \quad t \geq 0, \quad u(0) = x \in X,
\eeq
where $X$ is a Banach space. A function $u: {\mathbb R}_+ \to X$ is a (classical) solution of \eqref{eq:ACP} if $u$ is continuously differentiable with respect to $X$, $u(t) \in D(\cA)$ for all $t \geq 0$, and \eqref{eq:ACP} holds in $X$ for all $t>0$. We use the notation $u(x,t)$ to denote a solution $u$ from initial data $x \in X$, at time $t$. We recall:

\begin{theorem}[Theorem 6.7, Chapter II,~\cite{ElNa2000}]\label{thm:ACP}
Let $\cA: D(\cA) \subset X \to X$ be a closed operator. The following properties are equivalent:
\begin{itemize}
\item[(a)] $\cA$ generates a strongly continuous semigroup.
\item[(b)] For every $x \in D(\cA)$, there exists a unique solution $u(x,\cdot)$ of the initial value problem \eqref{eq:ACP}. The operator $\cA$ has dense domain, and for every sequence $\{ x_n \}_{n \geq 1} \subset D(\cA)$ satisfying $\lim_{n\to \infty} x_n = 0$, one has $\lim_{n \to \infty} u(x_n,t) = 0$ uniformly on compact intervals $[0, t_0]$. 
\end{itemize}
\end{theorem}

The above theorem states that $\cA$ generates a strongly continuous semi-group if and only if the linear PDEs is well-posed: \eqref{eq:ACP} has a unique solution, and satisfies continuous dependence on the data.

\subsubsection{Solving linear PDEs}\label{sec:lin_semi}

We define the domains of the operators $\cL = (\cL_U,\cL_\Sigma)$ and $\cD_m = \cL - \cK_m$ as 
\beq\label{eq:domain}
D (\cD_m) = D (\cL) \teq \{  (\UU, \S) \in \cX^m,  \cL (\UU, \S) \in \cX^m \}. 
\eeq

Recall from~\eqref{eq:lin:LU} and \eqref{eq:lin:LS} that the linearized operator $\cL= (\cL_U, \cL_{\S})$ is defined via
\begin{subequations}
\label{eq:lin_semi}
\begin{align}
 \cL_U(\UU, \S) &=  -  (r-1)  \UU -  (y +\bar{ \UU}) \cdot \na  \UU  -  \UU \cdot \na \bar{\UU} 
 - \alpha \S \na \bar{\S}  - \alpha \bar{\S} \na \S, \\
 \cL_{\S}(\UU, \S) &= -  (r-1) \S - (y +\bar{ \UU}) \cdot \na  \S
 -   \UU  \cdot \na \bar{\S}     -  \alpha \div(\bar \UU)
 \S - \alpha \div( \UU) \bar{\S},
 \end{align}
 resulting in the linearized versions of the evolution equations \eqref{eq:lin:a} and \eqref{eq:lin:c}
\begin{align}
 \pa_s (\UU,\S)  &= \cL(\UU, \S) = \bigl(\cL_U(\UU, \S),\cL_\S(\UU, \S) \bigr)
 \label{eq:lin_semi:c}.
\end{align}
\end{subequations}
Note that $\cL$ does not involve $A$, so that \eqref{eq:lin_semi} is a closed evolution for $(\UU, \S)$, which is why we drop the dependence on $A$.
We also emphasize that the linear operator $\cL$ preserves radial symmetry (recall~\eqref{eq:vec_iden}), and so for initial data in $D(\cL)\subset \cX^m$, we may expect the solution of \eqref{eq:lin_semi:c} to lie in $\cX^m$.

Denote $\VV = (\UU,\Sigma) = (U_1, U_2, \S) \subset {\mathbb R}^3$. The linearized equations~\eqref{eq:lin_semi:c} may be written in the form of  a symmetric hyperbolic system 
\beq\label{eq:lin_semi_V}
 \pa_s \VV +  B_1(y) \pa_1 \VV + B_2(y) \pa_2 \VV = M(y) \, \VV, 
 \quad B_i(y) = (y_i + \bar U_i(y)) {\rm Id} + \al \bar \S(y)  ( E_{i3} + E_{3i} ) ,
\eeq
where $E_{ij} \in \R^{3\times 3}$ is the matrix with $1$ at the $(i, j)$ entry and $0$s at all other entries, and the matrix $M\in \R^{3\times 3}$ depends only on the profiles $(\bar \UU,\bar \Sigma)$. In particular, since the profiles $(\bar \UU,\bar \Sigma)$ are smooth, so are the coefficients $B_i, M$. For initial data $\VV_{0} = \VV|_{s=0} \in C_c^{\infty}({\mathbb R}^2)$ the symmetric hyperbolic system~\eqref{eq:lin_semi_V}
has a unique, global-in-time, smooth solution $\VV$, and $\VV(\cdot,s)  \in C_c^{\infty}({\mathbb R}^2)$; this fact may for instance be established using the vanishing viscosity method (see e.g.~\cite[Section 7.3.2]{Ev2022}), using the fact that the $B_i$ and $M$ are bounded with bounded derivatives on any compact domain, and using finite speed of propagation.\footnote{A much more general existence and uniqueness theorem is established in~\cite{Ka1975}.}

Next, we use approximation of the initial data to show that for $\VV_0 \in \cX^m$ the solution $\VV$ of~\eqref{eq:lin_semi_V} satisfies $\VV(\cdot,s) \in \cX^m$ for all $s>0$.
For any $\VV_0 \in \cX^m$, there exists a radially symmetric $\VV_0^{(n)} \in C_c^{\infty}$ such that $\|\VV_0^{(n)} - \VV_0\|_{\cX^m} \to 0$ as $n\to \infty$. For initial datum $\VV_0^{(n)}$, we obtain a unique, global-in-time, radially symmetric smooth solution $\VV^{(n)}$ such that $\VV^{(n)}(\cdot,s) \in C_c^{\infty} \subset \cX^m$ for $s>0$. Using that the evolution~\eqref{eq:lin_semi_V} is linear, we may apply the energy estimates in Theorem~\ref{thm:coer_est} (more precisely, the bound~\eqref{eq:coer_est}) to both $\VV^{(n)} -\VV^{(n^\prime)}$ and $\VV^{(n)}$, to obtain 
\begin{align*}
  \tfrac{1}{2}\tfrac{d}{ds} \| \VV^{(n)} -\VV^{(n^\prime)}\|_{\cX^m}^2
 &\leq - \lambda \| \VV^{(n)} -\VV^{(n^\prime)}\|_{\cX^m}^2 + \bar C \| \VV^{(n)} -\VV^{(n^\prime)}\|_{\cX^0}^2 
 \leq \bar C \| \VV^{(n)} -\VV^{(n^\prime)}\|_{\cX^m}^2 
 \\ 
 \tfrac{1}{2}\tfrac{d}{ds} \| \VV^{(n)} \|_{\cX^m}^2
 &\leq - \lambda \| \VV^{(n)}\|_{\cX^m}^2 + \bar C \| \VV^{(n)}\|_{\cX^0}^2 
 \leq \bar C \| \VV^{(n)} \|_{\cX^m}^2 
 \end{align*}
 to deduce
 \begin{equation}
 \label{eq:lin_semi_EE}
 \|(\VV^{(n)} -\VV^{(n^\prime)})(\cdot,s)\|_{\cX^m} \leq e^{\bar C s} \| \VV^{(n)}_0 -\VV^{(n^\prime)}_0\|_{\cX^m},
 \qquad 
 \|\VV^{(n)}(\cdot,s)\|_{\cX^m} \leq e^{\bar C s} \| \VV^{(n)}_0\|_{\cX^m}
 .
\end{equation}
Since $\{\VV^{(n)}_0\}_{n \geq 1}$ is a Cauchy sequence in $\cX^m$, from~\eqref{eq:lin_semi_EE} we deduce
$\VV^{(n)}(\cdot,s) \to \VV(\cdot,s)$ in $\cX^m$ uniformly for $s \in [0,T]$, for any $T > 0$. Moreover,~\eqref{eq:lin_semi_EE} yields $\| \VV(\cdot,s) \|_{\cX^m} \leq e^{\bar C s} \| \VV_0 \|_{\cX^m}$. 

In order to show that  $\|\VV(\cdot,s) - \VV_0\|_{\cX^m} \to 0$ as $s \to 0$, for $n\geq 1$ we use the decomposition 
\begin{align}
\|\VV(\cdot,s) - \VV_0\|_{\cX^m} 
& \leq  \| \VV^{(n)}(\cdot,s) - \VV^{(n)}_0\|_{\cX^m} 
+ \|\VV^{(n)}(\cdot,s) - \VV(\cdot,s) \|_{\cX^m } + \| \VV^{(n)}_0 - \VV_0 \|_{\cX^m} 
\notag \\
& \leq \| \VV^{(n)}(\cdot,s) - \VV^{(n)}_0\|_{\cX^m} + (e^{\bar Cs} + 1)  \| \VV^{(n)}_0 - \VV_0 \|_{\cX^m}.
\label{eq:vomit:1}
\end{align}
We thus only need to show that for any {\em fixed}  $n\geq 1$, we have $\VV^{(n)}(\cdot,s) \to \VV^{(n)}_0$ in $\cX^m$ as $s \to 0$. Since we already know that $\VV^{(n)}(\cdot,s) \in C_c^{\infty}({\mathbb R}^2)$, with uniform compact support for $s \in [0,1]$, we deduce that $\cL \VV^{(n)} \in L^\infty([0,1]; \cX^m)$, and so $\pa_s \VV^{(n)} \in L^\infty((0,1); \cX^m)$. Thus, $\lim_{s\to 0} \|\VV^{(n)}(\cdot,s) - \VV^{(n)}_0\|_{\cX^m} = 0$. Finally, we first pass $s\to 0$, and then $n\to \infty$ in \eqref{eq:vomit:1}, to deduce $\lim_{s\to 0} \|\VV(\cdot,s) - \VV_0\|_{\cX^m} = 0$.

From the energy estimates in Theorem~\ref{thm:coer_est} and estimates similar to~\eqref{eq:lin_semi_EE}, we also obtain the  uniqueness of solutions to \eqref{eq:lin_semi_V} (hence~\eqref{eq:lin_semi:c}),  from any initial data in $\cX^m$. Therefore, for any $\VV_0 \in \cX^m$, we have constructed the solution operator $S_{\cL}(s)$ for  
\eqref{eq:lin_semi} satisfying 
\begin{subequations}
\label{eq:growbd}
\begin{align}
& S_{\cL}(s) : \cX^m \to \cX^m, \quad S_{\cL}(s) S_{\cL}(t) = S_{\cL}(s + t), \quad s, t \geq 0, \\
  & \| S_{\cL}(s)\|_{\cX^m \to \cX^m} \leq e^{\bar C s} , \quad \lim_{s \to 0} \| S_{\cL}(s) \VV_0 - \VV_0\|_{\cX^m} = 0.
\end{align}
\end{subequations}
The semigroup property follows from the uniqueness of the solution. The growth rate $\bar C$ is as in~\eqref{eq:coer_est}.

In order to apply Theorem~\ref{thm:ACP}, we need to verify condition (b). Consider $\VV_0 \in D(\cL)$, as defined in~\eqref{eq:domain}. Unique classical solvability requires that the solution $\VV$ of \eqref{eq:lin_semi} with initial data $\VV_0$, i.e.~$\VV(\cdot,s) = S_{\cL}(s) \VV_0$, satisfies $\VV(\cdot,s) \in D(\cL)$ for $s>0$. In order to prove this we observe that $[\pa_s,\cL]=0$ and hence
\[
\pa_s (\pa_s \VV) = \cL (\pa_s \VV), \quad  (\pa_s \VV) |_{s = 0} = \cL \VV_0 \in \cX^m. 
\]
Applying \eqref{eq:growbd} to $\pa_s \VV$, we obtain $\cL \VV(\cdot,s) =  \pa_s \VV (\cdot,s)= S_{\cL}(s) (\cL \VV_0) \in \cX^m$, for all $s>0$, as desired. This argument also shows that $\VV \in {\rm Lip}([0, T], \cX^m)$. 
The continuous dependence of data required in item~(b) of Theorem~\ref{thm:ACP} follows directly from the growth bound of $S_{\cL}(s)$, see~\eqref{eq:growbd}. 

Concluding, we apply Theorem~\ref{thm:ACP} with $(\cA, X) = (\cL, \cX^m)$, and obtain that $\cL$ generates a strongly continuous semigroup $S_{\cL}(s) = e^{\cL s}$, which satisfies the growth bound~\eqref{eq:growbd}.

\subsubsection{Semigroup for $\cD_m$}

In order to show that $\cD_m = \cL - \cK_m$ generates a strongly continuous semigroup, we note that $\cK_m \colon \cX^m\to\cX^m$ is bounded (see~Proposition~\ref{prop:compact}), and so the  Bounded Perturbation Theorem in~\cite[Theorem 1.3, Chapter III]{ElNa2000} applies directly. 
Using the coercive estimates \eqref{eq:dissip}, we moreover obtain 
\beq\label{eq:dissp_semi}
 e^{s \cD_m} : \cX^m \to \cX^m, \quad \|  e^{s \cD_m} \|_{\cX^m} \leq e^{-\lam s} .
\eeq
The decay estimate of the semigroup~\eqref{eq:dissp_semi} further implies the following spectral property of $\cD_m$ 
(see~\cite[Theorem 1.10, Chapter II]{ElNa2000})
\beq\label{eq:spec}
 \{ z \in {\mathbb C}: \Re(z) > -\lam \} \subset \rho(\cD_m) ,
\eeq
where $\rho(\cA)$ denotes the resolvent set of an operator $\cA$.

Note that the estimates
\eqref{eq:dissp_semi} and \eqref{eq:spec} apply to all  $(\cD_m, \cX^m)$ for $m \geq m_0$, with $\lam$ independent of $m$.

\subsection{Decay estimates of  $e^{\cL t}$ }\label{sec:semi_grow}

In this section, we follow \cite{ElNa2000} to obtain decay estimates for the semigroup $e^{\cL t}$. For a semigroup $e^{\cA t} : X \to X$ on a Banach space $X$ denote by ${\sigma}(\cA)$ the spectrum of $\cA$, i.e., the set $ \{ z \in {\mathbb C}: z - \cA \mbox{ is not bijective}\}$. We also introduce the  spectral bound ${\mathsf s}(\cA)$, the growth bound $\om_0(\cA)$, and the essential growth bound $\om_{ess}(\cA)$, defined by
\begin{align*}
{\mathsf s}(\cA) &\teq \sup \bigl\{ \Re (z) \colon z \in \sigma(\cA) \bigr\},\\
\om_0(\cA) &\teq \inf \bigl\{ \om \in {\mathbb R} \colon \mbox{ there exists } M_{\omega} \geq 1 \mbox{ such that } \| e^{\cA t}\|_{X \to X} \leq M_{\om} e^{\om t} \mbox{ for all } t \geq 0 \bigr\}, \\
\om_{ess}(\cA) & \teq \inf_{t > 0} \tfrac{1}{t} \log \| e^{\cA t} \|_{ess},
\end{align*}
where the norm $\| \cdot \|_{ess}$ is defined as 
$
 \| T \|_{ess} = \inf_{ K\colon X \to X \mathrm{ \ is \ compact} }  \| T - K\|_{ X \to X}.
$
With this notation, we have the following result:
\begin{proposition}[Corollary 2.11, Chapter IV,~\cite{ElNa2000}]\label{prop:growth_bd}
Let $e^{\cA t}$ be a strongly continuous semigroup generated by $\cA \colon D(\cA)\subset X \to X$, a closed operator. Then 
\[
\om_0 (\cA) = \max\{ \om_{ess}(\cA), {\mathsf s}(\cA) \}. 
\]
Moreover, for every $\om > \om_{ess}(\cA)$, the set $\s_c = \s(\cA) \cap \{ z \in {\mathbb C} \colon \Re (z) \geq \om  \} $ is finite, and the corresponding spectral projection has finite rank.
\end{proposition}

We apply Proposition~\ref{prop:growth_bd} to $(\cA,X) = (\cL,\cX^m)$, where $\cL =(\cL_U, \cL_{\S}) =   \cD_m + \cK_m$  and $\cK_m$ is the compact operator constructed in Proposition~\ref{prop:compact}. From the decay estimate \eqref{eq:dissp_semi}  we obtain $\om_{ess}(\cD_m) \leq - \lam$. Since $\om_{ess}$ is invariant under compact perturbations (see~\cite[Proposition 2.12, Chapter IV]{ElNa2000}), we get 
\beq\label{eq:spec_L}
\om_{ess}(\cL) = \om_{ess}( \cD_m) \leq -\lam.
\eeq

\subsubsection{Hyperbolic decomposition}\label{sec:hyper}

In this section, based on the estimate of $\om_{ess}(\cL)$ from~\eqref{eq:spec_L}, 
we orthogonally decompose the space $\cX^m$ (see~\eqref{eq:dec_X} below) and obtain decay estimates on $e^{\cL s}$ for the stable part, and growth estimates for the unstable part.
For this purpose, we first recall the following spectral decomposition for a closed operator $\cA$, based on Riesz projections.

\begin{proposition}[Theorem 2.1, Chapter XV, Part IV, page 326,~\cite{GoGoKa2013}]\label{prop:decom}
Suppose $\cA: D(\cA) \subset X \to X$ is a closed operator with specturm $\s(\cA) = \s \cup \tau$, where $\s$ is contained in a bounded Cauchy domain $\D$ such that $\bar \D \cap \tau = \emptyset$. Let $\G$ be the (oriented) boundary of $\D$. Then:
\begin{itemize}
\item[(i)] $P_{\s} = \f{1}{2 \pi i} \int_{\G} ( z  - \cA)^{-1} d z $ is a projection,

\item[(ii)] the subspaces $M = \Rg\, P_{\s}$ (the image of $P_{\s}$) and $N = \ker P_{\s}$ are $\cA-$invariant,

\item[(iii)] the subspace $M$ is contained in $D(\cA)$ and $\cA |_M$ is bounded,

\item[(iv)] $\s(\cA |_M) = \s$ and $\s(\cA |_N) = \tau$.
\end{itemize}

\end{proposition}

\noindent Our goal is to apply the above result to $(\cA,X) = (\cL,\cX^m)$. 
From~\eqref{eq:spec_L} and Proposition~\ref{prop:growth_bd},  we have that 
\[
\s_{\eta} \teq \s( \cL) \cap  \{ z\in {\mathbb C} : \Re(z) > - \eta \},
\qquad 
\mbox{where}
\qquad 
\eta \teq \tfrac{4}{5} \lam < \lam,
\]
only consists of finitely many eigenvalues of $\cL$, with finite multiplicity. Without loss of generality\footnote{The rigorous argument would be as follows. For $\bar \eta = \frac{9}{10} \lambda$, we have that $\s_{\bar \eta} = \s( \cL) \cap  \{ z\in {\mathbb C} : \Re(z) > - \bar \eta \}$ only consists of finitely many eigenvalues of $\cL$, with finite multiplicity. As such, for a.e. $\theta \in (\frac{7}{10},\frac{9}{10})$ the line in the complex plane with real-abscisa $-\theta \lambda$ does not intersect $\sigma(\cL)$. We have assumed without loss of generality that this $\theta$ may be taken to equal $\frac{4}{5}$, but of course the below arguments all hold {\em as is} if we replace $\eta = \frac 45 \lambda$ with $\eta = \theta \lambda$ for some $\theta \in (\frac{7}{10},\frac{9}{10})$.} we have that $\s_\eta$ is isolated from $\s(\cL) \backslash \s_{\eta}$. Applying Proposition~\ref{prop:decom} with $\cA = \cL$, $X = \cX^m$, and $\tau = \sigma(\cL) \backslash \sigma_\eta$, we obtain the hyperbolic decomposition 
\beq\label{eq:dec_X}
  \cX^m = \cX^m_{\mathsf u} \oplus \cX^m_{\mathsf s}, \quad \s( \cL|_{\cX^m_{\mathsf s}}) = \s(\cL) \backslash \s_{\eta}
\subset \{ z: \Re(z) \leq -\eta \}, 
  \quad \s( \cL|_{\cX^m_{\mathsf u}}) = \s_{\eta}.
\eeq
Moreover, from Proposition~\ref{prop:growth_bd} it follows that we can decompose the unstable part as 
\beq\label{eq:dec_Xu}
\cX^m_{\mathsf u} = \bigoplus_{ z \in \s_{\eta}} \ker( (z- \cL)^{\mu_z}), \quad \mu_z < \infty, \quad |\s_{\eta}| < +\infty.
\eeq
We will also show in~Section~\ref{sec:smooth_unstab} that $\cX^m_{\mathsf u}$ is spanned by smooth functions.

Since $\cX^m_{\mathsf s}$ is $\cL$ invariant, the restriction of the semigroup $e^{\cL s}$ on $\cX^m_{\mathsf s}$ is generated by $\cL|_{\cX^m_{\mathsf s}}$. Moreover, since $\max( \om_{ess}(\cL), {\mathsf s}(\s(\cL|_{\cX^m_{\mathsf s}} ))) \leq -\eta$, from Proposition~\ref{prop:growth_bd} and the definition of $\om_0(\cL|_{\cX^m_{\mathsf s}})$ we obtain
\beq\label{eq:decay_stab}
 \| e^{\cL s } \VV_0 \|_{\cX^m} \leq C_m e^{ - s \eta_{\mathsf s} } \| \VV_0 \|_{\cX^m}, \quad \forall V_0 \in \cX^m_{\mathsf s},  \quad s > 0,
 \qquad \mbox{where} \qquad \eta_{\mathsf s} \teq \tfrac{3}{5} \lam <  \tfrac{4}{5} \lam =  \eta .
\eeq
We also note that since $\cX^m_{\mathsf u} $ has finite dimension, $\s(\cL|_{\cX^m_{\mathsf u}}) = \s_{\eta}$, the operator
$\cL|_{\cX^m_{\mathsf u}}$ may be represented as a matrix with eigenvalues $> - \eta$; therefore, we obtain
\beq\label{eq:decay_unstab}
 \| e^{-\cL s} \VV_0 \|_{\cX^m} \leq  
 C_m e^{\eta s} \| \VV_0 \|_{\cX^m}.
   \quad \forall \VV_0 \in \cX^m_{\mathsf u} , \quad s > 0 .
\eeq

\subsubsection{Additional decay estimates of $\cL$ }
In order to localize the initial data (see~\eqref{eq:V2_form2:a} below), we will need the following decay estimates for the linear evolution~\eqref{eq:lin_semi:c}, with initial data supported in the far-field.

\begin{proposition}\label{prop:far_decay}
Let $R_4$ be as defined in Theorem~\ref{thm:coer_est}. Consider~\eqref{eq:lin_semi} with initial data $\VV_0 = (\UU_0, \S_0)$ with $\supp(\UU_0)\cup \supp(\S_0) \subset B( 0, R)^\complement$ for some $R > 4 R_4> \xi_s $. 
For any $m \geq m_0$, the solution $\VV(s) = (\UU(s), \S(s)) = e^{\cL s} \VV_0$ satisfies
\begin{equation*}
\supp(\VV(s)) \subset B( 0, 4 R_4)^\complement, \qquad  \| e^{\cL s} \VV_0\|_{\cX^m} \leq e^{-\lam s} \| \VV_0 \|_{\cX^m}.
\end{equation*}
\end{proposition}

\begin{proof}
The  proof of Proposition~\ref{prop:far_decay} uses that the support of the solution $e^{\cL s} \VV_0$ is moving away from $y=0$, remaining outside of $B(0, 4 R_4)$ for all time. Since $\supp(\cK_m f) \subset B(0,4 R_4)$ (item~(a) in Proposition~\ref{prop:compact}), we 
get $\cK_m \VV(s) = 0$ for all time $s$, and the desired decay estimate follows from \eqref{eq:dissip} or \eqref{eq:dissp_semi}.

Based on the above discussion, we only need to show that the solution $\VV(y,s)= 0$ for all $y \in B(0, 4 R_4)$ and $s>0$. Let $\chi$ be a radially symmetric cutoff function with $\chi(y) = 1$ for $|y| \leq 4 R_4$, $\chi(y)= 0 $ for $|y| \geq R>4 R_4$, and with $\chi(|y|)$ decreasing in $|y|$. Our goal is to show that the weighted $L^2$ norm $\int (|\UU(s)|^2 + |\S(s)|^2 ) \chi$ of the solution $\VV = (\UU,\S)$ of~\eqref{eq:lin_semi}, vanishes identically for $s\geq 0$. By assumption, we have that $\int (|\UU(0)|^2 + |\S(0)|^2 ) \chi = 0$, so it remains to compute $\frac{d}{ds}$ of this weighted $L^2$ norm using~\eqref{eq:lin_semi}.

Recall the decomposition of $\cL$ from \eqref{eq:lin_Hk} with $m=0$ (hence, with $\cR_{U, m} =\cR_{\S, m}=0$).  Performing weighted $L^2$ estimates analogous to the ones in the proof of Theorem~\ref{thm:coer_est}, for the transport terms we obtain 
\begin{align}
&\int  ( \cT_U \cdot \UU   + \cT_{\S} \, \S ) \chi  \notag\\
&\quad 
= - \int \B( (y + \bar \UU) \cdot \na \bigl( \tfrac 12 |\UU|^2 \bigr)
+ (y + \bar \UU) \cdot \na  \bigl( \tfrac 12 \S^2 \bigr)   + \al  \bar \S \,  \na  \S \cdot \UU
+ \al \bar \S \div( \UU)  \S 
\B) \chi \notag \\
&\quad = \int \tfrac{1}{2} \div \B( (y + \bar \UU ) \chi \B) 
(|\UU|^2 + \S^2)
+ \al \na ( \bar \S \chi) \cdot \UU \S \notag \\
&\quad \leq \int \tfrac{1}{2} \B(  (y + \bar \UU) \cdot \na \chi
 + \chi \div( y + \bar \UU)  \B) ( |\UU|^2  + \S^2) 
 + \al \bigl( \bar \S |\na \chi| + \chi |\na \bar \S| \bigr) |\UU \S|.
\label{eq:junk:label:1}
\end{align}
We focus on the terms in~\eqref{eq:junk:label:1} involving $|\na \chi|$. Since $\chi$ is radially symmetric, we get $(y + \bar \UU) \cdot \na \chi = (\xi + \bar U) \pa_{\xi} \chi$. Using Cauchy-Schwarz, the fact that $ \pa_{\xi} \chi(y) = 0, |y| \leq 4 R_4$ and  $\pa_{\xi} \chi \leq 0$ globally, and using that \eqref{eq:rep2} yields $ \xi + \bar U(\xi) - \al \bar \S(\xi) > 0$ for $\xi = |y| > \xi_s$ (hence for $\xi>4 R_4$), we obtain 
\begin{align*}
\tfrac{1}{2} \bigl(  (y + \bar \UU) \cdot \na \chi  \bigr) ( |\UU|^2  + \S^2) 
 + \al \bar \S |\na \chi| |\UU \S| 
& \leq \tfrac{1}{2}\bigl( (\xi + \bar U ) \pa_{\xi}\chi + \al \bar \S |\pa_{\xi} \chi| \bigr)  ( |\UU|^2  + \S^2)  \\
&\leq \tfrac{1}{2}(\xi + \bar U  -\al \bar \S) \pa_{\xi}\chi  ( |\UU|^2  + \S^2)  \leq 0.
\end{align*}

For remaining contributions, resulting from the $\chi$-terms in~\eqref{eq:junk:label:1}, and from the $\cD_U, \cS_U, \cD_\S,\cS_\S$-terms in \eqref{eq:lin_Hk}, in light of~\eqref{eq:dec_U} we have that 
\begin{align*}
&\tfrac 12 \chi \div( y + \bar \UU)  ( |\UU|^2  + \S^2) + \al  \chi |\na \bar \S| |\UU \S|
- (r -1) \chi ( |\UU|^2 + \S^2) 
- \chi (\UU\cdot\na\bar \UU \cdot \UU + \UU\cdot \na \bar \S \, \S) 
\notag\\
&\leq C \chi ( |\UU|^2  + \S^2)
\end{align*}
for some sufficiently large $C>0$ (depending on $\alpha, r, \bar \UU, \bar \S$). Thus, we obtain 
\[
\tfrac{1}{2} \tfrac{d}{ds} \int (|\UU|^2 + \S^2 )\chi 
= 
\int  ( \cL_U \cdot \UU  + \cL_{\S} \,  \S ) \chi 
\leq C \int (|\UU|^2 + \S^2 )\chi,
\]
which implies via Gr\"onwall that $\int (|\UU(s)|^2 + \S(s)^2 )\chi  = 0$ for all $s\geq0$. The claim follows.
\end{proof}

 \subsection{Smoothness of unstable directions}\label{sec:smooth_unstab}

In this section, we show that the unstable space $\cX_{\mathsf u}^m$ given in \eqref{eq:dec_Xu} is spanned by smooth functions. This fact will be shown to follow  from the following abstract lemma.
 
\begin{lemma}\label{lem:smooth}
Let $\{ X^i\}_{i \geq 0}$ be a sequence of Banach spaces with $X^{i+1} \subset X^{i}$ for all $i \geq 0$.
Assume that for any $i \geq i_0$ we can decompose the linear operator 
$\cA \colon D(\cA) \subset X^i \to X^i$ as 
 $\cA = \cD_i + \cK_i$, where the linear operators $\cD_i$ and $\cK_i$ satisfy
\beq\label{eq:smooth_ass}
\cD_i : D(\cA) \subset X^i \to X^i , \quad  \cK_i :  X^{i-1}  \to  X^i,
\quad \{ z \in {\mathbb C}\colon \Re(z) > - \lam \} \subset \rho(\cD_i).
\eeq
Here, $\rho(\cdot)$ denotes the resolvent set of an operator and $\lambda>0$ is independent of $i\geq i_0$. Fix $n\geq 0$ and $z \in {\mathbb C}$ with $\Re(z) > -\lam$. Assume that the functions $ f_0, \ldots , f_n  \in  X^{i_0}$ satisfy 
\[
(z - \cA) f_0 = 0,
\quad (z - \cA) f_{i} = f_{i-1}, \quad \mbox{ for } \quad 1 \leq i \leq n.
 \]
Then, we have $ f_0, \ldots , f_n  \in X^{\infty} \teq \cap_{i \geq 0} X^i$. 
\end{lemma}

\begin{proof}
It suffices to show that:  
\begin{equation*}
\mbox{if } f \in X^{i_0} \mbox{ satisfies } (z - \cA) f = F \in X^{\infty},
\mbox{ then } f \in X^{\infty}.
\end{equation*}
Indeed, the lemma then follows by  induction on $i \in \{0,\ldots,n\}$.

In order to prove the above claim, we further use induction on $i \geq i_0$ to show that:
\begin{equation}
\label{eq:smooth2}
\mbox{if } f \in X^{i_0}\cap X^{i-1} \mbox{ satisfies } (z - \cA) f = F \in X^{\infty},
\mbox{ then } f \in X^{i}.
\end{equation}
The base case $i= i_0$ in \eqref{eq:smooth2} holds true automatically. For the inductive step, for $i-1\geq i_0$ we assume that $f \in X^{i-1} = X^{i_0} \cap X^{i-1}$.  Using that $\cA = \cD_i + \cK_i$ and $\cK_i: X^{i-1} \to X^{i}$  from \eqref{eq:smooth_ass}, we get $(z - \cD_i) f = F + \cK_i f \in X^{i}$.
From \eqref{eq:smooth_ass} and the assumption $\Re(z) > -\lam$, we get $z \in \rho(\cD_i)$ and thus $f = (z - \cD_i)^{-1} (F + \cK_i f) \in X^{i}$. This proves~\eqref{eq:smooth2}. By induction on $i$, it follows that $f \in X^{\infty}$.
\end{proof}

We are now able to apply Lemma~\ref{lem:smooth} with $(\cA, \{X^i\}_{i\geq i_0}) \rightsquigarrow (\cL, \{\cX^m\}_{m\geq m_0})$, and the decomposition $\cL = \cD_m + \cK_m$ for $m \geq m_0$, with $\cK_m$ as constructed in Proposition~\ref{prop:compact}, in order to describe the regularity of elements of $\ker( (z- \cL)^{\mu_z})$ for $n \rightsquigarrow \mu_z < \infty$, and $z \in \sigma_\eta$ (see~\eqref{eq:dec_X} and~\eqref{eq:dec_Xu}). Indeed, the assumptions of Lemma~\ref{lem:smooth} are satisfied in light of Lemma~\ref{lem:Xm_chain}, Proposition~\ref{prop:compact}, and \eqref{eq:spec}.

 We fix $z \in \s_\eta = \s(\cL) \cap \{z\in \mathbb{C} \colon \Re(z) > -\eta\}$, where we recall $- \eta  >  - \lam$. We also fix $g \in \ker( (z- \cL)^{\mu_z}) \subset \cX_{\mathsf u}^{m}$ for some integer $\mu_z < \infty$, as constructed in \eqref{eq:dec_X} and \eqref{eq:dec_Xu}. 
We apply Lemma~\ref{lem:smooth} with  $f_i = (z- \cL)^{\mu_z-i} g$ for $0 \leq i \leq \mu_z$, to deduce that  $\{ f_i\}_{i=0}^{\mu_z} \subset \cX^{\infty} \teq \cap_{m \geq m_0} \cX^m  $. Since 
\[
\cX^{\infty}  = \cap_{m \geq m_0} \cX^m \subset C^{\infty} ,
\]
we deduce from~\eqref{eq:dec_Xu} that $\cX_{\mathsf u}^{m}$ is spanned by smooth radially-symmetric functions.

\section{Nonlinear stability}
\label{sec:non}

The goal of this section is to prove Theorem~\ref{thm:blowup}, by constructing global solutions $(\UU,\S, \vA)$ to \eqref{eq:euler_ss}, with a number of desirable properties, in the vicinity of the stationary profile $(\bar \UU,\bar \S,0)$. Throughout this section we fix $m\geq m_0$, where $m_0$ is sufficiently large, as in Theorem~\ref{thm:coer_est}. 

\subsection{Decomposition of the solution} 
 We will use $\VV = (\UU,  \S,\AA)$ to denote the nonlinear solution to~\eqref{eq:euler_ss}. As in~\eqref{eq:linear+nonlinear}, we denote the perturbation to the stationary profile as 
\[
\td \VV \teq (\td \UU,\td \S,\AA) \teq (\UU - \bar \UU, \S - \bar \S,\AA),
\]
and recall that this perturbation solves \eqref{eq:lin}--\eqref{eq:non}. We shall further decompose the perturbation\footnote{We emphasize that the $\td \cdot_1$ or $\td \cdot_2$ sub-index denote different parts of the perturbation, they  {\em do not represent} Cartesian coordinates.} as $\td \VV = \td \VV_1 + \td \VV_2$, with $\td \UU = \td \UU_1 + \td \UU_2$, $\td \S = \td \S_1 + \td \S_2$, so that
\beq\label{eq:init}
\bal
\VV= (\UU,  \S,\vA) 
= \td \VV_1 + \td \VV_2 + \bar \VV, \quad 
\td \VV_1 = (\td \UU_1, \td \S_1,\vA), \quad 
\td \VV_2 =  (\td \UU_2, \td  \S_2,0), \quad  
\bar \VV =  (\bar \UU, \bar \S,0).
 \eal
 \eeq
The fields $\td \VV_1$ and $\td \VV_2$ are defined as the solutions of
\begin{subequations}
\label{eq:non_V}
\begin{align}
 & \pa_s \td \VV_1   = \cL_V (\td \VV_1) 
 - \cK_m(\td \VV_1) + \cN_V(\td \VV) ,
 \label{eq:non_V:a} \\
 & \pa_s  \td \VV_2  = \cL ( \td \VV_2) 
 + \cK_m( \td \VV_1), 
 \label{eq:non_V:b}
\end{align}
\end{subequations}
with initial data at time $s = \sin = \log T^{-1/r}$ given by $\td \VV_{1,in}$ as described in Theorem~\ref{thm:non},  and with $\td \VV_{2,in}$ as given by~\eqref{eq:V2_form2:d}. The operator $\cK_m$ in~\eqref{eq:non_V} is as defined in Proposition~\ref{prop:compact}, 
while the linear operators $\cL$, $\cL_V$, and the nonlinear term $\cN_V$ are given in terms of \eqref{eq:lin}--\eqref{eq:non} and 
 \begin{equation*}
 \cL  = (\cL_{U}, \cL_{\S}), \qquad 
 \cL_V = (\cL_{U}, \cL_{\S}, \cL_A), \qquad
 \cN_V = (\cN_{U}, \cN_{\S}, \cN_A).
\end{equation*}
It is clear, by definition, that a global solution $\td\VV_1, \td \VV_2$ of \eqref{eq:non_V} provides via \eqref{eq:init} a global solution $\VV$ of \eqref{eq:euler_ss}.

In \eqref{eq:non_V}, and in the remainder of the section, we have abused notation two-fold. First, in~\eqref{eq:non_V:a} we note that the output of the $\cK_m$ operator is an element of $\cX^m$ (hence, a $\UU$ and a $\Sigma$) and so it does not contain an $\AA$-component (as is required by the $\td \VV_1$ representation in~\eqref{eq:init}); as such, in~\eqref{eq:non_V:a} we write $\cK_m(\td \VV_1)=\cK_m(\td \UU_1,\td \S_1)$ to mean $(\cK_m(\td \UU_1, \td \S_1),0)$, i.e., $\cK_m(\td \VV_1)$ does not act on the $\AA$-component. Second, we note that the $\td \VV_2$ representation in~\eqref{eq:init} has a $0$ in the $\AA$-component; as such, in~\eqref{eq:non_V:b} we identify $\td \VV_2$ with the pair $(\td \UU_2,\td \S_2)$, which then matches the domain and range of the operators $\cL$ and $\cK_m$.
 
There are a few important advantages to the  decomposition\footnote{This decomposition was first developed in~\cite{ChHo2023} for stable blowup analysis of the 3D incompressible Euler equations.} \eqref{eq:init} and the definitions~\eqref{eq:non_V}. First, the part $\td \VV_2$, which is used to capture unstable parts, is decoupled from the equation of $\td \VV_1$ at the linear level, and so we can obtain decay estimates for $\td \VV_1$ directly using energy estimates and the dissipativity of $(\cL-\cK_m,\cL_A)$ (see~\eqref{eq:coer_estA} and~\eqref{eq:dissip}), without appealing to semigroups.
Second, we can obtain a representation formula (and an estimate) for $\td \VV_2$ by using Duhamel's formula 
\begin{equation}
\label{eq:V2_form1}
\td \VV_2(s) = e^{\cL (s-\sin)} \td \VV_{2,in} + \int_{\sin}^s e^{\cL(s -s^\prime)} \cK_m(\td \UU_1, \td \S_1)(s^\prime) d s^\prime
\end{equation}
for some initial data $\VV_{2,in}$, to be chosen later. We note that for \eqref{eq:non_V} without the nonlinear parts, we can first obtain a decay estimate for $\td \VV_1$, and then construct $\td \VV_2$ using the above Duhamel formula, globally in time. 
We obtain bounds for the full nonlinear  system~\eqref{eq:non_V}  by treating the nonlinear terms perturbatively. 

The Duhamel representation~\eqref{eq:V2_form1} for the unstable part of the perturbation is deceivingly simple; this is because we did not yet specify the initial data $\td \VV_{2,in}$. In practice, we need to use the usual technique for constructing unstable manifolds, and apply the semigroup backward-in-time on the unstable piece of $\td \VV_{2}$. Additionally, in order to obtain initial data with compact support, we modify \eqref{eq:V2_form1} as follows
\begin{subequations}
\label{eq:V2_form2}
\begin{align}
\td \VV_2(s) 
&\teq \td \VV_{2,{\mathsf s}}(s) -\td \VV_{2, {\mathsf u}}(s) + e^{\cL (s-\sin)}\Bigl(\td \VV_{2, {\mathsf u}}(\sin) \bigl(1 - \chi\bigl(\tfrac{y}{8 R_4}\bigr) \bigr) \Bigr), 
\label{eq:V2_form2:a}\\
\td \VV_{2, {\mathsf s}} (s) 
&\teq \int_{\sin}^s e^{\cL(s - s^\prime)} \Pi_{\mathsf s} \cK_m(\td \UU_1, \td \S_1)(s^\prime) d s^\prime,  
\label{eq:V2_form2:b} \\
\td \VV_{2, {\mathsf u}} (s) 
&\teq  \int_s^{\infty} e^{-\cL(s^\prime-s)} \Pi_{\mathsf u} \cK_m(\td \UU_1, \td \S_1)(s^\prime) d s^\prime ,
\label{eq:V2_form2:c}
\end{align}
where $\chi$ is a smooth radial cutoff function with $\chi(y) = 1$ for $|y| \leq 2/3 $, and $\chi(y) = 0$ for $|y| \geq 1$, $\Pi_{\mathsf u}$ is the orthogonal projection from $\cX^m$ to $\cX_{\mathsf u}^m$ (see~\eqref{eq:dec_X}--\eqref{eq:dec_Xu}) and $\Pi_{\mathsf s} \teq {\rm Id} - \Pi_{\mathsf u}$.
It is not difficult to see that \eqref{eq:V2_form2} agrees with \eqref{eq:V2_form1} for initial data taken as
\begin{equation}
\td \VV_{2, in} 
= - \td \VV_{2,{\mathsf u}}(\sin) \chi\bigl(\tfrac{y}{8 R_4}\bigr)
=-\chi\bigl(\tfrac{y}{8 R_4}\bigr)
\int_{\sin}^{\infty} e^{- \cL (s^\prime-\sin)} \Pi_{\mathsf u} \cK_m(\td \UU_1, \td \S_1)(s^\prime) d s^\prime.
\label{eq:V2_form2:d}
\end{equation}
\end{subequations}
Since the cutoff $\chi$ and elements of $\cX^m$ (the range of $\Pi_{\mathsf u}\cK_m$ being a subset) are radially symmetric, if follows that so is $\td \VV_{2, in}$.
The detailed representation \eqref{eq:V2_form2} shows that $\td \VV_2$ is computed as a function of $(\td \UU_1, \td \S_1)$; for later purposes it is useful to codify this relation as a map, $\cT_2$, and to denote
\begin{equation}  
\cT_2(\td \UU_1, \td \S_1)  \teq
\mbox{Right Side of }\eqref{eq:V2_form2:a}
.
\label{eq:T2:def} 
\end{equation}

\subsection{Functional setting} 

We introduce the following space for the perturbation $\td \VV_1 = (\td \UU_1, \td \S_1,  \AA)$:
\beq\label{norm:Z}
\cZ^i := \cX^i \times \cX^i \times \cX_A^{i}. 
\eeq
Next, we introduce the spaces $\cW^{m+1}$ and $\cW_{A}^{m+1}$, which are used for closing nonlinear estimates. Our goal is to perform both weighted $H^{2m}$ and weighted $H^{2m+2}$ estimates on \eqref{eq:non_V}, using the {\em same compact operator} $\cK_m$ and the {\em same projections} $\Pi_{\mathsf s}, \Pi_{\mathsf u}$ appearing in~\eqref{eq:non_V} and~\eqref{eq:V2_form2:b}--\eqref{eq:V2_form2:c}; that is, we do not wish to change $\cK_m$ into $\cK_{m+1}$ for the weighted $H^{2m+2}$ bound. 

Fix an arbitrary $ m \geq m_0$. For some $\mu_{m+1} > 0$ to be chosen sufficiently small,  using Theorem~\ref{thm:coer_est}, Proposition~\ref{prop:compact} (which in particular gives that $\cK_m : \cX^{0} \to \cX^{m+1}$), and the fact that by definition we have $\|\cdot\|_{\cX^0} \leq \|\cdot\|_{\cX^m}$, we obtain 
\begin{align*}   
&\la  (\cL - \cK_m) f , f \ra_{\cX^m}   + \mu_{m+1} \la  (\cL - \cK_m) f , f \ra_{\cX^{m+1}} 
\\
&\qquad \leq - \lam \| f \|_{\cX^m}^2 
+ \mu_{m+1} \bigl( \la \cL f , f \ra_{\cX^{m+1}}   
+ \| \cK_m f \|_{\cX^{m+1}} \| f \|_{\cX^{m+1}}\bigr ) 
 \\
&\qquad \leq 
-\lam \| f \|_{\cX^m}^2 
+ \mu_{m+1} \bigl( - \lam \| f \|_{\cX^{m+1} }^2 + \bar C \|f\|_{\cX^0}^2 + C_m \|f\|_{\cX^0} \|f\|_{\cX^{m+1}} \bigr)
 \\
&\qquad \leq 
- \tfrac{9}{10} \lam \bigl( \| f \|_{\cX^m}^2  + \mu_{m+1} \| f \|_{\cX^{m+1}}^2 \bigr)
+
\bigl(-\tfrac{1}{10} \lam + \mu_{m+1} (\bar C + \tfrac{5}{\lambda} C_m^2) \bigr)
 \| f \|_{\cX^m}^2,
\end{align*}
for all $f \in \{ (\UU,\Sigma) \in \cX^{m+1} \colon \cL(\UU,\Sigma) \in \cX^{m+1}\}$.
Choosing $ \mu_{m+1} >0 $ small enough in terms of $m$ and $\lambda$, from the above bound we obtain the coercive estimate
\begin{equation*}
\la  (\cL - \cK_m) f , f \ra_{\cX^m}   + \mu_{m+1} \la  (\cL - \cK_m) f , f \ra_{\cX^{m+1}} 
\leq  
- \tfrac{9}{10} \lam \bigl( \| f \|_{\cX^m}^2  + \mu_{m+1} \| f \|_{\cX^{m+1}}^2 \bigr)
.
\end{equation*}
Note that coercive estimates for $\cL_A$ in both $\cX_A^m$ and $\cX_A^{m+1}$ are directly available from Theorem \ref{thm:coer_est}. 

In light of the above coercive bounds, with $\mu_{m+1}>0$ chosen as above, we  define the Hilbert spaces $\cW^{m+1}\subset \cX^{m+1}, \cW_A^{m+1} \subset \cX_A^{m+1}$ according to the inner products
\begin{subequations}
\label{norm:Wk}
\begin{align}
&\la f, g \ra_{\cW^{m+1}} 
\teq \mu_{m+1} \la f, g \ra_{\cX^{m+1}}   
+ \la f , g \ra_{\cX^m}, \qquad \| f \|_{\cW^{m+1}}^2 = \la f , f \ra_{\cW^{m+1}},\\
&\la f, g \ra_{\cW_A^{m+1}} 
\teq \mu_{m+1} \la f, g \ra_{\cX_A^{m+1}}   
+ \la f , g \ra_{\cX_A^m}, \qquad \| f \|_{\cW_A^{m+1}}^2 = \la f , f \ra_{\cW_A^{m+1}},
\end{align}
\end{subequations}
and obtain that 
\beq\label{eq:coer_Wk}
  \la (\cL - \cK_m ) f , f \ra_{\cW^{m+1}} \leq - \lam_1 \| f\|^2_{ \cW^{ m+1 }}, \qquad 
  \la  \cL_A f , f \ra_{\cW_A^{m+1}} \leq - \lam_1 \| f\|^2_{ \cW_A^{ m+1 }}, \qquad 
  \lam_1 = \tfrac{9}{10 } \lam,
\eeq
for all $f \in \{ (\UU,\Sigma) \in \cX^{m+1} \colon \cL(\UU,\Sigma) \in \cX^{m+1}\}$, respectively for all $f \in \{ \AA \in \cX_A^{m+1} \colon \cL_A(\AA) \in \cX_A^{m+1}\}$.
Estimate~\eqref{eq:coer_Wk} shows that we can use the same compact operator $\cK_m$ to simultaneously obtain coercive estimates in weighted $H^{2m+2}$ and weighted $H^{2m}$ spaces.

\subsection{Nonlinear stability and the proof of Theorem~\ref{thm:blowup}}
We recall from \eqref{eq:decay_stab} (shifted so that the initial time is $s=\sin$ instead of $s=0$) and~\eqref{eq:coer_Wk}, that the decay rates $\eta_s, \eta$, and $\lam_1$  are given by
\beq\label{eq:decay_para}
(\eta_s, \eta, \lam_1) = \bigl( \tfrac{3}{5}, \tfrac{4}{5}, \tfrac{9}{10}\bigr) \lam,\qquad  \eta_s < \eta < \lam_1 < \lam.
\eeq
We also recall that the regularity parameter $m \geq m_0$ has been fixed.
The goal of this section is to prove:

\begin{theorem}[\bf Nonlinear stability] \label{thm:non}
Fix $m\geq m_0$. 
There exists a sufficiently small $\d > 0$ such that
for any initial data $\td \VV_{1,in} = (\td \UU_1(\sin),\td \S_1(\sin),  \AA(\sin))$ which is {\em smooth} enough\footnote{We require the $\cZ^{m+2}$-regularity of $\td \VV_{1,in}$, a space which is stronger than $\cZ^{m+1}$, in order to obtain the local-in-time existence of a $\cZ^{m+2}$-solution (see Footnote~\ref{foot:footnote:details}); in turn, this allows us to justify a few estimates, e.g.~\eqref{eq:coer_Wk} for $(\td \UU_1, \td \S_1)$ which requires $\cL(\td \UU_1, \td \S_1) \in \cX^{m+1}$. Note that this regularity requirement is only \textit{qualitative}, and we only use Theorem~\ref{thm:non} with an $C^{\infty}$ initial perturbation (see~\eqref{eq:the:IC}) in order to prove Theorem~\ref{thm:blowup}.The only {\em quantitative} assumption on the initial data is given by~\eqref{eq:IC:small}.} to ensure $\td \VV_{1,in} \in  \cZ^{m+2}$ (see~\eqref{norm:Z})
and {\em small} enough to ensure
\begin{equation}
\|(\td \UU_1(\sin), \td \S_1(\sin))\|_{\cW^{m+1}}^2 
+ \|\vA(\sin)\|_{\cW_A^{m+1}}^2 < \d^2,
\label{eq:IC:small}
\end{equation} 
there exists 
a global solution $\td \VV_1$ to \eqref{eq:non_V:a} with initial data $\td\VV_{1,in}$, and a global solution $\td \VV_2$ of \eqref{eq:non_V:b}  given by~\eqref{eq:V2_form2}, satisfying the bounds
\begin{subequations}
\label{eq:solution:small}
\begin{align}
\|(\td \UU_1(s), \td \S_1(s))\|_{\cW^{m+1}}^2 + 
\|\vA(s)\|_{\cW_A^{m+1}}^2  &< 4 \d^2 e^{-2 \lam_1 (s-\sin)},
\label{eq:solution:small:a}\\
 \| ( \td \UU_2(s),\td  \S_2(s))  \|_{\cX^{m+2}} 
 &\les_m \d_Y e^{-\eta_s (s-\sin)},
 \qquad \mbox{where} \qquad \d_Y \teq \d^{2/3},
 \label{eq:solution:small:b}
\end{align}
\end{subequations}
for all $s\geq \sin$.
We emphasize that we cannot {\em prescribe} the initial data $\td \VV_{2,in} = (\td \UU_2(\sin),\td  \S_2(\sin))$; rather, this data is constructed via \eqref{eq:V2_form2:d} (simultaneously with the solution $\td \VV_1$) to lie in a finite-dimensional subspace of $\cX^{m+2}$.
\end{theorem}

\begin{remark}
\label{rem:data}
The initial data for $(\UU,\S,\AA)$ is obtained from Theorem~\ref{thm:non} and the decomposition~\eqref{eq:init} at time $s=\sin$. In light of Theorem~\ref{thm:non}, we identify the space $X_2$ mentioned in Remark~\ref{rem:initial:data:set} with an open ball in the weighted Sobolev space $\cW_A^{m+1}$ defined in~\eqref{norm:Wk}.
On the other hand, the space $X_1$ mentioned in Remark~\ref{rem:initial:data:set} consists of functions which are given as the sum of an element $(\td \UU_1(\sin), \td \S_1(\sin))$ which lies in open ball in the weighted Sobolev space $\cW^{m+1}$ (see definition~\eqref{norm:Wk}) and the element $(\td \UU_2(\sin),\td  \S_2(\sin))$ constructed in \eqref{eq:V2_form2:d}, which lies in a finite-dimensional subspace of $\cX^{m+2}$.
\end{remark}

Using the above Theorem~\ref{thm:non}, we are now ready to prove Theorem~\ref{thm:blowup}. We note though that Theorem~\ref{thm:non} implies more than what was claimed in Theorem~\ref{thm:blowup}; these additional fine properties of the solution are discussed in Section~\ref{sec:additional:estimates}.

\begin{proof}[Proof of Theorem \ref{thm:blowup}]
We choose an initial perturbation 
\[ 
\td \VV_{1,in}(y) = (\td \UU_1(y,\sin), \td \S_1(y,\sin), \vA(y,\sin)) = (\td U_1(\xi,\sin) \ee_R,\td \S_1(\xi,\sin), A(\xi,\sin) \ee_\th)
\] 
as in Theorem~\ref{thm:non}, where $\xi = |y|$. By definition, this initial data consists of radially-symmetric vectors/scalars.
Moreover, this initial perturbation may be chosen to ensure that 
\begin{equation*}
 \td U_1(\cdot,\sin) + \bar U \in C_c^{\infty}({\mathbb R}_+), 
 \quad 
 A(\cdot,\sin) \in C_c^{\infty}({\mathbb R}_+),\  \pa_{\xi} A(0, \sin) \neq 0, 
\quad \td \S_1(\cdot,\sin) + \bar \S = 1, \ \xi \geq C_{in}, 
\end{equation*}
for some $C_{in}\geq 1$ large enough, to be chosen below. For {\em example}, we can choose 
\begin{subequations}
\label{eq:the:IC}
\begin{align}
\td U_1(\xi,\sin) &= - \bar U(\xi) \Bigl(1- \chi\bigl(\tfrac{\xi}{C_{in}}\bigr) \Bigr), 
\label{eq:the:IC:a}\\
\td \S_1(\xi,\sin) &= \bigl(1 - \bar \S(\xi) \bigr) \Bigl(1 - \chi\bigl(\tfrac{\xi}{C_{in}} \bigr) \Bigr) 
\label{eq:the:IC:b}\\
A(\xi,\sin) &= \tfrac{\xi }{C_{in}}\chi(\xi) ,
\label{eq:the:IC:c}
\end{align}
\end{subequations}
where $C_{in}>0$ is to be determined, and $ \chi : \R_+ \to [0, 1]$ is a radially symmetric and smooth cutoff function with $\chi(\xi) = 1$ for $0\leq \xi \leq 1/2$, and $\chi(\xi) = 0$ for  $\xi \geq 1$. 

In order to verify that the data presented in~\eqref{eq:the:IC}  satisfies the assumption of Theorem~\ref{thm:non}, we need to verify~\eqref{eq:IC:small}.\footnote{Indeed, initial data such as the one in~\eqref{eq:the:IC} clearly satisfies the smoothness assumptions of Theorem~\ref{thm:non} since the profile $(\bar U,\bar \S)$ is $C^\infty$ smooth, radially-symmetric, and decays as $\xi \to \infty$ according to~\eqref{eq:dec_U}.} 
Letting $F(y) = (\td \UU_1(y,\sin),\td\S_1(y,\sin))$, we note from~\eqref{eq:dec_U} that $|F(y)| \les \one_{|y| \geq C_{in}/2} \la y \ra^{1 - r} \in L^2(\vp_g)$, where $\vp_g$ is as defined in~\eqref{eq:vp:g:def}. Similarly, for $i\in \{m,m+1\}$, from~\eqref{eq:dec_U} we have $|\na^{2i} F| \les_i \one_{|y| \geq C_{in}/2} \la y \ra^{ 1 - 2 i - r  } \in L^2(\vp_{2i}^2 \vp_g)$, with $\vp_{2i}$ as defined in~\eqref{eq:vp:g:def}. Thus, by the dominated convergence theorem we have that $\| F \|_{\cX^m} \to 0$ and $\|F\|_{\cX^{m+1}} \to 0$ as $C_{in} \to \infty$. In particular, we may ensure $\|(\td \UU_1(\sin), \td \S_1(\sin))\|_{\cW^{m+1}} \leq \delta/2$ by choosing $C_{in}$ sufficiently large. Next, we note that the function $F(y) = \xi \chi(\xi) \ee_\th = C_{in} \AA(y,\sin)$, is $C^\infty$ smooth and has compact support in the unit ball. Moreover, using that the weight $\vp_A$ defined in~\eqref{eq:wg_A} satisfies $\vp_A(\xi)\asymp \xi^{-\b_1}$ with $\b_1 \in (3,4)$, we deduce that  $\|F\|_{\cX^0_A}^2 \les \int_{|\xi|\leq 1}\xi^2 \xi^{-b_1} \xi d\xi \asymp 1$. Thus, $\|F\|_{\cW_A^{m+1}} \les_m 1$, and hence $\|\AA(\sin)\|_{\cW_A^{m+1}} \leq \delta/2$ for $C_{in}$ sufficiently large. Thus, the initial data in~\eqref{eq:the:IC} satisfies~\eqref{eq:IC:small} whenever $C_{in}$ is sufficiently large.

Without loss of generality, we may choose $C_{in} > 16 R_4$, and thus definitions~\eqref{eq:the:IC:a}--\eqref{eq:the:IC:b} imply that $\supp(\td \UU_1(\sin),\td \S_1(\sin)) \subset B(0,8 R_4)^\complement$. Moreover, we have $\supp(\td \UU_1(\sin) + \bar \UU, -1 + \td \S_1(\sin) + \bar \S ) \subset B(0,C_{in})$ and $\supp (\AA(\sin)) \subset B(0,1)$; combined with $\supp(\td\VV_{2,in}) = \supp (\td \UU_2(\sin),\td\S_2(\sin)) \subset B(0, 8 R_4)$ (see~\eqref{eq:V2_form2:d}), we deduce from~\eqref{eq:init} that the initial data $(\UU(\cdot,\sin), \S(\cdot,\sin) - 1, \AA(\cdot,\sin))$ to \eqref{eq:euler_ss} has compact support in ${\mathbb R}^2$. Additionally, the initial datum has no vacuum regions. Indeed, we have already shown that $\S(y,\sin)= 1$ for $|y| > 8 R_4$. For $|y| \leq 8 R_4$, we use the fact that $\td \S_1(y,\sin) = 0$, so that by~\eqref{eq:solution:small} we obtain $\S(y,\sin) = \bar \S(y) + \td \S_2(y,\sin) \geq \min_{|y|\leq 8 R_4} \bar \S(y) - C_m \delta_Y \geq \tfrac 12 \min_{|y|\leq 8 R_4} \bar \S(y) > 0$, since $\delta_Y =\delta^{2/3}$ was taken to be sufficiently small, and since $\|\cdot\|_{L^\infty} \les_m \|\cdot \|_{\cX^{m+1}}$ for functions of compact support. Thus, the initial data for \eqref{eq:euler_ss} is bounded away from vacuum.

By construction (see~\eqref{eq:the:IC:c}), we have that $\partial_\xi A(\xi,\sin)|_{\xi=0} = C_{in}^{-1} > 0$. Moreover, since $\omega_0(R,t)|_{t=0}  = \omega_0(R) = (R^{-1} + \partial_R) u^\theta_0(R)$, using the ansatz~\eqref{ss-var:b} we obtain 
\begin{equation*}
\omega_0(0) = \lim_{\xi \to 0^+} \tfrac{1}{r T} \bigl(\tfrac{1}{\xi} + \partial_\xi\bigr) A(\xi,\sin) = \tfrac{2}{rT} \partial_\xi A(0,\sin)  = \tfrac{2}{rT C_{in}} \neq 0 . 
\end{equation*}
Moreover, note that for all $t\in [0,T)$ we have $\uu(0,t) = \frac 1r (T-t)^{\frac 1r - 1} (\UU(0,s) + \AA(0,s)) = 0$, because $\UU$ and $\AA$ are smooth radially-symmetric vectors (hence, they vanish linearly in $\xi$ as $\xi \to 0^+$, for all $s\geq \sin$);  thus the Lagrangian flow $X(a,t)$ emanating from $a=0$ is frozen at $0$, i.e., $X(0,t) = 0$ for all $t\in [0,T)$. In light of the specific vorticity transport~\eqref{eq:vorticity:transport}, the relation $\rho = (\alpha \sigma)^{1/\alpha}$, and taking into account the self-similar transformation~\eqref{ss-var:b}, we deduce
\[
\om(0,t) = \rho(0,t)  \tfrac{\om_0(0)}{\rho_0(0)} 
= \om_0(0) \Bigl(\bigl(1- \tfrac{t}{T}\bigr)^{1/r-1} \tfrac{\Sigma(0,s)}{\Sigma(0,\sin)} \Bigr)^{1/\alpha} 
= \om_0(0)\bigl(1- \tfrac{t}{T}\bigr)^{- (r-1)/(r\alpha)} 
\bigl(\tfrac{\bar \S(0) + \td \Sigma(0,s)}{\bar \S(0) + \td \Sigma(0,\sin)}\bigr)^{1/\alpha}
, 
\]
for all $t\in [0,T)$. Since the perturbation $\td \S$ satisfies $\td \S(0,s) \to 0$ as $s\to \infty$ (equivalently, as $t\to T$), exponentially fast in light of \eqref{eq:solution:small}, and since $\omega_0(0) \neq 0$, we obtain the desired vorticity blowup as $t \to T$. In fact, the above identity also proves~\eqref{eq:blow:asym:d}. 
For a fixed $y \neq 0$, the asymptotic convergence claimed in~\eqref{eq:blow:asym:a}--\eqref{eq:blow:asym:b} follows from the ansatz~\eqref{ss-var:b} and the exponential decay of the perturbation $(\td \UU(\cdot,s), \td \S(\cdot,s) , \td \AA(\cdot,s))$ as the self-similar time $s\to \infty$ (equivalently, as $t\to T$), established in~\eqref{eq:solution:small}. 
\end{proof}

\subsection{The proof of Theorem~\ref{thm:non}}
The goal of this subsection is to prove Theorem~\ref{thm:non}.

Since the formula for $\td \VV_2$ (see~\eqref{eq:V2_form2:c}) involves the future of the solution $(\td \UU_1, \td \S_1)$, and since $\td \VV_2$ enters the evolution \eqref{eq:non_V:a} for $\td \VV_1$ through the nonlinear term, we cannot solve for the perturbation  $\td \VV_1$  directly. Instead, we reformulate \eqref{eq:non_V:a} as a fixed point problem. We fix the initial data $\td \VV_1(\sin) =(\td \UU_1, \td \S_1, \vA) |_{s=\sin} \in \cZ^{m+2}$ \eqref{norm:Z}  sufficiently smooth, and sufficiently small such that \eqref{eq:IC:small} holds. 
We define the spaces $Y_1$ and $Y_2$, which capture the decay of the solutions in time 
\begin{subequations}
\label{norm:fix}
\begin{align}
 \| (\td \UU_1,\td \S_1, \vA)\|_{ Y_1}  
 & \teq \sup_{s\geq \sin} e^{ \lam_1 (s-\sin)}   
 \bigl(\| (\td \UU_1(s), \td \S_1(s) )\|^2_{ \cW^{m+1}} + \| \vA(s) \|^2_{ \cW_A^{m+1}}  \bigr)^{1/2} , 
 \label{norm:fix:a} \\
 \| (\td \UU_1,\td \S_1, \vA)\|_{Y_2}  
 & \teq \sup_{ s \geq \sin} e^{ \lam_1 (s-\sin) }   
\bigl(\| (\td \UU_1(s), \td \S_1(s) )\|^2_{ \cX^{m}} + \| \vA(s) \|^2_{ \cX^{m}_A} \bigr)^{1/2} .
\label{norm:fix:b}
\end{align}
\end{subequations}
Showing $(\td \UU_1,\td \S_1, \vA) \in Y_i$ for $i\in\{1,2\}$ implies that suitable norms of $\td \UU_1(s), \td \S_1(s), \vA(s)$ decay with a rate $e^{-\lam_1 s}$ as $s\to \infty$; recall cf.~\eqref{eq:decay_para} that $\lambda_1 = \frac{9}{10} \lambda$. During the proof we sometimes abuse notation to write 
\beq\label{norm:fix2}
  \| (\td \UU_1, \td \S_1)\|_{Y_2} =   \| (\td \UU_1, \td \S_1, 0)\|_{Y_2}, 
\eeq
i.e., when the $\AA$-component is irrelevant for an estimate.

Next, we define an operator $\cT$ (see~\eqref{eq:map_T}), whose fixed point (see~\eqref{eq:the:fixed:point}) is the desired solution of~\eqref{eq:non_V:a}. We remark that throughout the remainder of this proof, we distinguish the $(\UU,\S)$-components of an input of a map  (e.g.~$\cT$, or $\cT_2$) by variables with a ``hat'' (e.g.~$( \wh \UU_1, \wh \S_1 )$), and the output of these maps by variables with a ``tilde'' (e.g.~$(\td \UU_1, \td \S_1)$). With this notational convention in place, the two-step process is:
\begin{itemize}
\item first, for  $( \wh \UU_1, \wh \S_1 ) \in Y_2$, we define 
\begin{equation}
\td \VV_2 = \cT_2( \wh \UU_1, \wh \S_1),
\label{eq:non_fix:V2}
\end{equation} 
where the linear map $\cT_2$ is defined by \eqref{eq:T2:def}, via~\eqref{eq:V2_form2};
\item second, we define $\td \VV_1 = (\td \UU_1, \td \S_1, \AA)$ as the solution\footnote{
To construct the solution of~\eqref{eq:non_fix} locally in time, we use an iterative scheme, and arguments similar to those in Section~\ref{sec:lin_semi}. 
For a fixed $\td \VV_2 = \cT_2(\hat \UU_1, \hat \S_1)$, we first construct a $C_c^{\infty}$ solution to an $\e$-regularized version of~\eqref{eq:non_fix}, in which the operator $\cK_{m}(\cdot)$ and the term $\td \VV_2$ are replaced by $C^{\infty}_c$ approximations $\cK_{m,\e}(\cdot)$ and $\td \VV_{2, \e}$. We start with initial data $G_0 = \td \VV_{1,\e}|_{s=\sin} \in C_c^{\infty}$ and iteratively for $n\geq 0$  construct $G_{n+1}$ by solving
\[
\pa_s G_{n+1} = \cL_V (G_{n+1}) - \cK_{m, \e} (G_{n+1}) + \wh \cN_{V}( G_n + \td \VV_{2,\e}, G_{n+1} + \td \VV_{2,\e}) ,
\qquad
G_{n+1}|_{s=\sin} = G_0,
\]
where we rewrite the quadratic nonlinearity $\cN_V(\cdot)$ from \eqref{eq:non} as the bilinear form $\wh \cN_V(\cdot, \cdot)$ with the derivative acting on the second entry (e.g.~$\wh \cN_U( (\UU,\S,\AA), (\UU^\prime,\S^\prime,\AA^\prime)) = - \UU \cdot \nabla \UU^\prime - \alpha \S \nabla \S^\prime - \AA \cdot \nabla \AA^\prime$). Clearly $\wh \cN_{V}(\td \VV, \td \VV) = \cN_V(\td \VV)$.
The above equation is a linear symmetric hyperbolic system for $G_{n+1}$, which can be solved following the arguments in Section~\ref{sec:lin_semi}.
By proving convergence as $n\to \infty$ of the Picard iterates $\{G_n\}_{n\geq 0}$, and afterwards taking the regularization parameter $\e\to 0$, similarly to  Section~\ref{sec:lin_semi} we obtain the local existence of a $L^\infty_s (\cZ^{m+1})$ solution, where we recall cf.~\eqref{eq:non_fix} that $\cZ^i = \cX^{i} \times \cX^{i} \times \cX^{i}_A$. We emphasize that for initial data $\td \VV_{1,in} \in \cZ^{m+2}$, e.g.~for the initial data from Theorem~\ref{thm:non}, since $\cK(\cdot) : \cX^0 \to \cX^{m+3}, \td \VV_2 \in \cX^{m+3} $ (item (c) in Proposition \ref{prop:compact} and Lemma \ref{lem:T2}), the above described construction procedure via $C^\infty_c$ approximations also gives the local existence of a solution with this higher $\cZ^{m+2}$ regularity, so that locally in time $\cL(\td \UU_1,\td \S_1) \in \cX^{m+1}$ and $ \cL_A(\td \AA) \in \cX_A^{m+1}$.
\label{foot:footnote:details}
}
of a modified version of~\eqref{eq:non_V:a}, namely
\begin{subequations}
\label{eq:non_fix}
\begin{align}
 \pa_s \td \VV_1 &= \cL_V (\td \VV_1) - \cK_m(\td \VV_1) + \cN_V(  \td \VV_1 + \cT_2( \wh \UU_1,  \wh \S_1 )), \\
\td \VV_1 |_{s = \sin} &= (\td \UU_1(\sin), \td \S_1(\sin), \vA(\sin) )
\quad 
\mbox{as in Theorem~\ref{thm:non}},
\end{align}
\end{subequations}
where we recall the notation $\cK_m(\td \VV_1)=(\cK_m(\td \UU_1,\td \S_1),0)$, and $ \td \VV_1 + \cT_2( \wh \UU_1,  \wh \S_1 ) = ( (\td \UU_1,\td \S_1) + \cT_2( \wh \UU_1,  \wh \S_1 ), \AA)$; that is, neither $\cK_m$ nor $\cT_2$ act on the $\AA$-component.  
\end{itemize} 
Concatenating the two steps given above defines a map which takes as input $(\wh \UU_1,\wh \S_1)$ and outputs the solution of \eqref{eq:non_fix}: 
\beq\label{eq:map_T}
\bigl( (\td \UU_1, \td \S_1), \AA\bigr)
= \td \VV_1
\stackrel{\eqref{eq:non_fix}}{=} \cT( \wh \UU_1, \wh \S_1)
=\bigl( \cT_{U,\S}( \wh \UU_1, \wh \S_1), \cT_{A}( \wh \UU_1, \wh \S_1) \bigr)
.
\eeq
Denoting by  $\cT_{U,\S}$ the restriction of $\cT$ to the $(\UU, \S)$-components, we have thus reformulated the system \eqref{eq:non_V}
as a {\em fixed point problem}: find $(\td\UU_1,\td \S_1)$ such that 
\begin{equation}
 (\td \UU_1, \td \S_1) = \cT_{U,\S}( \td \UU_1, \td \S_1 ),
 \label{eq:the:fixed:point}
\end{equation}
with $\AA$ and $(\td \UU_2,\td \S_2)$ computed as $\cT_A(\td \UU_1, \td \S_1 )$ and $\cT_2( \td \UU_1, \td \S_1 )$, respectively. 

The proof of Theorem~\ref{thm:non} reduces to establishing that the operator  $\cT_{U,\S}$ is a contraction with respect to the norm in~\eqref{norm:fix2}, in a small vicinity of the zero state, characterized by the smallness parameters
\beq\label{eq:para_fix}
 \d_Y = \d^{2/3}, \qquad \d \ll_m 1,
\eeq
as in the statement of Theorem~\ref{thm:non}. 
The proof of Theorem~\ref{thm:non} is broken down in two steps, according to Proposition~\ref{prop:onto} (which shows that the map $\cT_{U,\S}$ maps the ball of radius $\d_Y$ in $Y_2$ into itself), and Proposition~\ref{prop:contra} (which shows that $\cT_{U,\S}$ is a contraction for the topology $Y_2$).
 
\begin{proposition}\label{prop:onto}
There exists a positive $\d_0 \ll_m 1$ such that for any $\d < \d_0$ and any $( \wh \UU_1, \wh \S_1) \in Y_2$ with $\| ( \wh \UU_1, \wh \S_1)\|_{Y_2} < \d_Y$, we have 
\[
 \|\cT( \wh \UU_1, \wh \S_1) \|_{Y_1} < 2 \d, 
 \qquad     
 \|\cT_{U,\S}( \wh \UU_1, \wh \S_1) \|_{Y_2} < \d_Y,
 \qquad 
 \|\cT_2(\wh \UU_1,  \wh \S_1)(s) \|_{\cX^{m+ 3}} \les_m \d_Y e^{-\eta_s (s-\sin)} ,
\]
for all $s\geq \sin$.
\end{proposition}

\begin{proposition}\label{prop:contra}
There exists a positive $ \d_0\ll_m 1$ such that for any $\d < \d_0$ and any pairs $( \wh \UU_{1,a}, \wh \S_{1,a})$, $( \wh \UU_{1,b}, \wh \S_{1,b}) \in Y_2$ with $\| ( \wh \UU_{1,a}, \wh \S_{1,a})\|_{Y_2} < \d_Y$ and $\| ( \wh \UU_{1,b}, \wh \S_{1,b})\|_{Y_2} < \d_Y$, we have 
\[
 \|\cT_{U,\S}(\wh \UU_{1, a}, \wh \S_{1,a}) - \cT_{U,\S}( \wh \UU_{1,b}, \wh \S_{1,b}) \|_{Y_2} < \tfrac{1}{2} 
  \|( \wh \UU_{1, a}, \wh \S_{1,a}) - ( \wh \UU_{1,b}, \wh \S_{1,b}) \|_{Y_2}  .
\]
\end{proposition}

From Proposition~\ref{prop:onto} and Proposition~\ref{prop:contra} we directly obtain:

\begin{proof}[Proof of Theorem~\ref{thm:non}]
Propositions~\ref{prop:onto} and~\ref{prop:contra}   allow us to apply a Banach fix-point theorem for the operator $\cT_{U,\S}$, in the ball of radius $\d_Y$ around the origin in $Y_2$; this results in a unique fixed point $(\td \UU_1,\td \S_1)$ in this ball, as claimed in~\eqref{eq:the:fixed:point}. Upon defining $\AA\teq \cT_A(\td \UU_1, \td \S_1 )$ and $(\td \UU_2,\td \S_2) \teq \cT_2( \td \UU_1, \td \S_1 )$, by construction we have that $\td \VV_1 = (\td \UU_1, \td \S_1, \vA)$ solves \eqref{eq:non_V:a} and $\td \VV_2 = (\td \UU_2,\td \S_2,0)$ solves \eqref{eq:non_V:b}. In view of the definition of the $Y_1$ norm in~\eqref{norm:fix:a} and the bound $\|\cT( \td \UU_1, \td \S_1) \|_{Y_1} < 2 \d$ (which follows from the first estimate in Proposition~\ref{prop:onto}), we deduce that~\eqref{eq:solution:small:a} holds. Similarly, the third estimate in Proposition~\ref{prop:onto} (applied to $(\td \UU_1,\td \S_1)$) yields~\eqref{eq:solution:small:b}, thereby concluding the proof of Theorem~\ref{thm:non}. 
\end{proof}

The remainder of this subsection is dedicated to the proof of Propositions~\ref{prop:onto} and~\ref{prop:contra}. 
In subsection~\ref{sec:T2}, we obtain suitable estimates for the linear map $\cT_2$; in particular, in Lemma~\ref{lem:decay_Km} we demonstrate a smoothing effect  for $\td \VV_2$, which allows us to overcome the loss of a space derivative due to the term$ \na \td \VV_2$, present in the first equation of \eqref{eq:non_V}. In subsection~\ref{sec:fix_pt} we prove Proposition~\ref{prop:onto}, while in subsection~\ref{sec:contra}, we prove Proposition \ref{prop:contra}; here we emphasize that due to the transport term $\UU \cdot \na \UU$, estimating the difference $\cT_{U,\S}(\wh \UU_{1,a}, \wh \S_{1,a}) - \cT_{U,\S}( \wh \UU_{1,b}, \wh \S_{1,b} )$ is made possible because in~\eqref{norm:fix} we have chosen $Y_2$ (modeled on $\cX^m$ and $\cX^m_A$) to be less regular than $Y_1$ (modeled on $\cX^{m+1}$ and $\cX^{m+1}_A$) .

\subsubsection{Estimates on $\cT_2$}\label{sec:T2}
Recall the decomposition~\eqref{eq:dec_X} of $\cX^m$ into stable and unstable modes. In light of definitions~\eqref{eq:V2_form2:b} and~\eqref{eq:V2_form2:c}, we establish the following decay and smoothing estimates for the stable and unstable parts of $\cK_m$:
\begin{lemma}\label{lem:decay_Km}
For $f \in \cX^{m}$, we have 
\[
\bal
 \| e^{\cL (s-\sin)} \Pi_{\mathsf s} \cK_m f \|_{\cX^{m+3}} 
 & \les_m  e^{-\eta_s (s-\sin)}  \| f \|_{\cX^{m}}, \\
 \| e^{ - \cL (s-\sin)} \Pi_{\mathsf u} \cK_m f \|_{\cX^{m+3}} 
& \les_m  e^{\eta (s-\sin)}  \| f \|_{\cX^{m}},
\eal
\]
for all $s\geq \sin$, where $\eta$ and $\eta_s$ are as in~\eqref{eq:decay_para}.
\end{lemma}

\begin{proof}
Denote $F(s) = e^{\cL (s-\sin)} \Pi_{\mathsf s} \cK_m f$. Since $\Pi_{\mathsf u}$ projects onto the finite-dimensional space $\cX_{\mathsf u}^m$, defined in~\eqref{eq:dec_Xu}, which is spanned by $\cX^{\infty}$ smooth functions (see Section~\ref{sec:smooth_unstab}), we have the equivalence 
\beq\label{eq:equiv_Xu}
 \|\Pi_{\mathsf u} g \|_{\cX^n} \asymp_{n, m} \| \Pi_{\mathsf u} g \|_{\cX^m} , \qquad \forall g \in \cX^m, \ n \geq m. 
\eeq
Using \eqref{eq:equiv_Xu}, the fact that $\Pi_{\mathsf s} = {\rm Id} - \Pi_{\mathsf u}, \Pi_{\mathsf u} : \cX^m \to \cX^m$,  and the smoothing property $\cK_m : \cX^{m} \to \cX^{m+3} \subset \cX^m$ (which follows from item (c) in Proposition~\ref{prop:compact}), we obtain 
\begin{align*}
 \| F(\sin) \|_{\cX^{m+3}} 
 & \leq  \| \cK_m f\|_{\cX^{m+3}} + \| \Pi_{\mathsf u} \cK_m f\|_{\cX^{m+ 3}} \\
 &\les_m \| f \|_{\cX^{m}} +  \| \Pi_{\mathsf u} \cK_m f\|_{\cX^{m}}
 \les_m \| f \|_{\cX^{m}}  + \| \cK_m f \|_{\cX^m}   
 \les_m \| f \|_{\cX^{m}} .
\end{align*}
Using the decay property~\eqref{eq:decay_stab} of the semigroup and the fact that $F(\sin) \in \cX^m_{\mathsf s}$,  we deduce
\[
\| F(s)\|_{\cX^m} \leq C_m e^{- \eta_s (s-\sin)} \| F(\sin)\|_{ \cX^m}
\leq  C_m e^{- \eta_s (s-\sin)} \| f \|_{\cX^{m}}.
\]
Combining the above bound with the energy estimates on $\cX^{m+3}$, which follow from Theorem~\ref{thm:coer_est} (more specifically,~\eqref{eq:coer_est}), we obtain
\[
\quad 
\tfrac{1}{2}\tfrac{d}{ds}  
\| F(s) \|^2_{\cX^{m+ 3}} 
\leq - \lam 
\| F(s) \|_{\cX^{m+ 3}}^2 + \bar C  \| F(s) \|_{\cX^0}^2
\leq - \lam 
\| F(s) \|_{\cX^{m + 3}}^2 + \bar C C_m^2  e^{- 2\eta_s (s-\sin)} \| f \|_{\cX^{m}}^2.
\]
Since $\eta_s = \frac 35 \lambda < \lam$, integrating the above estimate in time we deduce the first inequality of Lemma~\ref{lem:decay_Km}. 

From Proposition~\ref{prop:decom} we have that  $\cX_{\mathsf u}^m$ is $\cL$-invariant and $\cL |_{\cX_u^m}$ is bounded. 
Using \eqref{eq:equiv_Xu} and \eqref{eq:decay_unstab}, we establish
\[
 \| e^{ - \cL (s-\sin)} \Pi_{\mathsf u} \cK_m f \|_{\cX^{m+ 3}} 
\les_m  \| e^{ - \cL (s-\sin)} \Pi_{\mathsf u} \cK_m f \|_{\cX^{m}} 
\les_m e^{\eta (s-\sin)} \| \cK_m f \|_{\cX^{m}} 
\les e^{\eta (s-\sin)} \| f \|_{\cX^{m}} ,
\]
which is the second inequality claimed in Lemma~\ref{lem:decay_Km}.
\end{proof}

Using Lemma~\ref{lem:decay_Km} and the fact that $\cL$ generates a semigroup,  we obtain a direct estimate for the operator $\cT_2$, as defined in~\eqref{eq:V2_form2}. We recall from~\eqref{norm:fix:b} and~\eqref{norm:fix2} the definition of the norm $Y_2$.
\begin{lemma}\label{lem:T2}
For $(\wh \UU_1, \wh \S_1) \in Y_2$, for all $s\geq \sin$ we have
\[
 \|\cT_2( \wh \UU_1, \wh \S_1 )(s) \|_{\cX^{m+3}} \les_m e^{-\eta_s (s-\sin)} 
 \|(\wh \UU_1, \wh \S_1)\|_{Y_2}. 
\]
\end{lemma}

\begin{proof}
Denote $M =  \|(\td \UU_1, \td \S_1)\|_{Y_2}$, and recall from~\eqref{eq:V2_form2:a} that $\cT_2( \wh \UU_1, \wh \S_1 )(s)$ is given by the sum of three terms. We also recall from~\eqref{eq:decay_para} that $\eta_s < \eta < \lam_1$.
Applying Lemma~\ref{lem:decay_Km} to $\td \VV_{2, {\mathsf s}}$ (as defined in~\eqref{eq:V2_form2:b}),  and to $\td \VV_{2, {\mathsf u}}$ (as defined in~\eqref{eq:V2_form2:c}), we obtain
\begin{subequations}
\begin{align}
\|\td \VV_{2, {\mathsf s}}(s)\|_{\cX^{m+ 3}} 
& \les_m \int_{\sin}^s e^{-\eta_s(s-s^\prime)}  
\| (\wh \UU_1(s^\prime), \wh \S_1(s^\prime) ) \|_{\cX^{m}}  ds^\prime \notag\\
&\les_m M  \int_{\sin}^s e^{- \eta_s (s-s^\prime)} e^{-\lam_1 (s^\prime-\sin)} d s^\prime
\les_m  M e^{-\eta_s (s-\sin)} , 
\label{eq:junk:label:10a}\\
\| \td \VV_{2, {\mathsf u}}(s)\|_{\cX^{m+ 3}} 
& \les_m  \int_s^{\infty} e^{\eta(s^\prime-s)} \| (\wh \UU_1(s^\prime), \wh \S_1(s^\prime) ) \|_{\cX^{m}}  d s^\prime \notag\\
&\les_m  M\int_s^{\infty} e^{\eta(s^\prime-s)}  e^{-\lam_1 (s^\prime-\sin)} d s^\prime 
\les_m M e^{-\lam_1 (s-\sin)}.
\label{eq:junk:label:10b}
\end{align}
The second estimates in particular implies  $ \| \td \VV_{2, {\mathsf u}}(\sin)\|_{\cX^{m+3}} \les_m M$. Denote 
$g =  \td \VV_{2, {\mathsf u}}(\sin) (1 - \chi(\frac{y}{8 R_4}))$, corresponding to the third term on the right side of~\eqref{eq:V2_form2:a}. By definition $\supp(g) \cap B(0, 5 R_4) = \emptyset$. Since $(1 - \chi(\frac{y}{8 R_4}))$ is a smooth function and since  
$\na^i(1 - \chi(\frac{y}{8 R_4}))$ has compact support for each $i\geq 1$, using the Leibniz rule and Proposition \ref{prop:far_decay}, we obtain 
\begin{equation}
\label{eq:junk:label:10c}
\| e^{\cL (s-\sin)} g \|_{\cX^{m+ 3}}
\leq e^{-\lam (s-\sin)} \| g \|_{\cX^{m+ 3}} 
\les_m e^{-\lam (s-\sin)} \| \td \VV_{2, {\mathsf u}}(\sin) \|_{\cX^{m +3}}
\les_m e^{-\lam (s-\sin)} M .
\end{equation}
\end{subequations}
Combining the bounds obtained in~\eqref{eq:junk:label:10a}--\eqref{eq:junk:label:10c} with the definition of the operator $\cT_2$ via \eqref{eq:V2_form2:a}, and with the ordering~\eqref{eq:decay_para} of the parameters,  we complete the proof of the lemma.
\end{proof}

\subsubsection{Proof of Proposition \ref{prop:onto}}\label{sec:fix_pt}

Before proving Proposition~\ref{prop:onto}, we record the following product estimate for the nonlinear terms appearing on the right side of~\eqref{eq:non_V:a}. 

\begin{lemma}\label{lem:non_main}
Let $\cN_{U}, \cN_{\S}, \cN_{A}$ be the nonlinear terms defined in \eqref{eq:non}.
For any $ k \geq m_0$, we have 
\[
\bal
 \left\| \vp_{2k} \vp_g^{1/2}  \B( 
 \D^k \cN_{U}(\td\UU, \td\S, \vA) 
+ \td\UU \cdot \na \D^k \td\UU 
+ \al \td \S \na \D^k \td \S
+ \vA \cdot \na \D^k \vA  
\B) \right\|_{L^2} 
& \les_k \| (\td \UU, \td \S, \vA) \|_{\cX^k}^2, \\ 
\left\| \vp_{2k}  \vp_g^{1/2} \B( 
\D^k \cN_{\S}(\td \UU, \td \S, \vA)    
+ \td \UU \cdot \na \D^k \td \S 
+ \al \td \S \div (\D^k \td \UU) 
\B) \right\|_{L^2} & \les_k \| (\td \UU,\td \S) \|_{\cX^k}^2 , \\
\left\| \vp_{2k} \vp_g^{1/2}  \B( 
\D^k \cN_{A}(\td \UU,\td  \S, \vA) 
+ \td \UU \cdot \na \D^k \vA 
+ \vA \cdot \na \D^k \td \UU
\B) \right\|_{L^2} & \les_k \| (\td \UU,\vA) \|_{\cX^k}^2.
\eal
\]
\end{lemma}

Note that since $\vp_g\les \vp_A$, we may a-posteriori bound the norm $\|\AA\|_{\cX^k}$ appearing on the right side of the above estimates with $\|\AA\|_{\cX^k_A}$.

\begin{proof}
We estimate  a typical term, for instance $\td \UU \cdot \na \td \UU$, appearing in $\cN_{U}$. Using the Leibniz rule, recalling the definitions of the weights $\vp_{2k}$ (see~\eqref{eq:phi:m:def} with \eqref{eq:wg_asym}) and $\vp_g$ (see~\eqref{eq:vp:g:def}), and using Lemma~\ref{lem:prod}, we have
\begin{equation*}
\left\| \vp_{2k} \vp_g^{1/2} \Bigl( \D^k( \td \UU \cdot \na \td \UU) -  \td \UU \cdot \na \D^k \td \UU  \Bigr) \right\|_{L^2}  \les_k \sum_{ 1 \leq i \leq 2 k} \left\| \vp_{2k} \vp_g^{1/2} |\na^i \td \UU| \cdot |\na^{2k+1-i} \td \UU | \right\|_{L^2} \les_k \| \td \UU \|_{\cX^k}^2. 
\end{equation*}
All other terms appearing in the nonlinear terms $\cN_U,\cN_\S$ and $\cN_A$ (see~\eqref{eq:non}) are estimated similarly. 
\end{proof}

Next, we prove Proposition \ref{prop:onto}. Recall the different notations from the beginning of Section~\ref{sec:non}, in particular~\eqref{eq:init}, and \eqref{eq:T2:def}.
Per the assumption of Proposition~\ref{prop:onto}, let $\| ( \wh \UU_1, \wh \S_1) \|_{Y_2} < \d_Y$ for some $\d_Y$ small to be determined. Define $\td \VV_2$ using \eqref{eq:non_fix:V2}, and then define $\td \VV_1$ as the solution of~\eqref{eq:non_fix}. Denote
\begin{subequations}\label{eq:non_E}
\begin{align}
  \td \VV_2 & =  \cT_2( \wh \UU_1, \wh \S_1 ), 
\quad  \td \VV = \td \VV_1 +  \td \VV_2, \quad  \td \UU = \td \UU_1 +  \td \UU_2, \quad \td \S = \td \S_1 + \td  \S_2 , \\
  \wh E & = \| ( \wh \UU_1, \wh \S_1) \|_{Y_2} < \d_Y, \\ 
E_{m+1}(s) &\teq \bigl( \| (\td \UU_1(s), \td \S_1(s)) \|_{\cW^{m+1}}^2 +  \| \vA(s) \|_{\cW_A^{m+1}}^2\bigr)^{1/2},   \\
  \cE(s)  & = E_{m+1}(s) + \|  \td \VV_2(s) \|_{\cX^{m+2}}.
\end{align}
\end{subequations}
Then, by the above definitions and~\eqref{norm:Wk} we have 
\beq\label{eq:non_err}
\| (\td \UU,\td \S) \|_{\cX^{m+1}}  + \| \AA \|_{\cX_A^{m+1}} + \| \td \VV_2\|_{\cX^{m+2}} 
\les_m \| (\td \UU_1,\td \S_1) \|_{\cX^{m+1}}  + \| \AA \|_{\cX_A^{m+1}} + \| \td \VV_2\|_{\cX^{m+2}}  
\les_m \cE .
\eeq
Note that since $\vp_g\les \vp_A$, the bound~\eqref{eq:non_err} also gives $\| \AA \|_{\cX^{m+1}}  \les_m \cE$, an estimate we use occasionally.

The term $\td \VV_2$ is already bounded in light of Lemma~\ref{lem:T2}. In order to estimate $\td \VV_1 = (\td \UU_1,\td \S_1,\AA)$, we performing $\cW^{m+1}$-energy estimates on \eqref{eq:non_fix}, to obtain 
\begin{subequations}
\label{eq:non_Y1}
\begin{align}
\tfrac{1}{2} \tfrac{d}{ds} E_{m+1}^2
&= \tfrac{1}{2} \tfrac{d}{ds} \bigl( \| (\td \UU_1, \td \S_1) \|_{\cW^{m+1}}^2 
 + \| \vA \|_{\cW_A^{m+1}}^2  \bigr)
  = I_{\cL} + I_{\cN}  , 
  \label{eq:non_Y1:a}\\
I_{\cL}& \teq \la (\cL - \cK_m) ( \td \UU_1, \td \S_1 ), (\td \UU_1, \td \S_1) \ra_{\cW^{m+1}} 
 + \la \cL_A (\vA), \vA \ra_{\cW_A^{m+1}} , 
 \label{eq:non_Y1:b}\\
 I_{\cN} & \teq \la (\cN_U(\td \VV),\cN_\S(\td \VV)), (\td \UU_1,\td \S_1) \ra_{\cW^{m+1}}
 + \la \cN_A(\td \VV), \vA \ra_{\cW^{m+1}_A}. 
 \label{eq:non_Y1:c}
\end{align}
\end{subequations}
For the linear contributions appearing in~\eqref{eq:non_Y1:b}, using \eqref{eq:coer_Wk}, we obtain the dissipative bound\footnote{Here we use that the solution $\td \VV_1$ of \eqref{eq:non_fix} has sufficient regularity to justify the applicability of~\eqref{eq:coer_Wk}; see  Footnote~\ref{foot:footnote:details}.}
\beq\label{eq:non_lin}
\bal
 I_{\cL} \leq & - \lam_1 \| (\td \UU_1, \td \S_1) \|_{\cW^{m+1}}^2
 - \lam_1  \| \vA \|_{\cW_A^{m+1}}^2 
 = -\lam_1 E_{m+1}^2. 
 \eal
\eeq
Thus, it is left to estimate the contribution to \eqref{eq:non_Y1:a} from the nonlinear terms in~\eqref{eq:non_Y1:c}.
 
\paragraph{\bf{Estimate of nonlinear terms.}}
We treat the nonlinear terms in~\eqref{eq:non_Y1:c} perturbatively. Our goal is to prove
\beq\label{eq:non_est}
 |I_{\cN}| \les_m \cE^2 E_{m+1} .
\eeq 

We first estimate the contribution to~\eqref{eq:non_Y1:c} from the nonlinear term $\cN_{U}(\td \VV)$, focusing on the terms containing most derivatives, with respect to the $\cX^{m+1}$-inner product. Using Lemma \ref{lem:non_main}, the Cauchy-Schwarz inequality, and the bound~\eqref{eq:non_err} to absorb the lower-order terms, we get 
\[
\bal
 I_{\cN, U} 
 &\teq   \int \vp_g \vp_{2m+2}^2 \D^{m+1} \cN_{U}(\td \VV) \cdot \D^{m+1} \td \UU_1  \\
 &  = - \int \vp_g \vp_{2m+2}^2 \B( \td \UU  \cdot 
 \na \D^{m+ 1}  \td \UU
  + \al \td \S  \na \D^{m+1} \td \S
 + \vA \cdot \na \D^{m+1}  \vA  \B) \cdot \D^{m+1} \td \UU_1
+ \mathcal{O}_m( \cE^2  E_{m+1} ).
 \eal
\]
Since $\vp_{k} = \vp_1^k, \vp_1 \asymp \la y \ra \gtr 1$, using Lemma~\ref{lem:prod} (with $i=0$, $j=2m+3$, $m\to m+1$, $n=m+2$), and the bound~\eqref{eq:non_err}, for $G \in \{ \td \UU, \td \S, \vA\}$ and $F \in \{\td \UU_2, \td  \S_2\}$ we get 
\[
 \| \vp_{2m+2} \vp_g^{1/2} G  \na \D^{m+1}  F \|_{L^2} \les_m 
  \| \la y \ra^{2m+2} \vp_g^{1/2} G  \na \D^{m+1}  F \|_{L^2} 
\les_m
\| G \|_{\cX^{m+1}} \|  F \|_{\cX^{m+2}} \les_m \cE^2.
\]
Thus, combining the two displayed estimates above, we can reduce bounding $I_{\cN, U}$ to bounding  the main terms, which involve $\na \D^{m+1} \td \VV_1$; that is, we have
\begin{equation*}
I_{\cN, U} = 
- \int \vp_g \vp_{2m + 2}^2  \B( \td \UU  \cdot 
 \na  \D^{m+ 1}  \td \UU_1   
  + \al \td \S  \na \D^{m+1} \td \S_1 
+ \vA \cdot \na \D^{m+1}  \vA  \B) \cdot  \D^{m+1} \td \UU_1 + \mathcal{O}_m( \cE^2 E_{m+1}).
\end{equation*}
In order to estimate the contribution of the transport term $\td \UU  \cdot \na \D^{m+ 1}  \td \UU_1 $, using Lemma~\ref{lem:wg}, the definition of $\vp_{2m+2}$ in~\eqref{eq:phi:m:def}, the definition of $\vp_g$ in~\eqref{eq:vp:g:def}, estimate~\eqref{eq:non_err}, and the embedding~\eqref{eq:Xm_Linf}, we obtain
\begin{subequations}
\label{eq:non_IBP}
\begin{align}
| \na (\vp_{2m+2}^2 \vp_g) | 
& \les_m \vp_{2m+2}^2 \vp_g \la y \ra^{-1}, \\
| \na (\vp_{2m+2}^2 \vp_g G) |
& \les_m \vp_{2m+2}^2 \vp_g \bigl(\la y \ra^{-1} |G| + |\nabla G| \bigr)
\les_m \vp_{2m+2}^2 \vp_g \cE,
\end{align}
\end{subequations}
for any $G \in \{\td \UU, \td \S, \AA\}$. 
Since $( \td \UU \cdot \na) \D^{m+1} \td \UU_1 \cdot \D^{m+1} \td \UU_1 =\tfrac{1}{2}( \td \UU \cdot \na) |\D^{m+1} \td \UU_1|^2$, using integration by parts and~\eqref{eq:non_IBP} we obtain
\[
 \B| \int \vp_g\vp_{2m+2}^2 (\td \UU \cdot \na) \D^{m+1} \td \UU_1 \cdot  \D^{m+1}  \td \UU_1 \B| 
 \les_m \cE \int \vp_{2m+2}^2 \vp_g |\D^{m+1} \td \UU_1|^2 
 \les_m  \cE E_{m+1}^2 . 
\]

For the lower order $\la \cdot , \cdot \ra_{\cX^{m}}$-terms (which already contain the $\la \cdot , \cdot \ra_{\cX^{0}}$-contribution) present in the inner product $\la \cdot , \cdot \ra_{\cW^{m+1}}$ (see~\eqref{norm:Wk}), using Lemma~\ref{lem:prod} and estimate \eqref{eq:non_err}, we get
\[
 \| \cN_U(\td \VV) \|_{\cX^m} \les_m \| \td \VV \|_{\cX^{m+1}}^2 
 \les_m  \cE^2.
\]

Upon recalling the definition~\eqref{norm:Wk} of the inner product $\cW^{m+1}$, and the definitions of the inner products $\cX^{m+1}$ and $\cX^m$ in \eqref{norm:Xk}, we may combine the above estimates to obtain
\begin{align}
 &\la \cN_U(\td \VV), \td \UU_1 \ra_{\cW^{m+1}} 
 = \e_{m+1} \mu_{m+1} I_{\cN, U}
+ \mathcal{O}_m (  \| \cN_U(\td \VV) \|_{\cX^m} \| \td \UU_1 \|_{\cX^m} ) \notag \\
 &= - \e_{m+1} \mu_{m+1}  \int \vp_g \vp_{2m}^2  \bigl( 
   \al  \td \S \na \D^{m+1} \td \S_1 
+ (\vA \cdot \na) \D^{m+1}  \vA  \bigr) \cdot \D^{m+1} \td \UU_1     + \mathcal{O}_m( \cE^2 E_{m+1}). 
\label{eq:non_Nu}
\end{align}

In a similar fashion, for the contribution to~\eqref{eq:non_Y1:c} from the nonlinear term $\cN_{\S}(\td \VV)$, by using Lemma~\ref{lem:non_main} and the same arguments/estimates as above, we may obtain
\beq\label{eq:non_Ns}
\bal
\la \cN_{\S}(\td \VV), \td \S_1 \ra_{\cW^{m+1}} 
& = -\e_{m+1} \mu_{m+1} 
\! \int \! \vp_g \vp_{2m+2}^2  \bigl( \al \td \S \div( \D^{m+1} \td \UU_1) \bigr) \D^{m+1} \td \S_1  \!+\! \mathcal{O}_m( \cE^2 E_{m+1}) .
 \eal
 \eeq

It thus remains to estimate the contribution to~\eqref{eq:non_Y1:c} from the nonlinear term $\cN_A(\td \VV)$. Note that the inner products on $\cX^{m}_A$ and $\cX^m$, and hence $\cW^m_A$ and $\cW^{m}$,  differ only at the weighted-$L^2$ level, where the weight $\vp_g$ needs to be replaced by $\vp_A$; see definitions~\eqref{norm:Xk} and \eqref{norm:Wk}. At this weighted-$L^2$ level, using \eqref{eq:non_err} and the embedding~\eqref{eq:Xm_Linf}, we first obtain
\[
|\na \td \UU| + |y|^{-1} |\td  \UU| \les_m  \| \na \td \UU  \|_{L^{\infty}}
\les_m \cE.
\]
Then, using integration by parts and the bound $|\na \vp_A| \les |y|^{-1} \vp_A $ (see~\eqref{eq:wg_A}), we deduce
\[
\Bigl| 
\int \cN_A(\td \VV) \cdot \vA \vp_A
\Bigr|
\leq \int \B(  \B| \f{ \na \cdot (\td \UU \vp_A)}{2 \vp_A} \B| 
+ \| \na \td \UU \|_{L^{\infty}} \B) |\vA|^2 \vp_A 
\les \cE \| \AA\|_{\cX^0_A}^2
\les \cE  E_{m+1}^2.
\] 
For the terms in $\la \cN_{A}(\td \VV), \vA \ra_{\cW_A^{m+1}} $ which contain derivatives, the bounds and arguments are nearly identical to those of employed earlier to handle $\la \cN_U(\td \VV), \td \UU_1 \ra_{\cW^{m+1}} $. Thus, using the same arguments as those leading up to~\eqref{eq:non_Nu} and~\eqref{eq:non_Ns}, and by appealing to the estimate displayed above, we may show
 \begin{equation}\label{eq:non_Na} 
  \la \cN_{A}(\td \VV), \vA \ra_{\cW_A^{m+1}} 
 =  -\e_{m+1} \mu_{m+1} 
\! \int \! \vp_{2m+2}^2 \vp_g  \bigl(  (\vA \cdot \na) \D^{m+1} \td \UU_1 \bigr) \cdot \D^{m+1} \vA  + \mathcal{O}_m(\cE^2 E_{m+1}) .
\end{equation}

To conclude, we estimate the term $I_{\cN}$ defined in \eqref{eq:non_Y1:c} by combining the bounds~\eqref{eq:non_Nu}, \eqref{eq:non_Ns}, and~\eqref{eq:non_Na}. Using the identities
\begin{subequations}
\label{eq:good:vector:identities}
\begin{align}
 (\vA \cdot \na) \D^{m+1} \td \UU_1 \cdot \D^{m+1} \vA
 + 
 (\vA \cdot \na) \D^{m+1} \vA \cdot \D^{m+1} \td \UU_1
 &= (\vA \cdot \na) ( \D^{m+1} \td \UU_1 \cdot \D^{m+1} \vA ),  \\
  \al \td \S  \div(\D^{m+1} \td \UU_1) \D^{m+1} \td \S_1 
 + \al \td \S \na \D^{m+1} \td \S_1 \cdot \D^{m+1} \td \UU_1
  & = \al \td \S  \div ( \D^{m+1} \td \UU_1 \;   \D^{m+1} \td \S_1  ),
\end{align}
\end{subequations}
integration by parts, and the estimates \eqref{eq:non_IBP}, we establish 
\begin{align}
\label{eq:non_IBP2}
&|I_{\cN}| = \bigl| \la \cN_U(\td \VV), \td \UU_1 \ra_{\cW^{m+1}}
 + \la \cN_{\S}(\td \VV), \td \S_1 \ra_{\cW^{m+1}}
 + \la \cN_A(\td \VV), \vA \ra_{\cW^{m+1}_A} \bigr|
 \notag \\
  &\les_m  \int |\div( \vp_{2m+2}^2 \vp_g \vA )|    |\D^{m+1} \td \UU_1 \cdot \D^{m+1} \vA|
 + \alpha |\na( \vp_{2m+2}^2 \vp_g  \td \S) |  \cdot |\D^{m+1} \td \UU_1 \D^{m+1} \td \S_1|
 + \cE^2  E_{m+1} 
 \notag \\
  &\les_m  \cE \int \vp_{2m+2}^2 \vp_g \bigl( | \D^{m+1} \td \UU_1 \cdot  \D^{m+1} \vA| 
 + |\D^{m+1} \td \UU_1 \D^{m+1} \td \S_1| \bigr)  + \cE^2  E_{m+1}  
 \les_m \cE^2 E_{m+1}.
\end{align}
This concludes the proof of~\eqref{eq:non_est}.

\paragraph{\bf{Conclusion of the proof.}}
Combining~\eqref{eq:non_Y1:a}, \eqref{eq:non_lin}, and \eqref{eq:non_est}, we obtain the energy estimate
\begin{equation*}
 \tfrac{1}{2}\tfrac{d}{ds} E_{m+1}^2 \leq -\lam_1 E_{m+1}^2 +C_m \cE^2 E_{m+1}. 
\end{equation*}
Applying Lemma~\ref{lem:T2}, recalling that $\wh E = \| (\wh \UU_1, \wh \S_1)\|_{Y_2}$, and using the definition of $\cE$ in~\eqref{eq:non_E}, we obtain 
\begin{equation}
 \label{eq:junk:label:101}
\cE 
\leq 
E_{m+1} + \|\td \VV_2(s) \|_{\cX^{m+2}} 
= 
E_{m+1} + \|\cT_2( \wh \UU_1, \wh \S_1)(s)\|_{\cX^{m+2}} 
\leq 
E_{m+1} + C_m e^{-\eta_s (s-\sin)} \hat E.
\end{equation}
From the above two estimates we thus deduce
\begin{equation} 
 \tfrac{d}{ds} E_{m+1} \leq -\lam_1 E_{m+1} + \bar C_m E_{m+1}^2 +  \bar C_m  e^{-2 \eta_s (s-\sin)} \hat{E}^2,
 \label{eq:junk:label:100}
\end{equation}
for some constant $\bar{C}_m >0$. 

We consider the bootstrap assumption 
\beq\label{eq:non_boot}
E_{m+1}(s) < 2 \d e^{-\lam_1 (s-\sin)},
\qquad 
\forall s\geq \sin.
\eeq
Using the assumption~\eqref{eq:IC:small} in Theorem~\ref{thm:non}, we have $E_{m+1}(\sin) < \d$, so that~\eqref{eq:non_boot} holds at $s=\sin$.

We also recall from \eqref{eq:non_E} and~\eqref{eq:para_fix} that by assumption $\hat E < \d_Y = \d^{2/3}$. Thus, since $2 \eta_s - \lam_1 > 0 $ (see~\eqref{eq:decay_para}), by integrating~\eqref{eq:junk:label:100} in time, and using the bootstrap~\eqref{eq:non_boot}, we deduce that for all $s\geq \sin$:
\begin{align*}
E_{m+1}(s)
&\leq  
 \d + \bar C_m  \int_{\sin}^{\infty} \Bigl( 4 \d^2 e^{-\lam_1 (s-\sin)}  + 
 \d_Y^2  e^{- (2 \eta_s - \lam_1) (s-\sin)} \Bigr) d s
 \notag\\
&= \d + \bar C_m \bigl(4\lambda_1^{-1} \d^2 + (2\eta_s-\lambda_1)^{-1} \d^{4/3})
= \bigl( 1 + \tfrac{40}{9} \lambda^{-1} \bar C_m \d 
+ \tfrac{10}{3} \lambda^{-1} \bar C_m \d^{1/3} \bigr) \d.
\end{align*}
Choosing $\d_0 \ll_m 1$ small enough to ensure $\tfrac{40}{9}  \d_0 
+ \tfrac{10}{3}  \d_0^{1/3} \leq \lambda \bar C_m^{-1}$, we deduce from the above estimate that for any $\d < \d_0$,  the bootstrap assumption~\eqref{eq:non_boot} is closed. 

Recalling the definition of the $Y_1$ norm in~\eqref{norm:fix:a}, it is clear that \eqref{eq:non_boot} is equivalent to the first estimate in Proposition~\ref{prop:onto}. From~\eqref{norm:Wk} and~\eqref{norm:fix} it follows that $\|\cdot\|_{Y_2} \leq \|\cdot\|_{Y_1}$, and hence \eqref{eq:non_boot} gives
\begin{equation}
\label{eq:non_boot_res}
  \| \cT_{U,\S} ( \wh \UU_1, \wh \S_1) \|_{Y_2} 
  \leq \|( \td \UU_1, \td \S_1) \|_{Y_1}
  \leq  {\sup}_{s\geq \sin} e^{\lambda_1(s-\sin)} E_{m+1}(s)
  < 2 \d < ( 2 \d_0^{1/3}) \d_Y.
\end{equation}
The second bound in Proposition~\ref{prop:onto} thus follows as soon as $2 \d_0^{1/3} \leq 1$. The third estimate claimed in Proposition~\ref{prop:onto} was previously established in Lemma~\ref{lem:T2}.
This concludes the proof of Proposition~\ref{prop:onto}.

\subsubsection{Proof of Proposition \ref{prop:contra}}\label{sec:contra}
As in the assumption of the proposition, let $\{\wh \UU_{1, \ell}, \wh \S_{1, \ell}\}_{\ell \in \{a,b\}} \in Y_2$ be such that $\hat E_{\ell} = \| \wh \UU_{1,\ell}, \wh \S_{1, \ell} \|_{Y_2} < \d_Y$. According to \eqref{eq:non_fix:V2}, \eqref{eq:non_fix}, and \eqref{eq:map_T}, denote the associated solutions
\[
\td \VV_{2,\ell} = \cT_2( \wh \UU_{1, \ell}, \wh \S_{1, \ell} ) = (\td \UU_{2,\ell}, \td \S_{2,\ell}, 0),
\quad
\td \VV_{1,\ell} = \cT(\wh \UU_{1, \ell}, \wh \S_{1, \ell} ) =  (\td \UU_{1,\ell}, \td \S_{1,\ell}, \AA_\ell),
\quad
\td \VV_{\ell} = 
\td \VV_{1,\ell} + \td \VV_{2,\ell},
\]
for $\ell\in \{a, b\}$.
Here and throughout this proof, we use the subscript $\ell\in \{a, b\}$ to denote two different solutions, and we adopt the notation introduced \eqref{eq:non_E}; e.g.~$E_{m+1,\ell}$ and $\cE_{\ell}$, for the ``energies'' of these two solutions. Additionally, we denote the difference of two solutions by a $\Delta$-sub-index:
\begin{subequations}
\label{eq:nota_dif}
\begin{align}
&( \wh \UU_{1, \D}, \wh \S_{1,\D})  = ( \wh \UU_{1, a}, \wh \S_{1,a}) - ( \wh \UU_{1,b}, \wh \S_{1,b}) \\
&\td \VV_{1, \D} = \td \VV_{1,a } - \td \VV_{1, b},  
\quad  
\td \VV_{2, \D} =  \td \VV_{2, a} - \td \VV_{2, b},  
\quad 
\cN_{V, \D} = \cN_{V}(\td \VV_{a} ) - \cN_{V}(\td \VV_{b}) , \\
 &  E_{m, \D}(s) \teq ( \| (\td \UU_{1,\D},\td \S_{1,\D})  \|_{\cX^m}^2 
 +  \| \vA_{ \D} \|_{\cX^m_A} )^{1/2} , \quad  
 \cE_{\D}(s) \teq E_{m, \D}(s) + \| \td \VV_{2, \D}(s)\|_{\cX^{m+2}}.
\end{align}
\end{subequations}

With this notation, proving Proposition~\ref{prop:contra} amounts to showing that 
\begin{equation}
\label{eq:need:new:label:1}
\|(\td \UU_{1,\D},\td \S_{1,\D},\AA_{\D})\|_{Y_2} = {\sup}_{s\geq \sin} e^{\lambda_1(s-\sin)} E_{m,\D}(s) < \tfrac 12 \|(\wh \UU_{1,\D},\wh \S_{1,\D},0)\|_{Y_2}.
\end{equation}
In order to perform energy estimates for $E_{m,\D}$, we use~\eqref{eq:non_fix} and deduce that $\td \VV_{1,\D} = (\td \UU_{1,\D},\td \S_{1,\D},\AA_{\D})$ solves the equation
\begin{equation}
\label{eq:non_fix:diff}
 \pa_s \td \VV_{1,\D} = \cL_V (\td \VV_{1,\D}) - \cK_m(\td \VV_{1,\D}) + \cN_{V,\D},
 \qquad 
\td \VV_{1,\D} |_{s = \sin} = 0.
\end{equation}
Similarly to~\eqref{eq:non_Y1}, appealing to~\eqref{eq:dissip} for $\cL -\cK_m$ and to \eqref{eq:coer_estA} for $\cL_A$, we obtain 
\begin{subequations}
\label{eq:non_Y2}
\begin{align}
\tfrac 12 \tfrac{d}{ds} E_{m,\D}^2 
&= \tfrac{1}{2} \tfrac{d}{ds} \B( \| (\td \UU_{1,\D},\td \S_{1,\D}) \|_{\cX^{m}}^2 
 + \| \vA_{\D} \|_{\cX^{m}_A}^2  \B) \notag \\
 &= \la (\cL - \cK_m) (\td \UU_{1,\D},\td \S_{1,\D} ), (\td \UU_{1,\D},\td \S_{1,\D}) \ra_{\cX^{m}} 
 + \la \cL_A (\vA_{\D}), \vA_{\D} \ra_{\cX_A^{m}} + I_{\cN,\D} 
 \notag\\
 &\leq - \lam E_{m,\D}^2  + I_{\cN, \D},
 \label{eq:non_Y2:a}
\end{align}
where we have denoted
\begin{equation}
I_{\cN,\D} \teq \la (\cN_{U, \D}, \cN_{\S,\D}), (\td \UU_{1,\D}, \td \S_{1,\D}) \ra_{\cX^{m}}
 + \la \cN_{A, \D}, \vA_{\D} \ra_{\cX^{m}_A} .  
  \label{eq:non_Y2:b}
\end{equation}
\end{subequations}
For the nonlinear term defined in~\eqref{eq:non_Y2:b}, in analogy to \eqref{eq:non_est}, our goal is to establish the bound
\beq
\label{eq:non_est_dif}
|I_{\cN, \D}| \les_m  (\cE_{a} + \cE_b) \cE_{\D} E_{m,\D},
\eeq
where $\{ \cE_{\ell} \}_{\ell \in \{a,b\}}$ is defined according to the last line in~\eqref{eq:non_E}, and $\cE_{\D}$, $E_{m,\D}$ are defined in~\eqref{eq:nota_dif}.

\paragraph{\bf{Estimate of nonlinear terms.}}
From~\eqref{eq:non} we have that $\cN_{V, \D} = \cN_{V}(\td \VV_{a} ) - \cN_{V}(\td \VV_{b})$ is given by
\begin{subequations}
\label{eq:non:diff}
\begin{align}
{\cN}_{U,\D} 
&=  {\cN}_{U,\D}(\td \VV_{a}) -  {\cN}_{U,\D}(\td \VV_{b}) \notag\\
&= - \td \UU_\D \cdot \na \td \UU_a - \alpha \ts_\D \na \ts_a   - \vA_\D \cdot \na \vA_a 
- \td \UU_b \cdot \na \td \UU_\D - \alpha \ts_b \na \td \S_\D   - \vA_b \cdot \na \vA_\D, 
\label{eq:non:diff:a}\\
{\cN}_{\S,\D}
&=  {\cN}_{\S,\D}(\td \VV_{a}) -  {\cN}_{\S,\D}(\td \VV_{b}) \notag\\
&=  - \td \UU_\D \cdot \na \td \S_a -  \alpha \ts_\D   \div(\td \UU_a)
 - \td \UU_b \cdot \na \td \S_\D -  \alpha \ts_b  \div(\td \UU_\D), 
 \label{eq:non:diff:b}\\
{\cN}_{A,\D} 
&=  {\cN}_{A,\D}(\td \VV_{a}) -  {\cN}_{A,\D}(\td \VV_{b}) \notag\\
&=   - \td \UU_\D \cdot  \na \vA_a  - \vA_\D \cdot \na \td \UU_a
 - \td \UU_b \cdot  \na \vA_\D  - \vA_b \cdot \na \td \UU_\D .
 \label{eq:non:diff:c}
\end{align}
\end{subequations}
In order to simplify our estimates for the terms in~\eqref{eq:non:diff}, we note that by~\eqref{eq:non_err} (a bound which is applicable since $\{(  \wh \UU_{1, \ell}, \wh \S_{1, \ell} )\}_{\ell \in \{a,b\}}$ satisfy the assumptions of Proposition~\ref{prop:onto}), for $\ell \in \{a,b\}$ we have
\begin{subequations}
\beq
\label{eq:prop:a}
\|  \td \VV_\ell \|_{\cX^{m+1}}   + \| \td \VV_{2,\ell}\|_{\cX^{m+2}} 
\les_m \| (\td \UU_{1,\ell},\td \S_{1,\ell}) \|_{\cX^{m+1}}  + \| \AA_\ell \|_{\cX_A^{m+1}} + \| (\td \UU_{2,\ell},\td \S_{2,\ell})  \|_{\cX^{m+2}}  
\les_m \cE_{\ell} .
\eeq
Moreover, by the definition of $E_{m,\D}$ and $\cE_\D$ in~\eqref{eq:nota_dif}, and the fact that $\cX_A^m \subset \cX^m$, $\cX^{m+2}\subset \cX^m$, we have 
\beq
\label{eq:prop:b}
\|  \td \VV_\D\|_{\cX^{m}}    
\leq \|\td \VV_{1, \D} \|_{\cX^m}
 + \|\td  \VV_{2, \D} \|_{\cX^{m+2}} \les_m \cE_{\D}.
\eeq
\end{subequations}

First, we focus on the contribution of ${\cN}_{U,\D}$ to~\eqref{eq:non_Y2:b}, and denote 
\[
 I_{U,\D} = \la \cN_{U, \D}, \td \UU_{1,\D} \ra_{\cX^{m}}.
\]
From~\eqref{eq:non:diff:a}, \eqref{eq:prop:a}, \eqref{eq:prop:b},  using the Leibniz rule, Lemma~\ref{lem:non_main}, and Lemma~\ref{lem:prod}, similarly to the proof of~\eqref{eq:non_Nu}, we obtain  
\begin{subequations}
\label{eq:I:U:S:A:D}
\begin{align}
 I_{U,\D}&=
 - \e_m \int \vp_{2m}^2 \vp_g \bigl( \td \UU_b \cdot \na \Delta^m \td \UU_{1,\D} + \alpha \ts_b \na \Delta^m  \td \S_{1,\D} + \vA_b \cdot \na \Delta^m  \vA_\D\bigr) \cdot \Delta^m \td \UU_{1,\D}
 \notag\\
 &\quad + \mathcal{O}_m\bigl((\cE_a+\cE_b) \cE_\Delta E_{m,\D}\bigr).
\end{align}
For the contributions of ${\cN}_{\S,\D}$ and ${\cN}_{A,\D}$  to~\eqref{eq:non_Y2:b},
by using~\eqref{eq:non:diff:b}--\eqref{eq:non:diff:c}, \eqref{eq:prop:a}, \eqref{eq:prop:b}, the Leibniz rule, Lemma~\ref{lem:non_main}, and Lemma~\ref{lem:prod}, similarly to the proofs of~\eqref{eq:non_Ns}--\eqref{eq:non_Na}, we obtain  
\begin{align}
I_{\S,\D}  = \la \cN_{S, \D}, \td \S_{1,\D} \ra_{\cX^{m}}  
&=
 - \e_m \int \vp_{2m}^2 \vp_g \bigl( \td \UU_b \cdot \na \Delta^m \td \S_{1,\D} + \alpha \ts_b \div (\Delta^m  \td \UU_{1,\D}) \bigr) \cdot \Delta^m \td \S_{1,\D}
 \notag\\
 &\quad + \mathcal{O}_m\bigl((\cE_a+\cE_b) \cE_\Delta E_{m,\D}\bigr),
 \\
I_{A,\D}  = \la \cN_{A, \D},   \AA_{\D} \ra_{\cX_A^{m}}  
&=
 - \e_m \int \vp_{2m}^2 \vp_g \bigl( \td \UU_b \cdot \na \Delta^m   \AA_{\D}  + \vA_b \cdot \na \Delta^m  \td \UU_{1,\D}\bigr) \cdot \Delta^m   \AA_{\D}
 \notag\\
 &\quad + \mathcal{O}_m\bigl((\cE_a+\cE_b) \cE_\Delta E_{m,\D}\bigr).
 \end{align}
\end{subequations}
Finally, we sum the three identities in~\eqref{eq:I:U:S:A:D}, appeal to the vector calculus identities in~\eqref{eq:good:vector:identities}, integrate by parts, use~\eqref{eq:non_IBP} with $2m+2$ replaced by $2m$ and with $G \in \{(\td \UU_\ell,\td\S_{\ell},\AA_{\ell})\colon \ell \in \{a,b\}\}$, and appeal to the bound~\eqref{eq:prop:a}, to conclude
\begin{align*}
I_{\cN,\D} &= I_{U,\D} + I_{\S,\D} + I_{A,\D}
\notag\\
&= 
- \e_m \int \vp_{2m}^2 \vp_g \Bigl( \tfrac 12  (\td \UU_b \cdot \na) \bigl| \Delta^m \td \VV_{1,\D}\bigr|^2 \notag\\
&\qquad \qquad \qquad  \qquad
+
\alpha \ts_b \div \bigl( \Delta^m  \td \S_{1,\D} \Delta^m  \td \UU_{1,\D} \bigr)
+ 
\vA_b \cdot \na \bigl( \Delta^m  \vA_\D \cdot \Delta^m \td \UU_{1,\D} \bigr)
\Bigr)
 \notag\\
&\quad  + \mathcal{O}_m\bigl((\cE_a+\cE_b) \cE_\Delta E_{m,\D}\bigr) \notag\\
&= \mathcal{O}_m\bigl((\cE_a+\cE_b)  E_{m,\D}^2\bigr)
 + \mathcal{O}_m\bigl((\cE_a+\cE_b) \cE_\Delta E_{m,\D}\bigr).
\end{align*}
This concludes the proof of~\eqref{eq:non_est_dif}.

\paragraph{\bf{Conclusion of the proof.}}
Combining the estimates \eqref{eq:non_Y2}--\eqref{eq:non_est_dif}, we deduce
\beq\label{eq:non_EE}
\tfrac{d}{ds} E_{m, \D} + \lam  E_{m, \D} \les_m (\cE_a + \cE_b) \cE_{\D},
\eeq
where we recall that $\cE_a, \cE_b, \cE_{\D}$ are defined in~\eqref{eq:nota_dif}. 
In order to bound the right side of \eqref{eq:non_EE}, we note that since $\cT_2$ is a linear map, $\td \VV_{2,\D} = \cT_2( \wh \UU_{1, \D}, \wh \S_{1,\D})$, and so using Lemma~\ref{lem:T2}, we obtain 
\[
\cE_{\D}(s) - E_{m,\D}(s)
= \|\td \VV_{2, \D}(s)\|_{\cX^{m+2}}  
 \les_m   e^{- \eta_s (s-\sin)} \| (\wh \UU_{1, \D}, \wh \S_{1, \D}) \|_{ Y_2}.
\]
On the other hand, for the term $\cE_a + \cE_b$, using \eqref{eq:junk:label:101} and~\eqref{eq:non_boot} (applicable since $\{(  \wh \UU_{1, \ell}, \wh \S_{1, \ell} )\}_{\ell \in \{a,b\}}$ satisfy the assumptions of Proposition~\ref{prop:onto}), we obtain 
\[
\cE_{\ell}(s) = E_{m+1, \ell} + \|\cT_2(  \wh \UU_{1, \ell}, \wh \S_{1, \ell} ) \|_{\cX^{m+2}}
\leq 2 \d e^{- \lam_1 (s-\sin)} + C_m e^{-\eta_s (s-\sin)}  \d_Y,
\]
for $\ell \in \{a,b\}$.
From the two estimates above and~\eqref{eq:non_EE}, we obtain
\[
\tfrac{d}{ds} E_{m, \D} + \lam  E_{m, \D} 
\leq \td C_m ( \d e^{- \lam_1 (s-\sin)} +  e^{-\eta_s (s-\sin)}  \d^{2/3}) \bigl(E_{m,\D} +  e^{- \eta_s (s-\sin)} \| (\wh \UU_{1, \D}, \wh \S_{1, \D}) \|_{ Y_2}\bigr),
\]
for some constant $\td C_m\geq 1$. Since $\eta_s < \lam_1 < \lam$ (see~\eqref{eq:decay_para}), we may choose 
$\d_0 \in (0,1]$ small enough to ensure $\td C_m ( \d_0 +   \d_0^{2/3}) \leq \lam - \lam_1$, and thus, for all $s\geq \sin$ and any $\d \leq \d_0$ we have 
\[
\tfrac{d}{ds} E_{m, \D} + \lam_1  E_{m, \D} 
\leq2 \td C_m  \d_0^{2/3}  e^{- 2 \eta_s (s-\sin)} \| (\wh \UU_{1, \D}, \wh \S_{1, \D}) \|_{ Y_2}.
\]
To conclude, recall that $E_{m,\D}(\sin) = 0$, and $2\eta_s > \lam_1$ (see~\eqref{eq:decay_para}), so that by integrating the above estimate we arrive at 
\[
e^{\lam_1(s-\sin)} E_{m,\D}(s) \leq 2 \td C_m  \d_0^{2/3} \| (\wh \UU_{1, \D}, \wh \S_{1, \D}) \|_{ Y_2}  \int_{\sin}^\infty  e^{- (2 \eta_s-\lam_1) (s-\sin)} ds 
= \tfrac{20}{3} \lam^{-1} \td C_m  \d_0^{2/3} \| (\wh \UU_{1, \D}, \wh \S_{1, \D}) \|_{ Y_2} .
\]
Upon further decreasing $\d_0$, to also ensure that $\frac{20}{3} \lam^{-1} \td C_m  \d_0^{2/3} < \frac 12$, the above estimate proves~\eqref{eq:need:new:label:1}, thus concluding the proof of Proposition~\ref{prop:contra}.

\subsection{Additional estimates of the solution from Theorem~\ref{thm:blowup}}
\label{sec:additional:estimates}
In this section we establish sharp spatial decay estimates for the self-similar solution $(\UU,\S,\AA) = (\td \UU_1 + \td \UU_2 + \bar \UU, \td \S_1 + \td \S_2 + \bar \S, \AA)$ constructed in Theorem~\ref{thm:non}, with {\em initial data as specified in}~\eqref{eq:the:IC}; that is, for the self-similar solution described by Theorem~\ref{thm:blowup}. 
We claim that:
\begin{subequations}
\label{eq:decay_est}
\begin{align}
 |\UU(y, s)| + | \AA(y, s)| &\les  \la y \ra^{1 - r } 
 \label{eq:decay_est:a}\\
 |\na \UU(y, s)| + |\na \S(y, s) |  + |\na \AA(y, s)| & \les \la y \ra^{-r},
 \label{eq:decay_est:b}
\end{align}
\end{subequations}
pointwise for $y\in {\mathbb R}^2$, for all $s\geq \sin$, where we recall that $r>1$ (see~\eqref{r-range}). Note that \eqref{eq:decay_est:a} is missing a decay estimate for $\S$; this is due to the fact that we have chosen the rescaled sound speed $\S$ to equal a constant in the far-field at the initial time (see~\eqref{eq:the:IC:b}). 

The purpose of the decay estimates~\eqref{eq:decay_est} is to prove that there are no singularities in the Euler solution $(\uu,\sigma)$ (in original $(x,t)$ variables) away from $(x,t) = (0,T)$, as claimed in Remark~\ref{rem:blowup:at:0}. Indeed, using the decay estimates in~\eqref{eq:decay_est} and the self-similar relation~\eqref{ss-var}, which gives 
 $x = y (T-t)^{1/r}$ and $T-t = e^{-rs}$, for any $t\in [0,T)$ and $x \neq 0$ (so that $|y| = |x| (T-t)^{-1/r} \to \infty$ as $t \converges T$) we obtain 
\begin{subequations}
\label{eq:bounds:for:original:variables}
\begin{align}
| f(x,t) | & = \tfrac{1}{r} (T-t)^{\frac{1}{r}-1} | F( y,s) | 
\notag\\
&\les (T-t)^{\frac{1}{r}-1} \la y \ra^{1-r}
\leq |x|^{1-r}, 
\qquad \forall (f, F) \in \bigl\{(\uu^R , \UU), (\uu^{\th}, \AA) \bigr\},
\label{eq:bounds:for:original:variables:a}
\\
|(\na f)(x,t) | & = \tfrac{1}{r} (T-t)^{-1} | \na F( s, y) | 
\notag\\
&\les (T-t)^{-1} \la y\ra^{-r}
\leq |x|^{-r}, 
\qquad \forall (f, F) \in \bigl\{ (\uu^R , \UU), ( \s, \S), (\uu^{\th}, \AA) \bigr\}. 
\label{eq:bounds:for:original:variables:b}
\end{align}
\end{subequations}
As a consequence of~\eqref{eq:bounds:for:original:variables}, we have that  $ \uu, \na \uu, \na \s$ may only blow up at the origin $x=0$, at time $t=T$. 

In order to control the undifferentiated rescaled sound speed (missing from~\eqref{eq:bounds:for:original:variables}), we estimate $\s$ via the Lagrangian flow associated to $\uu$. By  choosing $\d$ to be sufficiently small in Theorem~\ref{thm:non}, the estimates established in Theorem~\ref{thm:non} imply that the total self-similar radial velocity  
$ \UU = \td \UU_1 + \td \UU_2 + \bar \UU $ 
satisfies the assumption of Lemma~\ref{lem:traj}. Let $a\mapsto X(a,t)$ be the flow map (recall, Footnote~\ref{footnote:flow:map}) associated with $\uu = \uu^R + \uu^\th$, where $\uu^R$ and $\uu^\th$ are given in terms of $\UU$ and $\AA$ via \eqref{ss-var:b}--\eqref{ss-var:c}. Fix $x \neq 0$. For any $0\leq t < T$, denote $x_0 = X^{-1}(x, t) $, i.e.~$X(x_0, t) = x$. 
From Lemma \ref{lem:traj}, for $0\leq t^\prime <  t < T $, we have
\begin{equation}
 |x| = |X(x_0, t)|  \les |x_0|, 
 \quad
 R_l(x) \teq |x| \min(1, |x|^{c_1}) \les |x_0| \min(1, |x_0|^{c_1}) \les |X(x_0, t^\prime)| \les |x_0|,
\label{eq:bounds:for:original:variables:c}
\end{equation}
for some positive constant $c_1$. In particular, the above bounds imply that $x\neq 0 \Leftrightarrow x_0 \neq 0$. 
Next, we use~\eqref{eq:euler:b} and the definition $x=X (x_0,t)$  to write
\[
\sigma(x,t) 
= \sigma_0(x_0) \exp\Bigl(-\alpha \int_0^t (\div \uu)(X(x_0,t^\prime),t^\prime) d t^\prime \Bigr).
\]
Since for axisymmetric flows we have $(\div \uu) (x^\prime,t)= (\tfrac{1}{R}+\partial_R)u^R (|x^\prime|,t)$, by~\eqref{eq:bounds:for:original:variables} and~\eqref{eq:bounds:for:original:variables:c} it follows that 
$|(\div \uu)(X(x_0,t^\prime),t^\prime)| \les |X(x_0,t^\prime)|^{-r} \leq R_l(x)^{-r}$, for all $t^\prime \in [0,t]$. Therefore, we deduce
\begin{equation}
\| \sigma_0^{-1}\|_{L^\infty}^{-1} \exp\bigl(- C R_l(x)^{-r} \bigr)
\leq 
\sigma(x,t)
\leq
\| \sigma_0\|_{L^\infty} \exp\bigl(C R_l(x)^{-r} \bigr)
\label{eq:decay_est:c}
\end{equation}
for all $x\neq 0$ and $t\in [0,T)$, for some constant $C>0$ which is allowed to depend on $T$ but not on $x$. 
In particular, we obtain that $\s(x, t)$ does not become vacuous and does not blow up away from $(x,t) = (0,T)$.

The goal of the remaining part of this section is to prove~\eqref{eq:decay_est}.

\subsubsection{Sharp decay estimates of the solution}

Recall cf.~\eqref{eq:dec_U} that the steady state $(\bar \UU,\bar \S,0)$ satisfies the bounds~\eqref{eq:decay_est}. We thus only need to perform weighted $W^{i,\infty}$ estimates, for $i\in \{0,1\}$, on the perturbation $\td \VV = (\UU - \bar \UU, \S - \bar \S, \AA)= (\td \UU, \td \S, \AA) = \td \VV_1 + \td \VV_2$ (see~\eqref{eq:init}). Adding the equations for $\td \VV_1$ and $\td \VV_2$ in \eqref{eq:non_V}, we obtain 
\[
 \pa_s \td \VV  = \cL_V( \td \VV) + \cN_V(\td \VV),
\]
where $\cL_V = (\cL_U, \cL_{\S}, \cL_A)$ and $\cN_V = (\cN_U, \cN_A, \cN_{\S})$ are defined in \eqref{eq:lin}. In order to establish~\eqref{eq:decay_est}, we consider the  weighted quantity 
\begin{subequations}
\label{norm:linf}
\begin{equation}
F \teq (\phi_0^2 ( |\td \UU|^2 + |\AA|^2)+ \phi_1^2 |\na \td \VV|^2 )^{1/2}, 
\qquad \mbox{where} \qquad 
\phi_0 = \la y \ra^{r-1},
\quad \phi_1 =  \la  y \ra^{r} ,
\label{norm:linf:a}
\end{equation}
and note that~\eqref{eq:decay_est} is a consequence of the uniform boundedness of the associated ``energy''
\beq
\label{norm:linf:b}
E_{\infty}(s) \teq \| F(\cdot,s) \|_{L^{\infty}},
\eeq
\end{subequations}
for $s\geq \sin$.
Note that since $\td \S$ does not decay, we do not include the term $\phi_0 |\td \S|$ in the definition of $F$.

In order to estimate $E_\infty(s)$, for $f \in \{\td \UU, \td \S, \AA\}$ we first focus on the linear terms $(r-1) f$ and $y \cdot \na  f$ present in~\eqref{eq:lin:LU}--\eqref{eq:lin:LS}; accordingly, we re-write \eqref{eq:lin:a}--\eqref{eq:lin:c} as
\begin{subequations}
\label{eq:linf_main}
\begin{align}
 \pa_s f + y \cdot \na f   &=  (1-r) f + \cL_{f, R}  + \cN_f , \\
 \cL_{U, R} & =   -  \bar{ \UU} \cdot \na \td \UU  -   \td \UU \cdot \na \bar{\UU} 
 - \alpha \td \S  \na \bar{\S}  - \alpha \bar{\S} \na \td \S, 
 \label{eq:linf_main:b}\\ 
 \cL_{A, R}  &= -  \bar{ \UU} \cdot \na \vA  - \AA \cdot \na \bar \UU,  
\label{eq:linf_main:c} \\
 \cL_{\S, R} &=  - \bar{ \UU}  \cdot \na \td \S
 -  \td \UU   \cdot \na \bar{\S}  -  \alpha \td \S  \div(\bar \UU)  - \alpha \bar{\S}  \div(\td \UU)  
 \label{eq:linf_main:d}
\end{align}
\end{subequations}
for terms $ \cL_{f, R} , \cN_f$, which are quadratic in either the perturbation or the profile, and which decay faster than the linear terms; we will treat these perturbatively and think of them as lower order. 

For $(i,f) \in \{ (0,\td \UU), (0,\AA), (1, \td \UU), (1,\td \S), (1,\AA) \}$, a direct computation yields 
\beq\label{eq:linf}
\pa_s (\phi_i \na^i f ) + y \cdot \na (\phi_i\na^i f ) = \Bigl( (1-r) - i + \f{y \cdot \na \phi_i}{\phi_i} \Bigr) ( \phi_i \na^i  f)  + \phi_i \na^i (  \cL_{f, R} f + \cN_f ).
\eeq
Since $r>1$ we may directly verify that 
\beq\label{eq:linf_1}
(1 - r ) - i + \f{y \cdot \na \phi_i}{\phi_i} 
= (1 - r) - i + (r-1 + i) \f{|y|^2}{\la y \ra^2}
= - (r-1 + i) \f{1}{\la y \ra^2} \leq 0
\eeq
for all $y\in{\mathbb R}^2$. 
For the contribution of the lower order terms present on the right side of~\eqref{eq:linf}, in Sections~\ref{sec:linf_lin} and~\ref{sec:linf_non} below, we will establish the pointwise estimates
\begin{equation}
\label{eq:linf_2}
 |  \phi_i \na^i \cL_{f, R}  | \les \| \td \VV \|_{\cX^m},
 \quad | \phi_i \na^i \cN_f | \les   \| \td \VV \|_{\cX^m}^2,  
\end{equation}
for all $(i,f) \in \{ (0,\td \UU), (0,\AA), (1, \td \UU), (1,\td \S), (1,\AA) \}$.

Multiplying equation~\eqref{eq:linf} with $\phi_i \na^i f $,  summing over $(i,f) \in \{ (0,\td \UU), (0,\AA), (1, \td \UU), (1,\td \S), (1,\AA) \}$, recalling the definition of $F$ in~\eqref{norm:linf}, and combining the resulting identity with the bounds \eqref{eq:linf_1} and \eqref{eq:linf_2}, we deduce
\[
 \pa_s F^2 + y \cdot \na F^2 \leq C \| \td \VV \|_{\cX^m} ( 1 + \| \td \VV \|_{\cX^m}  ) F.
\]
Upon dividing both sides by $F$, composing with the Lagrangian flow of the operator $\pa_s + y \cdot \na$, which is given explicitly as $Y(y_0,s) = y_0 e^{s-\sin}$, and using the nonlinear estimate $\|\td \VV(s)\|_{\cX^m} \les \delta^{2/3} e^{-\eta_s (s-\sin)}$ established\footnote{The norms $\cW^{m+1}, \cW_A^{m+1}, \cX^{m+2}$ are stronger than the  $\cX^m$ norm; see~\eqref{norm:Xk}, the bound $\vp_g \les \vp_A$, Lemma~\ref{lem:Xm_chain}, and~\eqref{norm:Wk}.} earlier in Theorem~\ref{thm:non},  we obtain 
\[
  \tfrac{d}{d s} F(Y(y_0, s), s) \leq  C \delta^{2/3} e^{-\eta_s (s-\sin)} ( 1 + C \delta^{2/3} e^{-\eta_s (s-\sin)} )
\]
for some $C>0$ independent of $\delta$ and of $(y_0,s)$. Using that by definition $F(y_0, \sin) \leq E_{\infty}(\sin)$ for all $y_0 \in \mathbb{R}^2$, integrating the above ODE with respect to $s\geq \sin$, then taking the supremum over $y_0$ (here we use that the map $y_0 \mapsto Y(y_0,s)$ is a bijection of $\mathbb{R}^2$ for any $s$) we arrive at 
\begin{align*}
  E_{\infty}(s)  
 &\leq 
  E_{\infty}(\sin) + C \delta^{2/3} \int_{\sin}^s  e^{- \eta_s (s^\prime -\sin)} d s^\prime + C^2 \delta^{4/3} \int_{\sin}^s e^{- 2 \eta_s (s^\prime -\sin)} d s^\prime
  \notag\\
 &\leq 
  E_{\infty}(\sin) + C \delta^{2/3} \eta_s^{-1}   + C \delta^{4/3} (2 \eta_s)^{-1},
\end{align*}
for all $s\geq \sin$. 
Combining the above estimate with the boundedness of  $E_{\infty}(\sin)$ (which follows from the definition of the initial data in~\eqref{eq:the:IC}, from the decay of the self-similar profiles $(\bar U,\bar \S)$ in~\eqref{eq:dec_U}, and from the fact that $\td \VV_2(\sin)$ has compact support~\eqref{eq:V2_form2:d}), we deduce that the function $F$ defined  in~\eqref{norm:linf:a} is uniformly bounded in space and time, thereby proving~\eqref{eq:decay_est}. 

It only remains to establish~\eqref{eq:linf_2}, which we achieve in the next two subsections.

\subsubsection{Estimates of the linear terms}\label{sec:linf_lin}

In this section, we establish the first estimate in \eqref{eq:linf_2}, which concerns the linear terms defined in~\eqref{eq:linf_main:b}--\eqref{eq:linf_main:d}. From Lemma \ref{lem:prod}, we have 
\beq\label{eq:decay_est2}
 | \na^i \td \VV | \les  \|\td \VV \|_{\cX^m} \la y \ra^{-i + \kp_1/2} ,  \qquad \kp_1 = \tfrac 14,  \quad 0\leq i \leq 2m -2.
\eeq

For the sake of brevity, we estimate only two ``typical'' terms  appearing in the definition of  $ \phi_i \na^i \cL_{f, R}$ for $(i,f) \in \{ (0,\td \UU), (0,\AA), (1, \td \UU), (1,\td \S), (1,\AA) \}$; see~\eqref{eq:linf_main:b}--\eqref{eq:linf_main:d}. More precisely, we consider $\td \S  \na \bar \S$ and $\bar \S \na \td \S$, which arise when $f=\td \UU$. Recall the decay estimates of the profile in~\eqref{eq:dec_U} and the fact that $\kp_1 \in (0, 1)$. Applying \eqref{eq:decay_est2}, for $i\in\{0,1\}$, we obtain 
\[
\bal
   \la y \ra^{i + r-1} |\na^i(\td  \S  \na \bar \S )  |
  & \les 
\la y \ra^{i + r - 1} ( | \na^i \td \S  | \; |\na \bar \S| 
+ |\td \S| \; |\na^{1 + i} \bar \S| ) \\
& \les \la y \ra^{i + r - 1} \| \td \VV \|_{\cX^m} 
   \la y \ra^{ - i+ \kp_1/ 2 -r} 
\les \| \td \VV \|_{\cX^m} , \\
\la y \ra^{i + r-1}   | \na^i( \bar  \S  \na \td  \S)  |
 & \les
 \la y  \ra^{i + r - 1} ( | \na^i \bar \S | \; |\na \td \S | 
+ |\bar \S| \; |\na^{1 + i} \td \S | ) \\
& \les 
\la y \ra^{i + r - 1} \| \td \VV \|_{\cX^m} 
 \la y \ra^{ - i+ \kp_1/ 2 -r} 
\les \| \td \VV\|_{\cX^m} .
\eal
\]
All the other terms present in~\eqref{eq:linf_main:b}--\eqref{eq:linf_main:d} are estimated similarly; we omit these details and conclude the proof of the first estimate in~\eqref{eq:linf_2}.

\subsubsection{Estimates of the nonlinear terms}\label{sec:linf_non}
In this section, we establish the second estimate in \eqref{eq:linf_2}, which concerns the nonlinear terms defined in~\eqref{eq:non}.

We note that each nonlinear terms $\cN_f$ in \eqref{eq:non} can be written as a sum of terms of the kind $F \na G$ for $F, G \in \{\td \UU, \td \S, \AA \}$. By
using estimate~\eqref{eq:decay_est2}, we obtain
\begin{equation*}
\phi_0 |F \nabla G| 
\les  \phi_0 \; \|\td \VV \|_{\cX^m} \la y \ra^{ \kp_1/2} \; \|\td \VV \|_{\cX^m} \la y \ra^{-1+ \kp_1/2}
= \la y \ra^{r -2 + \kp_1} \|\td \VV \|_{\cX^m}^2.
\end{equation*}
At this stage, we recall from~\eqref{eq:vp:g:def} that $\kp_1 = 1/4$, and from \eqref{r-range} that $r < r_{\mathsf{eye}}(\alpha)$. Moreover, since \eqref{r-range} gives an explicit expression for $r_{\mathsf{eye}}(\alpha)$, we may directly verify that $r_{\mathsf{eye}}(\alpha) < \sqrt{2}$ for all $\alpha>0$. Therefore,
\begin{equation}
 r -2 + \kp_1 < \sqrt{2} - 2 + \tfrac 14 < 0.
\label{eq:r:less:r:eye}
\end{equation}
The above two bounds proves the second estimate in~\eqref{eq:linf_2} for $i=0$.

It thus remains to consider the second bound in~\eqref{eq:linf_2} for $i=1$ and $f \in \{ \td \UU, \td \S, \AA\}$. We proceed as above, using \eqref{norm:linf} and \eqref{eq:decay_est2} we deduce 
\begin{align*}
\phi_1 |\nabla (F \nabla G)| 
&\les 
\phi_1 |\nabla F| \, |\nabla G|
+
\phi_1 |F| \, |\nabla^2 G|
\notag\\
&\les 
\phi_1 \, \|\td \VV \|_{\cX^m} \la y \ra^{-1 + \kp_1/2} \, \|\td \VV \|_{\cX^m} \la y \ra^{-1 + \kp_1/2}
+
\phi_1 \, \|\td \VV \|_{\cX^m} \la y \ra^{\kp_1/2} \, \|\td \VV \|_{\cX^m} \la y \ra^{-2 + \kp_1/2}
\notag \\
&= \|\td \VV \|_{\cX^m}^2 \la y \ra^{r -2 + \kp_1}.
\end{align*}
Appealing to~\eqref{eq:r:less:r:eye}, the  above estimate yields $\phi_1 |\nabla (F \nabla G)| \les \|\td \VV \|_{\cX^m}^2$ for $F, G \in \{\td \UU, \td \S, \AA \}$. This concludes the proof of~\eqref{eq:linf_2}.

\appendix

\section{Trajectory estimates and the proof of Lemma \ref{lem:non_blowup}}\label{sec:non_blowup}

In this section, we first establish a trajectory estimate and then prove Lemma~\ref{lem:non_blowup} using the properties of the profile in Lemma~\ref{lem:profile} .

\begin{lemma}\label{lem:traj}
Assume that the physical space velocity $\uu$ is self-similar according to the transformation~\eqref{ss-var}. In particular let $\UU= \UU(y,s) = U(|y|,s) \ee_R$ denote the self-similar radial velocity, and recall that $x = y e^{-s}$ and $T-t = e^{- r s}$. 
Suppose that $\UU \in C^1_{y,s}$ and that $\UU$ satisfies the quantitative bound
\begin{equation}
\label{eq:lem:traj:ass}
| \UU(s, y) - \bar \UU(y)| \leq \tfrac{1}{2} \min(\kp, \tfrac{r-1}{2\al} ) \min( |y| , |y|^{1-\e}  ),
\end{equation}
for all $y \in \mathbb{R}^2$, $s\geq \sin$, and some $\e > 0$, where $\bar \UU = \bar U \ee_R$ and the parameters $\kp, r, \al> 0$ satisfy \eqref{eq:profile:properties}. Then, there exists a sufficiently small $c_1>0$, with $c_1 = c_1(\e,\kp,r,\alpha,\bar U,T)$, such that the Lagrangian flow map  $a\mapsto X(a,t)$ associated with $\uu(x,t)$ in physical space variables, satisfies the uniform estimates 
\begin{equation}
C^{-1}  |a| \min( 1, |a|^{c_1} )
 \leq 
   |X( a,  t)|   \leq  C |a|,
\label{eq:lem:traj:concl}
\end{equation}
for all $a \in \mathbb{R}^2$ and all $t\in [0,T)$, for some constant $C = C (\e,\kp,r,\alpha,\bar U,T) \geq 1$.
\end{lemma}

\begin{proof}
Fix $a \neq 0$. According to~\eqref{ss-var:a}, define the radial component of the corresponding self-similar trajectory to be
\[
R(s) \teq \frac{|X(a, t)|}{(T-t)^{1/r}} = |X(a, t)| e^s.
\]
Since $\UU(y,s) = U(|y|,s) \ee_R$, it follows from $\frac{d}{dt} X(a,t) = \uu(X(a,t),t)$ and the definition~\eqref{ss-var} of the self-similar transformation that $R(s)$ solves the ODE
\[
\partial_s R(s) = R(s) + U(R(s),s)  ,
\quad  R(\sin) = |a| T^{-\frac 1r}= |a| e^{\sin}.
\]
By~\eqref{eq:rep22}, the assumption~\eqref{eq:lem:traj:ass} on $\UU$ we get
\beq   \label{est:1}
\partial_s R(s) \ge   R(s) + \bar U(R(s)) - \tfrac{1}{2} \kp R(s) \geq \tfrac{1}{2} \kp R(s)
\quad 
\Rightarrow
\quad 
R(s) \ge e^{\frac 12 \kp s} R(\sin) = e^{\f{1}{2} \kp (s-\sin)} |a| e^{\sin}.
\eeq
The above estimates show that $R(s)$ is increasing and diverges exponentially as $s\to \infty$. 

Since by the definition of $R(s)$ we have $|X(a,t)| = R(s) e^{-s}$, we compute
\begin{equation*}
\tfrac{d}{ d s} \bigl( R(s) e^{-s} \bigr)
= \bigl(R(s) e^{-s}\bigr)  \tfrac{ U(R(s),s)}{R(s)}
,
\end{equation*}
and therefore 
\begin{equation}
\label{est:1:a} 
R(s) e^{-s} = |a| e^{I(s)},
\quad 
\mbox{where}
\quad 
I(s) \teq \int_{\sin}^s \tfrac{ U(R(s^\prime), s^\prime)}{R(s^\prime)} d s^\prime,
\end{equation}
resulting in an exact implicit formula for $e^{-s} R(s)$.
The usefulness of~\eqref{est:1:a} is seen once we recall the decay estimate \eqref{eq:dec_U} and the assumption \eqref{eq:lem:traj:ass}; together with~\eqref{est:1}, these facts imply that the exponent appearing and \eqref{est:1:a} satisfies
\[
| I(s)|
 \les \int_{\sin}^{s} \min(1, R(s^\prime)^{-c})  d s^\prime
\les \int_{\sin}^{\infty} \min\Bigl(1,|a|^{-c}e^{-c \sin}  e^{- \f{1}{2}c \kp (s-\sin)}  \Bigr) d s
\]
where $c = \min(r, \e) > 0$. 
Denoting $s_+ = \max(\sin, (1- \frac{2}{\kp}) \sin - \frac{2}{\kp}\log|a| )$, we obtain from the above estimate that  
\begin{align}
|I(s)| 
&\les \int_{\sin}^{s_+} 1  d s + |a|^{-c}e^{-c \sin} \int_{s_+}^{\infty} e^{- \f{1}{2}c \kp (s-\sin)}  d s \notag\\
&\les  s_+ 
+ |a|^{-c}e^{-c \sin}  e^{- \f{1}{2}c \kp (s_+-\sin)}
\les s_+ + e^{-\frac 12 c \kappa \sin}
\leq c_1 \max\bigl(1, - \log |a|\bigr),
\label{est:1:b}
\end{align}
where $c_1 = c_1 (\e,\kappa,r,\alpha,\bar U,T) >0$ is a constant. From~\eqref{est:1:a} and~\eqref{est:1:b} we deduce
\begin{equation}
|X(a,t)| = R(s) e^{-s} \geq |a| e^{- |I(s)|} 
\geq |a| e^{- c_1 \max(1, - \log |a|)}
= |a| \min(|a|, \tfrac{1}{e})^{c_1},
\label{eq:junk:junk}
\end{equation}
which concludes the proof of the lower bound in~\eqref{eq:lem:traj:concl}.

In order to prove of the uper bound in~\eqref{eq:lem:traj:concl}, we seek an upper bound for $I(s)$, without absolute values. Since by~\eqref{eq:rep3} we have $\pa_{\xi} \bar U(0)
= -\f{r-1}{2\al} < 0$, by the continuity of $\bar U$ and assumption~\eqref{eq:lem:traj:ass} on $\UU$, there exists  a constant $\xi_b> 0$ such that for all $\xi \in [0,\xi_b]$ we have
\[
 \bar U(\xi) \leq - \tfrac{1}{2} \tfrac{r-1}{2\al}  \xi, 
\quad U(\xi) \leq \bar U(\xi) + |U(\xi) - \bar U(\xi)| \leq 0, 
\qquad \forall \xi \in [0, \xi_b] .
\]
Then, since $R(s)$ is increasing and diverges as $s\to \infty$ (see~\eqref{est:1}), there exists a time $s_b = s_b(|a|) \geq \sin$ such that $R(s_b) = \max(|a| e^{\sin},\xi_b)$. From~\eqref{est:1} we obtain $R(s) \geq R(s_b) e^{\frac 12 \kappa (s-s_b)} \geq \xi_b e^{\frac 12 \kappa (s-s_b)} $ for all $s \geq s_b$. On the other hand, by construction we have that for $s\in [\sin,s_b)$ we have either $s_b=\sin$ or $R(s) \leq \xi_b$ (and hence $U(R(s),s)\leq 0$). Putting this information together, we deduce
\[
I(s) = \int_{\sin}^{s_b} \tfrac{ U(R(s^\prime), s^\prime)}{R(s^\prime)} d s^\prime + \int_{s_b}^{s} \tfrac{ U(R(s^\prime), s^\prime)}{R(s^\prime)} d s^\prime
\leq \int_{s_b}^{s} \tfrac{|U(R(s^\prime), s^\prime)|}{R(s^\prime)} d s^\prime.
\]
Using the decay estimates $|U(\xi,s)| \les \xi^{1-c}$ for $c = \min(r,\e)$ (which follows from \eqref{eq:dec_U} and the assumption \eqref{eq:lem:traj:ass}), we get 
\begin{equation}
I(s) 
\les \int_{s_b}^{s} (R(s^\prime))^{-c} d s^\prime
\les \int_{s_b}^{\infty} \bigl( \xi_b e^{\frac 12 \kappa (s^\prime-s_b)})^{-c} ds^\prime 
\les 1
\label{est:2}
\end{equation}
where the implicit constant depends on $\e,\kappa,r,\alpha,\bar U,T$. We deduce
\begin{equation*}
|X(a,t)| = R(s) e^{-s} 
= |a| e^{I(s)}
\leq C |a|,
\end{equation*}
for some $C = C(\e,\kappa,r,\alpha,\bar U,T)\geq 1$, 
which concludes the proof of the upper bound in~\eqref{eq:lem:traj:concl}.

In order to conclude the proof, we note that if $a = 0$, since $\UU(0,s) =0$, we get $X(0, t) =0$ for all $t\in[0,T)$, and in this case the bound \eqref{eq:lem:traj:concl} holds trivially.
\end{proof}

\begin{remark}[\bf The three-dimensional case] \label{rem:2D:3D}
We note that while Lemma~\ref{lem:traj} was stated in two space dimensions, the same proof applies ``as is'' to the three-dimensional case; the only requirements are that the radial part of the  self-similar velocity $\UU$ can be written as $U(\xi,s) \ee_R$ with $U(0,s) =0$, $\UU$ is at least $C^1$ smooth in space and time, satisfies the bound \eqref{eq:lem:traj:ass}, and that the profile $\bar U$ satisfies~\eqref{eq:dec_U},~\eqref{eq:rep22}, and $\pa_\xi \bar U (0) < 0$ (we did not use the precise formula~\eqref{eq:rep3}, only negativity of the slope near the origin).
\end{remark}

With Lemma~\ref{lem:traj} and Remark~\ref{rem:2D:3D}, we turn to the proof of Lemma~\ref{lem:non_blowup}. 

\begin{proof}[Proof of Lemma~\ref{lem:non_blowup}]
We adopt the notation from the proof of Lemma \ref{lem:traj}. Applying Lemma~\ref{lem:traj}, respectively Remark~\ref{rem:2D:3D}, with $\UU = \bar \UU$ (so that $\UU$ is smooth, radially symmetric, and trivially satisfies~\eqref{eq:lem:traj:ass}),
we get from~\eqref{eq:junk:junk} that 
\begin{equation}
R(s) e^{-s} \geq R_l(a) , \qquad R_l(a) \teq |a| \min(\tfrac{1}{e}, |a| )^{c_1},
\label{est:3}
\end{equation}
for all $s\geq \sin$ and some constant $c_1>0$. From~\eqref{ss-var:a}, \eqref{ss-var:b}, \eqref{eq:dec_U}, and~\eqref{est:3}, we get the estimate 
\beq \label{est:ss:s}
\sigma(X(a, t), t) \les e^{s(r-1)} \bar \Sigma(R(s)) 
\les e^{s(r-1)}  R(s)^{-(r-1)} 
= (e^{-s} R(s))^{-(r-1)}
\les R_l(a)^{-(r-1)}  .
\eeq
Similarly, from~\eqref{ss-var:a}, \eqref{ss-var:b}, \eqref{eq:dec_U}, and  \eqref{est:3}, we have
\begin{align} \label{est:ss:gr}
 |(\na \uu)( X(a, t), t)|
  & \les (T-t)^{-1}| (\na \bar \UU) (R(s))|
  \les (T-t)^{-1}\la R(s)\ra^{-r} \notag \\
  & \les \min((T-t)^{-1}, ( e^{-s} R(s) )^{-r} ) 
  = \min((T-t)^{-1}, R_l(a)^{-r} ).
\end{align}

We now proceed to estimate $ \nabla_a X(a,t) $. We observe that $ \nabla_a X(a,t) $ solves the  ODE: $\partial_t \nabla_a X(a,t)  = \nabla \uu(X(t), t)  \nabla_a X(a,t)$, with initial data $\nabla_a X(a,0) = {\rm Id}$. 
Introducing $t_+ = t_+(a) \teq \max( T - R_l(a)^r, 0)$ we have $\min(T, R_l(a)^r ) \leq T-t_+ \leq R_l(a)^r$, and so by \eqref{est:ss:gr} we find  
\begin{align} \label{est:gr:fl}
|  \nabla_a X(a,t)| 
&\lesssim \exp\left( \int_0^t | \nabla \uu  (X(a, t^\prime), t^\prime ) | dt^\prime \right) \lesssim 
\exp\B( C \int_0^t \min( (T- t^\prime)^{-1}, R_l(a)^{-r} ) d t^\prime  \B) 
\notag \\
& \les \exp\B( C \int_0^{t_+} (T-t^\prime)^{-1} d t^\prime+ C \int_{t_+}^T R_l(a)^{-r} d t^\prime \B)
\notag\\
&= \exp(C \log T -C \log(T-t_+) + (T-t_+) R_l(a)^{-r})
\notag\\
&\les \min(T, R_l(a)^r )^{-C}
\les \max(1, R_l(a)^{-c_2} ) 
\end{align}
for some constant $c_2>0$. 

In the assumption of Lemma~\ref{lem:non_blowup}, we  choose $k_{\mathsf{van}} \geq (c_2 +  \f{r-1}{\al} ) (c_1 + 1) $, where $c_1$ is as in \eqref{est:3} and $c_2$ is as in \eqref{est:gr:fl}; with this choice, if $|f_0(a)|\les \min(1,|a|^{k_{\mathsf{van}}})$, since the initial data is bounded away from vacuum, by combining \eqref{est:ss:s} and \eqref{est:gr:fl} we obtain
\begin{align*}
|\rho( X(a,t), t) \nabla_a X(a,t) \tfrac{f_0}{\rho_0}(a,t) |
&\les \|\rho_0^{-1}\|_{L^\infty}
\max(1, R_l(a)^{-c_2- \f{(r-1)}{\al} } )
\min( 1, |a|^{k_{\mathsf{van}}})
\notag\\
&\les 
\max(1, |a|^{- k_{\mathsf{van}}}) \min(1, |a|^{k_{\mathsf{van}}}) \les 1,
\end{align*}
uniformly for all $0\leq t < T$ and all $a\neq 0$. If $a = 0$, then the assumption on $f_0$ implies $f_0(0) = 0$ and so $\rho( X(0,t), t) \nabla_a X(0,t)  f_0(0) = 0$ for $0\leq t < T$, which concludes the proof of Lemma~\ref{lem:non_blowup}. 
\end{proof}

% START %%%%%%%%%%%% APPENDIX B 
\section{Additional properties of the profile}\label{app:ODE}
The goal of this appendix is to clarify the proof of Lemma~\ref{lem:profile}, based on the arguments in~\cite{MeRaRoSz2022a}.  
For this purpose, we recall from~\cite{MeRaRoSz2022a} the properties of autonomous ODE for the profile $(\bar U, \bar \Sigma)$. Denote 
\beq\label{eq:ODE_nota}
W(x) =  - \xi^{-1} \bar U(\xi), \quad
S(x) =  \al \bar \S(\xi), 
\quad x(\xi) = \log \xi, \quad 
l = \tfrac{1}{\al} =\tfrac{2}{\g-1},
\quad 
W_e = \tfrac{l(r-1) }{d},
\eeq
with $d=2$.
{\em Here, the quantities $(W, S, l)$ are the same as $(w, \s, l)$ in \cite{MeRaRoSz2022a}.} The notation in~\eqref{eq:ODE_nota} is only used in this appendix, and it does not affect the rest of the paper.

\subsection{Repulsive properties}
We recall that the interior repulsive property~\eqref{eq:rep1} is proved in~\cite[Lemma 1.6]{MeRaRoSz2022a} for $\xi \in [0, \xi_s]$. By continuity, there exists $\xi_1>\xi_s$ such that the condition~\eqref{eq:rep1} holds for a domain $[0, \xi_1]$, thereby proving~\eqref{eq:rep1}. 

It thus remains to prove \eqref{eq:dec_U} and \eqref{eq:rep2}--\eqref{eq:rep3}.
From \cite[Section 2]{MeRaRoSz2022a}, we have that $(W, S)$ solve the $2\times2$ system of autonomous ODEs
\begin{subequations}
\label{eq:ODE}
\begin{equation}
  \tfrac{d W}{d x}  = - \tfrac{\D_1}{\D} , \quad  \tfrac{d S}{d x} = - \tfrac{\D_2}{\D},
  \end{equation}
 where
 \begin{align}
 \D & = (1-W)^2 - S^2, \\
  \D_1 & = W(W-1)(W-r) - d (W- W_e) S^2 =  (W- W_1(S))( W- W_2(S) )(W- W_3(S)),  \\
  \D_2 & = 
\tfrac{S}{l} \B( (l+d-1) W^2 - W(l+d + lr - r) + lr - l S^2  \B) = \tfrac{S(l+d-1)}{l} (W- W_2^{-}(S)(W-W_2^+(S)) ,
\end{align}
\end{subequations}
where $\{W_i(S)\}_{i=1}^3$ are the roots of the cubic polynomial $\D_1$, and $\{W_2^{\pm}(S)\}$ are the roots of the quadratic polynomial $\D_2$. Here, $(W_2^{\pm}, W_i)$ are the same as $(w_2^{\pm}, w_i)$ in \cite{MeRaRoSz2022a}. From \cite[Lemma 2.1]{MeRaRoSz2022a} for the roots $W_2^{\pm}$ of $\D_2$, and from \cite[Lemma 2.3]{MeRaRoSz2022a} for the roots of $W_i$ for $\D_1$, we have the following properties: 
\begin{subequations}
\label{eq:ODE_prop}
\begin{align}
& (W_2^{-})^{\pr}(S) < 0,  \quad W_2^-(S) \leq W_2^+(S) , \quad \forall S > S_2^{(0)}, \\
& W_1^{\pr}(S) < 0, \quad W_2^{\pr}(S) < 0, \quad  W_1(S) \leq 0 < W_e < W_2(S) \leq 1 < r \leq W_3(S) , \quad \forall S > 0.
\end{align}
\end{subequations}
From the proof of~\cite[Lemma 2.1]{MeRaRoSz2022a}, the value $S_2^{(0)} \geq 0$ is chosen such that if $S_2^{(0)} > 0$, then for any fixed $S \in (0, S_2^{(0)})$ the equation $\D_2(W, S) = 0$ does not have a real root $W$. From the definition of $\D_2$, this implies that 
\beq\label{eq:sign_Del2}
\D_2(W, S) > 0, \quad  \forall \ S \in (0, S_2^{(0)}) , \ W \in \R.
\eeq

To prove \eqref{eq:rep2} and \eqref{eq:rep22}, we rewrite $ \xi + \bar U(\xi) - \al \bar \S(\xi), 1 + \xi^{-1} \bar U(\xi)$ as  
\beq\label{eq:rep2_equiv}
 \xi + \bar U(\xi) - \al \bar \S(\xi) = \xi ( 1 - W - S), \quad 
 \xi + \bar U(\xi) = \xi (1 - W).
\eeq
Denote\footnote{We use $(Q_S, Q_W)$ to denote the coordinates of a point $Q$ in the system of $(S, W)$.} by $P_2 = (P_{2, S}, P_{2, W} )$ and $P_3 =(P_{3, S}, P_{3, W} )$ 
 the sonic points, where $\D = \D_1 = \D_2 = 0$, and by $P_5 = (P_{5, S}, P_{5, W} )$ other root of $\D_1 = \D_2 = 0$. In particular, we have $P_{2,S} + P_{2, W} = P_{3, S} + P_{3,  W} = 1$. 
 From \textit{Location} in \cite[Lemma 2.4]{MeRaRoSz2022a}, $P_2, P_3, P_5$ are on the curve of the middle root $(S, W_2(S))$, which along with \eqref{eq:ODE_prop} implies 
\beq\label{eq:mid_Pi}
  P_{i, W} = W_2( P_{i, S})  > 0 , \quad i = 2, 3, 5.
\eeq

 From \textit{Position of the middle root} in  \cite[Lemma 2.4]{MeRaRoSz2022a}, we have the properties 
\beq\label{eq:below_sonic}
 S + W_2(S) < 1,  \quad  P_{3, S} < S < P_{2, S} .
\eeq
The following properties are proved in \cite[Lemma 3.2]{MeRaRoSz2022a}.

\begin{lemma}\label{lem:S_in}
Assume \eqref{r-range}. Let $W(S)$ be the solution to \eqref{eq:ODE} and $W_2(S)$ be the middle root to $\D_1$ \eqref{eq:ODE_prop}, \eqref{eq:ODE}. Suppose that $W(S)$ intersects $W_2(S)$ at $S^*$ with $P_{5, S} < S^* < P_{2, S}$ and $S^* + W(S^*) < 1$.\footnote{
The additional assumption $S^* + W(S^*) < 1$ (absent in the statement of \cite[Lemma 3.2]{MeRaRoSz2022a}), that the point $P_* \teq (S^* , W(S^*))= ( S^*,  W_2(S^*))$ lies below the sonic line, can be inferred from the phase portrait in \cite[Figures 1, 2]{MeRaRoSz2022a}, where $P_*$ is the intersection between the pink and red curves.}
Then we have:
\begin{itemize}
\item[(a)] $W(S) \in C^{\infty}( 0, S^*]$ and
\[
  \quad {\lim}_{S\to 0} W(S ) = 0.
\]
Moreover, $S \to 0$ corresponds to $ x \to \infty$	and there exists $(W_{\infty}, S_{\infty})
\in \R \times \R_+^*$ such that for $x \to \infty$, we have 
\beq\label{eq:decays}
S(x) = S_{\infty} e^{-r x } (1 + \mathcal{O}( e^{-r x }) ), 
\quad W(x) = W_{\infty} e^{-r x}(1 + \mathcal{O}( e^{-r x }) ) .
\quad 
\eeq

\item[(b)] $W(S)\in C^{\infty}[S^*, P_{2, S} )$ and 
\beq\label{eq:ODE_left}
W_2(S) < W(S) < W_2^-(S) ,\quad \forall S^* < S < P_{2, S}, \quad {\lim}_{S \to ( P_{2, S})^-} W(S) =
P_{2, W}.
\eeq
\end{itemize}
\end{lemma}

Using the above properties and \eqref{eq:rep2_equiv}, we are ready to prove \eqref{eq:rep2}, which is equivalent to
\beq\label{eq:rep2_equiv2}
1 - W(S) - S > 0, \quad S < P_{2, S}.
\eeq
The proof essentially justifies the phase portrait for $S < S^*$ in \cite[Figure 3]{MeRaRoSz2022a}, which is omitted in \cite{MeRaRoSz2022a}. 

\begin{proof}[Proof of \eqref{eq:rep2_equiv2}]

From $P_{2, W} > 0$ \eqref{eq:mid_Pi} and $ P_{2, S} + P_{2, W} = 1 $,
we get 
\beq\label{eq:ODE_S2}
P_{2, S} < 1.
\eeq
Since $W_2^{-}(S)$ is only defined for $S > S_2^{(0)} \geq 0$, we define its extension on $(0, \infty)$ as
\beq\label{eq:W2_ext}
W_{2, e}^{-}(S) = W_{2}(S_2^{(0)}), \quad S < S_2^{(0)},
\quad W_{2, e}^{-}(S) = W_{2}^{-}(S),  \quad S \geq S_2^{(0)}.
\eeq

Note that $ S^* + W(S^*) = S^* + W_2(S^*) < 1$. If $ W(\td  S) + \td S = 1$ for some $\td S \in (S^*, P_{2, S} )$, since the smooth solution $(S, W(S))$ can only cross the sonic line $1 = W + S$ at $P_2$ or $P_3$, where $\D_2$ vanishes, we get $\D_2(\td S, W(\td S)) = 0$. However, since $\td S \in (S^*, P_{2, S} )$, from \eqref{eq:ODE_left}, \eqref{eq:ODE}, \eqref{eq:ODE_prop}, we get 
\[
W(\td S) < W_2^-(\td S) \leq W_2^+(\td S) ,  \quad \D_2(S, W_2(\td S)) > 0,
\]
which contradicts $\D_2(\td S, W(\td S)) = 0$. Thus, by continuity, we get 
\[
  W(S) + S < 1, \quad S \in [S^*, P_{2, S} ).
\]

Next, we consider \eqref{eq:rep2_equiv2} for $S < S^*$. Since $W(S^*) = W_2(S^*)$ and $W_2(S)$ is a root of $\D_1(S, W) = 0$, we get $\D_1(S^*, W(S^*)) =0$ and $\f{d W(S)}{d S}|_{S = S^*} = 0$. Moreover, using  $W_2^{\pr}(S^*)<0, (W_2^-)^{\pr}(S^*)<0$ \eqref{eq:ODE_prop} and $0< W_e \leq W(S^*) = W_2(S^*)\leq W_2^{-}(S^*)$ \eqref{eq:ODE_left}, we obtain 
\[
W_1(S) \leq 0 < W(S) \leq W(S^*) < W_2(S), W_2^{-}(S), 1- S, \quad S \in [S^* - \e, S^*),
\] 
for some small $\e>0$. 

Let $ 0\leq S_a  <  S^*-\e$ be the largest value, where $ W(S)$ intersects the curves $W_2(S), W_{2, e}^-(S), 1-S$, or $W_1(S)$. If such a point does not exist, we choose $S_a = 0$. Then for $S \in (S_a, S^*)$, we obtain 
\[
  W_1(S) < W(S) < W_2(S), W_{2, e}^{-}(S), 1- S,
\]
which along with \eqref{eq:ODE}, \eqref{eq:ODE_prop} implies 
\[
\D_2(S, W(S)) > 0, \quad \D_1(S, W(S)) > 0, \quad \tfrac{d W}{d S} = \tfrac{\D_1}{\D_2} > 0,
\quad S \in (S_a, S^*). 
\]
We have $\D_2>0$ for the following reasons. If $S < S_2^{(0)}$, we obtain $\D_2>0$ from \eqref{eq:sign_Del2}; if $ S \geq S_2^{(0)}$, from \eqref{eq:W2_ext}, we obtain $W(S) < W_{2,e}^{-}(S) = W_2^{-}(S)$, which along with \eqref{eq:ODE}, \eqref{eq:ODE_prop} implies $\D_2 > 0$.

The monotonicity of $W(S)$ implies 
$W(S) < W(S^*),S \in (S_a, S^*)$. For $S \in(S_a, S^*)$, we prove 
\[
 S + W(S) < W(S^*) + S^* < 1.
\]

If $S_a=0$, we prove \eqref{eq:rep2_equiv2}. If $S_a> 0$, since $W_2(S) , W_{2, e}^{-}(S), 
 1-S$ are decreasing on $(S_a, S^*)$, \eqref{eq:ODE_prop} and $W(S)$ is increasing on $(S_a, S^*)$, $W(S)$ cannot intersect these three curves at $S_a$. Therefore, $W(S)$ must intersect $W_1(S)$ at $S_a$:  $W(S_a) = W_1(S_a)$. Since $W_1^{\prime}(S) < 0$ \eqref{eq:ODE_prop}, $W_1(S)$ is a root of $\D_1(S, W)$, and $\D_2(S_a, W(S_a)) \neq 0$ (the double roots satisfies $P_{i, W} >0 \geq W(S_a)$ \eqref{eq:mid_Pi})
\[
  \tfrac{d W}{ d S} |_{S = S_a} = 0, \quad W(S_a) = W_1(S_a), 
\]
and $\D(W, S) >0$ for $W \leq 0, S \leq S^* < P_{2, S} < 1$ \eqref{eq:ODE_S2}, we get $W(S)-W_1(S) < 0$ for $S \in (S_a - \e_2, S_a)$ with some small $\e_2$. 
If $W(S_I) = W_1(S_I) \leq 0$ for some $S_I \in (0, S_a)$, since the roots $P_i$ of $\D_1 = 0, \D_2 = 0$ satisfies $P_{i, W} > 0$ \eqref{eq:mid_Pi}, we obtain $\D_2(S_I, W_1(S_I)) \neq 0$. Thus, using
\[
\tfrac{d (W-W_1)}{dS} \B|_{S = S_I }
= \tfrac{\D_1(S_I, W_1(S_I))}{\D_2(S_I, W_1(S_I))}  - W_1^{\pr}(S_I)
= 0  - W_1^{\pr}(S_I) > 0,
 \]
we get that $W_1(S)$ is a upper barrier of $W(S)$ for $S \in (0, S_a)$. From \eqref{eq:ODE_prop}
and \eqref{eq:below_sonic}, for $S \leq S_a < S^* < P_{2, S} < 1$, we get
$
 W(S) \leq W_1(S) \leq 0, \ W(S) + S \leq S < 1.
$
We complete the proof of \eqref{eq:rep2_equiv2}.
\end{proof}

The following property of $W(S)$ for $S > P_{2, S}$ is proved in item (2) \textit{Original variables} and item (3) \textit{Reaching $P_2$} in \cite[Lemma 3.1]{MeRaRoSz2022a}.

\begin{lemma}\label{lem:S_out}
Assume \eqref{r-range}. We have $W(S) \in C^{\infty}(P_{2, S}, \infty)$ and 
\[
 W_2^-(S) < W(S) < W_2(S), \quad {\lim}_{S \to (P_{2, S})^+ } W(S) = P_{2, W},
\quad {\lim}_{S \to \infty} W(S) = W_e.
\]

For $S_0> P_{2, S}$ large enough,  the curve $(S, W(S))$ with $S>S_0$ corresponds to the 
spherically symmetric profile $( \bar \S(|y|), \bar \UU(y))$ (see Theorem \ref{thm:merle:implosion}) defined on $|y| \in [0, y_0]$ for some $y_0>0$.
\end{lemma}

\paragraph{\bf{Proof of  \eqref{eq:rep22} and \eqref{eq:rep3} }}

Below, we prove \eqref{eq:rep22}, \eqref{eq:rep3}.

Since $W_e,   W_2(P_{2,S})=P_{2, W} = 1 - P_{2,S} < 1, W_2(S) \leq 1$ \eqref{eq:ODE_prop}, 
using Lemma \ref{lem:S_out}, we obtain 
\[
 W(S) < c < 1,  \qquad \forall S \geq P_{2, S}
\]
for some $c>0$. For $0 < S < P_{2, S}$, using \eqref{eq:rep2_equiv2} and continuity, we obtain 
\[
W(S) < c_2 < 1, \qquad \forall S < P_{2, S}.
\]
By combining the above two inequalities, we have thus proven that $W(S) \leq c_3 < 1$ for any $S \geq 0$. Using \eqref{eq:rep2_equiv}, the inequality \eqref{eq:rep22} now follows.

From \eqref{eq:ODE_nota}, we have $ W_e = \f{r-1}{2\al} $ for $d = 2$. Since $S(x) \to \infty$ corresponds to $\xi \to 0$, using Lemma \ref{lem:S_out} and $W(x) = - \f{\bar U}{\xi}$ \eqref{eq:ODE_nota}, we obtain 
\[
{\lim}_{\xi \to 0} \xi^{-1} \bar U(\xi) =
- \lim_{S \to \infty} W(S) = - W_e=- \tfrac{r-1}{2\al}. 
\]
The bound $-\f{r-1}{2\al} > - \f{r}{2}$ follows from \eqref{r-range}. We establish \eqref{eq:rep3}.

\subsection{Decay of the profile for large $\xi$}

In this section, we prove \eqref{eq:dec_U}. For $x > x_l$ with $x_l$ sufficiently large, using \eqref{eq:decays} in Lemma \ref{lem:S_in}, we obtain $|W| \ll 1, |S| \ll1, \D > \f{1}{2}$. We focus on this region. 
For $x > x_l$, we use induction on $k \geq 0$ to prove 
\beq\label{eq:decay_ind}
 |\pa_x^k W| \les_k \min(1, e^{- r x}), \quad  |\pa_x^k S| \les_k \min(1 ,e^{- r x}).
\eeq
The base case $k=0$ follows from \eqref{eq:decays}. Suppose that \eqref{eq:decay_ind} holds for $k \leq n-1, n \geq 1$. To estimate $\pa_x^{n-1} W$, we take $\pa_{x}^{n-1}$ on both sides of the ODEs \eqref{eq:ODE} for $W, S$.
Since $\D_1, \D_2, \D $ \eqref{eq:ODE} are polynomials of $W, S$, using Leibniz's rule, the inductive hypothesis, and $( \min(1, e^{- r x}) )^l \leq \min(1, e^{- r x}), l \geq 1$, we yield 
\[
 |\pa_x^i \D_1 | ,  |\pa_x^i \D_2 |,  |\pa_x^j \D | \les_n \min(1, e^{- r x}),
 \quad 0  \leq i \leq n-1, 1 \leq j \leq n - 1.
\]
We need $j \geq 1$ so that the constant $1$ in $\D$ vanishes in $\pa_x^j \D$ \eqref{eq:ODE}.
Since $\D \geq \f{1}{2}$, using Leibniz's rule again, we obtain 
\[
| \pa_x^n W | = \bigl|\pa_x^{n-1} \bigl( \tfrac{\D_1}{\D} \bigr) \bigr| \les_n \min(1, e^{- r x}),
\quad | \pa_x^n S | = \bigl| \pa_x^{n-1}  \bigl( \tfrac{\D_2}{\D} \bigr) \bigr| \les_n \min(1, e^{- r x}) .
\] 
Note that in the full expansion of $\pa_x^{n-1} (\D_i/\D), i=1,2$, the numerator only involves $\pa_x^j \D_1, j \geq 1$. We prove the estimates for the case of $k = n$ and establish \eqref{eq:decay_ind} by induction.

Using the relations $\pa_x W(x) = \xi \pa_{\xi} W( x(\xi))$,$W(x) = - \xi^{-1} \bar U(\xi)$, and $x = \log \xi$ (see~\eqref{eq:ODE_nota}), for $\xi > e^{x_l}$, we obtain 
\[
 \bigl|(\xi \pa_{\xi} )^{i} \bigl( \xi^{-1} \bar U(\xi) \bigr) \bigr| \les_i \la \xi \ra^{-r}, \quad  i\geq 0.
\]
Applying the above estimate and the Leibniz rule to $\bar U = (\bar U / \xi) \cdot \xi$, we obtain $ | (\xi \pa_{\xi} )^{i} \bar U | \les_i \la \xi \ra^{-r+1}$ for $\xi > e^{x_l}$. 
Using induction on $i$, we further obtain 
$
  |\pa_{\xi}^i \bar U| \les_i \la \xi \ra^{-r + 1 - i}
$
for $\xi > e^{x_l}$. Since $\bar U \in C^{\infty}$, 
for $\xi \leq e^{x_l}$, the bound $  |\pa_{\xi}^i \bar U| \les_i \la \xi \ra^{-r + 1 - i}$ is trivial. The decay estimate of $\bar \S$ is similar, thereby proving \eqref{eq:dec_U}.

% END %%%%%%%%%%%% APPENDIX B

\section{Functional inequalities}
The goal of this appendix is to gather a few functional analytic bounds that are used throughout the paper. First, we record a Leibniz rule for radially symmetric vectors/scalars.

\begin{lemma}[\!\cite{BuCLGS2022}]\label{lem:leib}
Let $f, g$ be radially symmetric scalar functions over $\R^d$ and let $\FF = F \ee_R = (F_1, .., F_d)$ and $\GG = G \ee_R = (G_1, .., G_d)$ be radially symmetric vector fields over $\R^d$. For integers $m\geq 1$ we have
\[
\bal
  | \D^m( \FF \cdot \na G_i )- \FF \cdot \na \D^m G_i - 2 m \pa_{\xi} F \; \D^m G_i |
 & \les_m \sum_{ 1 \leq j \leq 2 m  } | \na^{2m +1-j} \FF| \cdot |\na^{j} G_i | , \\
  |\D^m (f \na g ) - f \na \D^m g - 2 m \na f \D^m g  | 
 & \les_m \sum_{ 1 \leq j \leq 2 m  } | \na^{2m +1-j} f| \cdot |\na^{j} g| , \\
|\D^m( \FF \cdot \na g) - \FF \cdot \na \D^m g - 2 m \pa_{\xi} F \; \D^m g |
& \les_m  \sum_{ 1 \leq j \leq 2 m  } | \na^{2m +1-j} \FF| \cdot |\na^{j} g |
   \\
 | \D^m( f \div(\GG ) - f \div( \D^m \GG) - 2 m \na f \cdot \D^m \GG|
 & \les_m \sum_{ 1 \leq j \leq 2 m  } | \na^{2m +1- j} f| \cdot |\na^{j} \GG|,
\eal
\]
whenever  $f, g, \{F_i\}_{i=1}^d, \{G_i\}_{i=1}^d$ are sufficiently smooth.
\end{lemma}

Lemma~\ref{lem:leib}  was established for $d=3$ in the proof of~\cite[Lemma A.4]{BuCLGS2022}. The proof given in~\cite{BuCLGS2022} also works for any dimension $d \geq 2$, so we do not repeat these arguments here.

Next, we  focus on Gagliardo-Nirenberg-type interpolation bounds with {\em weights}. In all of the following lemmas, we do not assume that the functions are radially symmetric. 
\begin{lemma}\label{lem:interp_wg}
Let $\d_1 \in (0, 1]$ and $\d_2 \in \R$. For integers $n\geq 0$ and sufficiently smooth functions $f$ on $\R^d$, we denote
\beq\label{eq:interp_J}
 \b_n \teq 2 n \d_1 + \d_2, 
 \qquad  
 I_n \teq \int |\na^n f(y) |^2 \la y \ra^{\b_n} d y , 
 \qquad 
 J_n \teq I_n + I_0,
 \eeq
where as usual we let $\la y \ra = (1+ |y|^2)^{1/2}$. 
Then, for $n<m$ and for  $\e>0$, there exists a constant  $C_{\e,n,m} = C(\e,n,m,\d_1,\d_2,d)>0$ such that 
\begin{subequations}
\label{eq:interp_convex}
\begin{equation}
I_n \leq \e I_m + C_{\e,n,m} I_0. 
\label{eq:interp_convex:a}
\end{equation}  
Moreover, for $p < r < q$ we have the following interpolation inequality on $\R^d$:
\begin{equation}
J_r \leq C_{p,q,r} J_p^{\al} J_q^{1-\al}, 
\qquad \mbox{where} \qquad
\al = \tfrac{q-r}{q-p},
\label{eq:interp_convex:b}
\end{equation}
\end{subequations}
for some constant $C_{p,q,r} = C(p,q,r,\d_1,\d_2,d)>0$.
\end{lemma}

\begin{proof}
Throughout the proof, all implicit constants depend on the index $n$ of $I_n$ or $J_n$, and on the parameters $\d_1, \d_2$. 
Using integration by parts, we have 
\begin{align*}
 I_n 
 &= \sum_i \int  \pa_i \na^{n-1} f \cdot \pa_i \na^{n-1} f \la y \ra^{\b_n} \notag\\
 &= -\sum_i \int 
 \bigl( \Delta \na^{n-1} f \cdot \na^{n-1} f \la y \ra^{\b_n}
  + \b_n \pa_i \na^{n-1} f \cdot \na^{n-1} f y_i \la y \ra^{\b_n-2} \bigr)
\end{align*}
Using that $\d_1 \leq 1$ we have $\b_n - 1 \leq \frac 12 (\b_{n-1} + \b_n)$, while by definition we have $\b_n = \tfrac{1}{2}(\b_{n-1} + \b_{n+1})$; thus, by the Cauchy-Schwarz inequality we obtain
\beq
\label{eq:interp_convex2}
I_n \leq (I_{n+1} I_{n-1})^{1/2} + \b_n (I_n I_{n-1} )^{1/2}, \quad n \geq 1.
\eeq

To prove~\eqref{eq:interp_convex:a}, it is sufficient to show that for any $ n \geq 0$ and $\e>0$, there exists $C_{\e,n}>0$ such that 
\beq\label{eq:interp_convex3}
I_n \leq \e I_{n+1} + C_{\e,n} I_0.
\eeq
We prove~\eqref{eq:interp_convex3} by induction on $n$. The base case $n=0$ is trivial, with $C_{\e,0}=1$. 
For the induction step, let $n\geq 1$. We use~\eqref{eq:interp_convex2}, the inductive hypothesis for $n-1$, and the Cauchy-Schwartz inequality, to conclude 
\[
I_n \leq \e I_{n+1} + \tfrac{1}{2} I_n + C_{\e,n} I_{n-1}
\leq \e I_{n+1} + \tfrac{3}{4}  I_n  + C_{\e,n} I_0.
\]
The above estimate concludes the proof of the induction step for~\eqref{eq:interp_convex3}; thus, we have proven~\eqref{eq:interp_convex:a}.

Combining \eqref{eq:interp_convex2} and  \eqref{eq:interp_convex3} we obtain 
\[
I_n \leq  (I_{n+1} I_{n-1})^{1/2} + \tfrac{1}{2} I_n + C_n I_{n-1}
\leq  (I_{n+1} I_{n-1})^{1/2} + \tfrac{3}{4} I_n + C_n I_0. 
\]
It follows that for all $n\geq 1$, 
\[
J_n = I_n + I_0 \leq 4 (I_{n+1} I_{n-1})^{1/2} + C_n I_0 
\leq C_n J_{n+1}^{1/2} J_{n-1}^{1/2}. 
\]
Iterating the above estimate, we deduce~\eqref{eq:interp_convex:b}.
\end{proof}

\begin{lemma}\label{lem:norm_equiv}
Let $\d_1 \in (0, 1], \d_2 \in \R$, and define $\b_n = 2 n \d_1 + \d_2$. Let $\td \vp_{n}$ be a weight satisfying the pointwise properties $\td \vp_{n}(y) \asymp_n \la y \ra^{\b_{n}}$ and $|\na \td \vp_{n}(y)| \les_n \la y \ra^{\b_{n}-1}$. Then, for any $\e > 0$ and $n\geq 0$, there exists a constant $C_{\e,n} = C(\e,n,\d_1,\d_2,d)>0$ such that\footnote{Throughout the paper we denote by $|\nabla^k f|$ the Euclidean norm of the $k$-tensor $\nabla^k f$, namely,  $|\nabla^k f| = (\sum_{|\alpha|=k} |\partial^\alpha f|^2 )^{1/2}$.}
\[
 \int |\na^{2n} f|^2 \td \vp_{2n}
 \leq (1 + \e)\int  | \D^n f |^2  \td \vp_{2n} + 
 C_{\e, n} \int |f|^2 \la y \ra^{\b_0},
\]
for any function $f$ on $\R^d$ which is sufficiently smooth and has suitable decay at infinity.
\end{lemma}

\begin{proof}
We use the standard summation convention on repeated indices. Using integration by parts, we get 
\[
\bal
B_n &\teq \int  |\D^n f|^2 \vp_{2 n} 
=  \sum_{i_1, \ldots, i_n}\sum_{j_1, \ldots, j_n}  
\int\pa_{i_1}^2\pa_{i_2}^2 \ldots \pa_{i_n}^2 f
\; 
\pa_{j_1}^2 \pa_{j_2}^2 \ldots \pa_{j_n}^2 f 
\; \td \vp_{2 n} \\
& = -  \sum_{i_1, \ldots, i_n}  \sum_{j_1, \ldots, j_n}
\int \pa_{i_1}\pa_{i_2}^2 \ldots \pa_{i_n}^2 f 
\; 
\bigl( \pa_{i_1} \pa_{j_1}^2 \ldots \pa_{j_n}^2 f 
\; \td \vp_{2n} +  
\pa_{j_1}^2 \ldots \pa_{j_n}^2 f  
\;\pa_{i_1}  \td \vp_{2n} \bigr)
\\
& =  \sum_{i_1, \ldots, i_n}  \sum_{j_1, \ldots, j_n}
\int\pa_{i_1j_1}\pa_{i_2}^2 \ldots \pa_{i_n}^2 f 
\; 
\pa_{i_1j_1} \pa_{j_2}^2 \ldots \pa_{j_n}^2 f 
\; \td \vp_{2n} \\
&\qquad \qquad \qquad 
+ 
\int
\pa_{i_1}\pa_{i_2}^2 \ldots \pa_{i_n}^2 f   
\bigl(
\pa_{i_1} \pa_{j_1} \ldots \pa_{j_n}^2 f 
\pa_{j_1}  \td \vp_{2n}
- 
\pa_{j_1}^2 \ldots \pa_{j_n}^2 f 
\; \pa_{i_1}  \td \vp_{2n}
\bigr)
\teq I + II.
\eal
\]
The term $II$ in the above identity involves a derivative of the weight and we do not further perform integration by parts for it. For the first term, we iteratively integrate by parts with respect to $\pa_{i_l}^2$ and $\pa_{j_l}^2$, with $l\in\{2,\ldots, n\}$, to obtain 
\begin{equation*}
B_n \geq \int |\na^{2n} f|^2 \td \vp_{2n}
- C_n \int |\na^{2n-1} f | \; |\na^{2n} f| \; |\na \td \vp_{2n}|.
\end{equation*}
By assumption we have $\d_1\leq 1$, so that
\begin{equation*}
|\na \td \vp_{2n}| 
\les_n \la y \ra^{\b_{2n}-1}
\les 
\la y \ra^{\b_{2n-1}/2}
\la y \ra^{\b_{2n}/2}
\les | \td \vp_{2n-1}|^{1/2} | \td \vp_{2n}|^{1/2}.
\end{equation*} 
Using the Cauchy-Schwarz inequality, for any $\e > 0$ we obtain   
\[
B_n 
\geq (1-\e) \int |\na^{2n} f|^2 \vp_{2n} - C_{n, \e} \int |\na^{2n-1} f|^2 \vp_{2n-1}.
\]
In analogy to the proof of~\eqref{eq:interp_convex:a} in Lemma~\ref{lem:interp_wg} (we need to replace the weight $\la y \ra^{\beta_n}$ with $\td \vp_{n}$, which is permissible in light of the assumed properties of $\td \vp_n$), we may further show that for any $\e^\prime>0$, 
\[
\int |\na^{2n-1} f|^2 \td \vp_{2n-1}
\leq \e^\prime \int |\na^{2n} f|^2 \td \vp_{2n} + C_{\e^\prime,n} \int | f|^2 \td \vp_{0}.
\]
Since $\e, \e^\prime > 0$ are arbitrary and $\td \vp_0 \asymp \la y \ra^{\beta_0}$, rewriting the above two inequalities completes the proof.
\end{proof}

Lastly, we record the following product estimates for the functional spaces $\cX^m$ defined in~\eqref{norm:Xk}. It is convenient to state estimates for a general dimension $d$, not just for $d=2$. For this purpose, we recall from~\eqref{eq:vp:g:def} that $\vp_g = \la y \ra^{-\kp_1 - d}$, with $\kp_1 = 1/4 \in (0, 1)$. Moreover, we recall from Lemma~\ref{lem:wg} that the weights $\vp_m$ satisfy $\vp_m(y) \asymp_m \la y \ra^{m}$, and $|\na \vp_m(y)|\les_m \vp_{m-1}(y)$.

\begin{lemma}\label{lem:prod}
Let $m \geq d/2$.
For $f : \R^d \to \R$ which lies in $\cX^m$ and $0 \leq i \leq 2m - d $, we have
\beq\label{eq:Xm_Linf}
 |\na^i f (y)| \les_m  \| f \|_{\cX^m}  \la y \ra^{-i + \frac{\kp_1}{2}},
\eeq
pointwise for $y \in \R^d$.
Moreover, for any $i\leq 2m$, $j \leq 2n$, with $ i \leq 2m - d$ or $ j \leq 2n -d$, and for any $\b \geq \kp_1/2$, we have 
\beq\label{eq:Xm_prod}
 \| \na^i f \; \na^{j} g \; \la y \ra^{i+j -\b  } \vp_g^{1/2} \|_{L^2} 
\les_{m, n,\b} \| f \|_{\cX^m} \| g \|_{\cX^n}.
\eeq
\end{lemma}
\begin{proof}
Fix $1\leq i \leq 2m-d$. 
By a standard density argument we can assume that $\nabla^i f \in C_c^{\infty}(\R^d)$.
Consider the cone with vertex at $y$ extending towards infinity: $\Om(y) \teq \{ z \in \R^d \colon z_j \sgn(y_j) \geq |y_j|, \forall 1\leq j\leq d \}$. In particular, for any $z \in \Om(y)$ have $|z| \geq |y|$. By integrating on rays extending to infinity, we have
\[
 |\na^i f(y) | \les \int_{\Om(y)} | \pa_1 \pa_2 .. \pa_d \na^i f (z) | d z.
\]
Upon letting  $\d_2 =- (\kp_1 + d)$ and $\d_1 = 1$ (choices with are consistent with the weights in~\eqref{norm:Xk}), from Cauchy-Schwartz, Lemma~\ref{lem:interp_wg} with  $\beta_n = 2n + \d_2$ (applicable since $i+d \leq 2m$), and Lemma~\ref{lem:norm_equiv} with $\td \vp_n = \vp_n \vp_g^{1/2} \asymp \la y \ra^n \vp_g^{1/2}$,  we  deduce
\begin{align*}
 |\na^i f(y) | 
 &\les \| \la z \ra^{i + d + \frac{\d_2}{2}}  \na^{i+d} f(z) \|_{L^2(\Omega(y))}
 \B( \int_{|z| \geq |y|} \la z \ra^{ - 2 i - 2 d - \d_2 } d z \B)^{1/2} 
 \\
 &\les 
 \| \la z \ra^{\frac{1}{2}\beta_{i+d}}  \na^{i+d} f(z) \|_{L^2(\R^d)}
 \B( \int_{|y|}^\infty \la R \ra^{ - 2 i - 2 d + (\kp_1 + d)} R^{d-1} dR \B)^{1/2} 
 \\
 &\les 
 \left( \| \la z \ra^{\frac 12 \beta_{2m}}  \na^{2m} f(z) \|_{L^2(\R^d)}
 + \| \la z \ra^{\frac 12 \beta_{0}}f(z) \|_{L^2(\R^d)} \right)
 \B( \int_{|y|}^\infty \la R \ra^{ - 2 i + \kp_1 - 1} d \la R \ra \B)^{1/2} 
 \\
 &\les \|f\|_{\cX^m} \la y \ra^{-i + \frac{\kp_1}{2}}.
\end{align*}
For the last inequality we have crucially used that $2i-\kp_1 > 0$, which holds since $i\geq 1$.

For $i = 0$, the aforementioned inequality (namely, $2i-\kp_1 > 0$) does not hold, and so the integral appearing in the above estimates does not converge. Instead, using the fundamental theorem of calculus and the decay estimate for $|\na f|$ obtained above, we have
\[
 |f(y)| \leq |f(0)| + \int_0^{|y| }  |\na f| \B(  \f{ t y }{ |y|} \B)  d t 
 \les  |f(0)|  +  \| f \|_{\cX^m}  \int_0^{|y|} \la t \ra^{-1 + \frac{\kp_1}{2}} d t
 \les |f(0)|  + \| f \|_{\cX^m} \la y \ra^{\frac{\kp_1}{2}}.
\]
To conclude the proof of~\eqref{eq:Xm_Linf} for $i=0$, we note that $|f(0)| \les \|f\|_{\cX^m}$. To see this, let $\chi$ be a smooth bump function $0\leq \chi\leq 1$ with $\chi\equiv1$ for $|y|\leq 1/2$ and $\chi \equiv 0$ for $|y|\geq 1$. Then, we have $|f(0)| \leq \| \chi f\|_{L^\infty} \les \|\chi f\|_{H^{2m}} \les \|\chi f\|_{\cX^{m}} \les \|f\|_{\cX^m}$, where the last inequality holds by  Lemma~\ref{lem:interp_wg} and Lemma~\ref{lem:norm_equiv}.

In order to prove~\eqref{eq:Xm_prod}, assume without loss of generality that $i \leq 2m- d$. Using \eqref{eq:Xm_Linf}, and recalling that $\vp_g = \la y \ra^{-\kp_1 - d} $, we have 
\[
 \| \na^i f  \na^{j} g \la y \ra^{i+j -\b } \vp_g^{1/2} \|_{L^2} 
 \les_m \|\la y \ra^{i -\frac{\kappa_1}{2}}  \na^i f  \|_{L^{\infty}}
 \|\na^j g  \la y \ra^{j - \b + \frac{\kappa_1}{2} } \vp_g^{1/2} \|_{L^2} 
 \les_m \| f \|_{\cX^m}  \|\na^j g  \la y \ra^{j - \b - \frac{d}{2} } \|_{L^2}.
\]
Next, Lemma~\ref{lem:interp_wg} with $\d_1 = 1$ and $\d_2 = - (\kappa_1+d)$, to obtain
\[
\|\na^j g  \la y \ra^{j - \b - \frac{d}{2} } \|_{L^2}
\les
\|\na^j g  \la y \ra^{j + \frac{\d_2}{2} } \|_{L^2}
\les_m  \| g \|_{\cX^m},
\]
where we have used that $- \beta - d/2 \leq \d_2/2$, which holds since $\beta \geq \kp_1/2$.
\end{proof}

\section*{Acknowledgments}
%Jiajie
The work of J.C.~was supported in part by the NSF Grant DMS-2408098 and by the Simons Foundation.
%Giorgio
The work of G.C.~was in part supported by the Collaborative NSF grant DMS-2307681 by the Simons Foundation.
%Steve
The work of S.S.~was in part supported by the NSF grant DMS-2007606 and the Collaborative NSF grant DMS-2307680. 
%Vlad
The work of V.V.~was in part supported by the Collaborative NSF grant DMS-2307681 and a Simons
Investigator Award.

%
%\bibliographystyle{plain}
%\bibliography{Vorticity_Blowup.bib}

\end{document}